\documentclass[a4paper,12pt]{article}
\usepackage{amsmath}
\usepackage[vlined]{algorithm2e}
\usepackage{diagbox}
\usepackage{enumerate}
\usepackage{appendix}
\usepackage{longtable}
\usepackage{cases}

\title{Lowest-order equivalent nonstandard finite element methods for biharmonic plates}
\author{Carsten Carstensen\footnote{Department of Mathematics, 
Humboldt-Universit\"{a}t zu Berlin, 10099 Berlin, Germany.
Distinguished Visiting Professor, Department of Mathematics, Indian institute of 
Technology Bombay, Powai, Mumbai-400076, India.  cc@math.hu-berlin.de}        
\quad\text{and}\quad 
Neela Nataraj\footnote{Department of Mathematics, Indian Institute of Technology Bombay, Powai, Mumbai 400076, India. neela@math.iitb.ac.in}
}

\SetFuncSty{textsc}
        \SetKwFunction{frun}{Run}
        \SetKwHangingKw{arun}{\frun}
        \SetKwFunction{fset}{Set}
        \SetKwHangingKw{aset}{\fset}
        \SetKwFunction{fselect}{Select}
        \SetKwHangingKw{aselect}{\fselect}
        \SetKwFunction{fcompute}{Compute}
        \SetKwHangingKw{acompute}{\fcompute}
        \SetKwFunction{fsolve}{Solve}
        \SetKwHangingKw{asolve}{\fsolve}
        \SetKwFunction{festimate}{Estimate}
        \SetKwHangingKw{aestimate}{\festimate}
        \SetKwFunction{fmark}{Mark}
        \SetKwHangingKw{amark}{\fmark}
        \SetKwFunction{frefine}{Refine}
        \SetKwHangingKw{arefine}{\frefine}
        \SetKwFunction{compute}{Compute}
        \SetKwFunction{set}{Set}

{\algorithm}%
   {\endalgorithm}

\usepackage{mathtools}

\DeclareFontEncoding{LS1}{}{}
\DeclareFontSubstitution{LS1}{stix}{m}{n}
\DeclareSymbolFont{symbols2}{LS1}{stixfrak} {m} {n}
\DeclareMathSymbol{\operp}{\mathbin}{symbols2}{"A8}

\usepackage{amsmath,amsthm,amssymb,enumerate}
\usepackage[T1]{fontenc}
\usepackage{gensymb}
\usepackage[square,sort&compress,comma,numbers]{natbib}
\usepackage[sc]{mathpazo}
\usepackage{multirow} 
\usepackage{newtxtext,newtxmath}
\usepackage{subfig}
\usepackage{graphicx}
\usepackage{epstopdf}
\usepackage{cancel}
\usepackage[margin=3cm]{geometry}
\usepackage{tikz}
\usepackage{caption}
\usetikzlibrary{shapes,calc}
\usepackage{verbatim}
\usepackage{mathrsfs}
\usepackage{accents}
\usepackage[utf8]{inputenc}

\usetikzlibrary{decorations.pathmorphing}
\usetikzlibrary{decorations.pathreplacing}
\usetikzlibrary{positioning}
\usetikzlibrary{shapes}
\usetikzlibrary{arrows}
\usetikzlibrary{patterns}
\usetikzlibrary{fadings}
\usetikzlibrary{plotmarks}
\usetikzlibrary{calc}
\usetikzlibrary{intersections}
\tikzstyle{every picture}+=[font=\footnotesize]
\usepackage{paralist}
\usepackage{bbm}
\usepackage{latexsym}           
\usepackage{enumerate}
\usepackage{enumitem}

\setlist{noitemsep, topsep=0.8ex, partopsep=0pt
	, leftmargin=3em}
\setlist[1]{labelindent=\parindent}

\newlist{axioms}{enumerate}{1}
\setlist[axioms]{font=\bfseries}

\newlist{alphenum}{enumerate}{1}
\setlist[alphenum]{label=\textbf{(\alph*)}, leftmargin=4em}

\newlist{alphienum}{enumerate}{1}
\setlist[alphienum]{label=\textit{(\alph*)}}

\newlist{romanenum}{enumerate}{1}
\setlist[romanenum]{label=\textit{(\roman*)}}

\newlist{romaninenum}{enumerate*}{1}
\setlist[romaninenum]{label=\textit{(\roman*)}}
\usepackage[vlined]{algorithm2e}
\SetKwIF{If}{ElseIf}{Else}{if}{}{else if}{else}{endif}
\SetKwFor{For}{for}{}{endfor}
\usepackage[noabbrev, capitalise]{cleveref}
\usepackage{tabu}
\crefname{equation}{\unskip}{\unskip}
\creflabelformat{equation}{#2(#1)#3}
        \SetFuncSty{textsc}
        \SetKwFunction{frun}{Run}
        \SetKwHangingKw{arun}{\frun}
        \SetKwFunction{fset}{Set}
        \SetKwHangingKw{aset}{\fset}
        \SetKwFunction{fselect}{Select}
        \SetKwHangingKw{aselect}{\fselect}
        \SetKwFunction{fcompute}{Compute}
        \SetKwHangingKw{acompute}{\fcompute}
        \SetKwFunction{fsolve}{Solve}
        \SetKwHangingKw{asolve}{\fsolve}
        \SetKwFunction{festimate}{Estimate}
        \SetKwHangingKw{aestimate}{\festimate}
        \SetKwFunction{fmark}{Mark}
        \SetKwHangingKw{amark}{\fmark}
        \SetKwFunction{frefine}{Refine}
        \SetKwHangingKw{arefine}{\frefine}
        \SetKwFunction{compute}{Compute}
        \SetKwFunction{set}{Set}

\usepackage{tikz}
\usetikzlibrary{matrix}
\usepackage{tikz-cd}
\newtheorem{thm}{Theorem}[section]

\newtheorem{lem}[thm]{Lemma}

\theoremstyle{definition}
\newtheorem{defn}{Definition}[section]

\newtheorem{rem}{Remark}[section]
\newtheorem{example}{\bf Example}[section]

\numberwithin{equation}{section}

\newcommand{\cE}{\mathcal E}

\newcommand{\cT}{\mathcal T}

\newcommand{\pw}{\mathrm{pw}}
\newcommand{\nc}{\mathrm{M}}
\newcommand{\jc}{\mathrm{J}}
\newcommand{\osc}{\mathrm{osc}}
\newcommand{\qo}{\mathrm{qo}}
\newcommand{\dc}{\mathrm{dc}}

\newcommand{\hnorm}[1]{\lVert #1 \rVert_{h}}

\newcommand{\St}{P_2(\cT)}
\newcommand{\NC}{\text{pw}}
\newcommand{\C}{\mathrm{C}}
\newcommand{\w}{\mathrm{P}}

\newcommand{\ip}{{\mathrm{IP}}}

\newcommand{\half}{\frac{1}{2}}
\newcommand{\trinl}{\ensuremath{|\!|\!|}}
\newcommand{\trinr}{\ensuremath{|\!|\!|}}

\newcommand{\dx}{{\rm\,dx}}

\newcommand{\ds}{{\rm\,ds}}

\newcommand{\dg}{{\rm dG}}

\DeclareMathOperator{\E}{\mathcal{E}}
\newcommand{\M}{\mathrm{M}}

\newcommand{\T}{\mathcal{T}}
\newcommand{\J}{\mathcal{J}}

\newcommand{\dgnorm}[1]{\lVert #1 \rVert_{\mathrm{dG}}}
\newcommand{\ipnorm}[1]{\lVert #1 \rVert_{\mathrm{IP}}}
\newcommand{\norm}[1]{\left\Vert #1 \right\Vert}

\def\mean#1{\left<#1\right>_E}
\def\jump#1{\left[ #1\right]_E}

\newcommand{\mmid}{\operatorname{mid}}
\newcommand{\dist}{\operatorname{dist}}

\newcommand{\V}{\mathcal{V}}

   \usepackage{xparse}
   
   \newcounter{const}
\NewDocumentCommand{\constant}{o}
 {
  \IfValueTF{#1}%
  {C_{#1}}%
  {\refstepcounter{const}%
  C_{\theconst}}%
 }

\usepackage{tikz,pgfplots}
\usetikzlibrary{calc,angles,positioning,intersections,quotes,decorations.markings}
\usepackage{tkz-euclide}
\pgfplotsset{compat=1.5}
\begin{document}	
\maketitle
\begin{abstract}
The popular (piecewise) quadratic  schemes for the biharmonic equation based on triangles
are the nonconforming  Morley finite element, the 
discontinuous Galerkin,   the $C^0$ interior penalty, and the WOPSIP schemes. Those methods are modified in their right-hand side $F\in H^{-2}(\Omega)$ replaced by 
 $F\circ (JI_\M) $ and then are quasi-optimal in their respective discrete norms.
 The smoother  $JI_\M$ is defined for { a} piecewise 
 smooth input function by a (generalized) Morley interpolation $I_\M$ followed by 
 a companion operator $J$.  An abstract framework for the error analysis in the energy, 
 weaker and piecewise Sobolev norms for the schemes is outlined and applied to the biharmonic equation. Three errors are also equivalent in some particular 
 discrete norm from [Carstensen, Gallistl,  Nataraj:   Comparison results of nonstandard 
 $P_2$ finite element methods for the biharmonic problem, ESAIM Math. Model. Numer. Anal. (2015)]
 without data oscillations.
 This paper extends the work 
[Veeser, Zanotti:  Quasi-optimal nonconforming methods for symmetric elliptic problems, 
SIAM J. Numer. Anal. 56 (2018)]
to the  discontinuous Galerkin scheme and adds  error estimates in weaker and piecewise Sobolev norms.
\end{abstract}



\section{Introduction}\label{sectionintroduction}
The paper contributes to lower-order nonstandard finite element methods for a biharmonic plate problem
in a real Hilbert space  $(V, a)$.  Given $F\in L^2(\Omega)$,  \cite{CarstensenGallistlNataraj2015} compares the errors for  nonstandard finite element methods (FEM) of the clamped biharmonic plate problem
based on piecewise quadratic polynomials, namely the 
nonconforming Morley FEM \cite{Ciarlet}, 
the symmetric interior penalty discontinuous Galerkin FEM (dGFEM) \cite{FengKarakashian:2007:CahnHilliard}, 
and
the $C^0$ interior penalty method ($C^0$IP) \cite{BS05}, 
with respective solutions 
$u_\M$, $u_h$, and $u_\ip$; Table~\ref{table:summary} displays details of the respective schemes. For $F\in L^2(\Omega)$, dGFEM and hp-dGFEM for biharmonic and fourth-order problems,  are extensively studied in \cite{Engel02, FengKarakashian:2007:CahnHilliard,   MozoSuli03,MozoSuli07,MozoSuliBosing07, Georgoulis2009, Georgoulis2011}.

\medskip \noindent For a general right-hand side $F \in H^{-2}(\Omega)$, the standard right-hand side $F(v_h)$ remains undefined for nonstandard finite element methods. A postprocessing procedure in \cite{BS05} enables to introduce a new $C^0$IP method for right-hand sides in $H^{-2}(\Omega)$. 
 In  \cite{veeser_zanotti1,veeser_zanotti2, veeser_zanotti3}, the discrete  test functions are transformed into conforming functions ($J$ is called smoother in those works) before applying the load functional  and quasi-optimal energy norm estimates
\begin{align*}
& \|{u-u_\M}\|_{\NC}  \approx \min_{v_\M \in\M(\T)} \|u-v_\M\|_{\NC}, \;
\|{u-u_\ip} \|_{\rm IP}  \approx \min_{v_{\rm IP} \in {\rm IP}(\T)} \|u-v_{\rm IP}\|_{\rm IP}
\end{align*}
are derived for the Morley FEM and $C^0$IP method.  

\medskip 
{The papers \cite{veeser_zanotti1,veeser_zanotti2, veeser_zanotti3} discuss minimal conditions on a smoother for each problem, while this  paper presents one smoother $J I_\M$ for {\it all} schemes; the best-approximation for the dGFEM is a new result. The smoother  $J I_\M$ also allows a post-processing with a~priori error estimates in weaker and piecewise Sobolev norms.}

\medskip
\noindent Table~\ref{table:summary} summarizes the notation of spaces, bilinear forms,  and an operator for the four second-order methods for the biharmonic problem detailed in 
Subsection~\ref{subsection:ncfem} and  in Sections~\ref{sec:DGFEM}, \ref{sec:C0IP}, and \ref{sec:wopsip}. 

{\footnotesize{
 \begin{longtable}[c]{|p{8em}|p{5.5em}|p{7.5em}|p{7.5em}|p{5em}|} \hline
\diagbox{Notation}{Schemes}
&{Morley FEM}&{dGFEM}&{$C^0$IP}
&{WOPSIP}\\ \hline\hline
Section reference& Subsec~\ref{subsection:ncfem} & Sec~\ref{sec:DGFEM} & Sec~\ref{sec:C0IP} &Sec~\ref{sec:wopsip} 
\\ \hline 
$V_h$ & $\M(\T)$ & $P_2(\T)$ & $S^2_0(\T)$ & $P_2(\T)$ \\ \hline
 ${A}_h$ &  $a_\pw$ in \eqref{eq:a_pw}  & $ A_\dg$ in \eqref{eq:AhindG} & $A_\ip$ in \eqref{eq:AhinC0IP} & $A_{\w} $ in \eqref{eq:ap} \\ \hline
 $b_h$   & 0 & $-\Theta~\mathcal{J}~-~\mathcal{J}^*$~in~(6.2) 
 & 
 $-\Theta~\mathcal{J}~-~\mathcal{J}^*$~in~(6.2) 
  &  0\\ \hline
 $c_h$  & 0 & $c_\dg$ in \eqref{eq:ccdefineJsigma1sigma2} & $c_\ip$ in \eqref{eq:chinC0IP} & $c_\w$ in \eqref{eq:chinWOPSIP}\\ \hline
  $I_h: \M(\T) \rightarrow V_h$ & id &id & $I_\C $ in \eqref{eq:ic}& id\\\hline
\caption{\label{table:summary} Overview of notation for four discrete schemes with  discrete bilinear form  \eqref{eq:bilinear}}
\end{longtable}
}}

\subsection*{Contributions}
The main contributions of this paper are 

\begin{enumerate}
 \item[(a)] the design and analysis of a  {\it generalized Morley interpolation} operator $I_\M$ for piecewise smooth functions in $H^2(\T)$, 

\noindent
\item[(b)] the design of modified schemes for Morley FEM, dGFEM, $C^0$IP method, and a weakly over-penalized symmetric interior penalty (WOPSIP) method  for the biharmonic problem for data in $F \in H^{-2}(\Omega)$, 

\noindent \item[(c)] an abstract framework for the best-approximation property and weaker (piecewise) Sobolev norm estimates,

\noindent
\item[(d)]  a~priori error estimates in (piecewise) Sobolev norms  for the lowest-order nonstandard finite element methods for biharmonic plates,

\noindent
\item[(e)] an extension of the results of \cite{CarstensenGallistlNataraj2015} to an equivalence 
 \begin{align} \label{eq:eq}
 &\hnorm{u-u_\M}
\approx
 \hnorm{u-u_\ip}
\approx
 \hnorm{u-u_\dg}
\approx \norm{(1-\Pi_0) D^2 u}_{L^2(\Omega)}
\end{align}
{\it without data oscillations} for the modified schemes,
\noindent
\item[(f)] the proof
of the best approximation for the modified dGFEM  that extends \cite{veeser_zanotti3} and  \cite[Thm 4.3]{CarstensenGallistlNataraj2015}. 
\end{enumerate}

\begin{rem}[medius analysis]
 \noindent
 The quasi-optimality of nonconforming and discontinuous
Galerkin methods was established in the seminal paper \cite{Gudi10} for the original method up to data oscillations for $F \in L^2(\Omega)$. 
Arguments from a posteriori error analysis \cite{Verfurth13} give new insight in the  
consistency term from the Strang-Fix lemmas. The techniques in this paper 
circumvent any a posteriori error analysis and take advantage of  the extra benefits of the companion  operator $J$.
 \end{rem}
 
 \begin{rem}[smoother]
 The fundamental series of contributions \cite{veeser_zanotti1,veeser_zanotti2, veeser_zanotti3} on the quasi-optimality concerns best-approximation for a modified scheme with $F_h:= F\circ J$ for a smoother $J$. Elementary algebra indicates a key identity (of \eqref{a2} below) that is already mentioned in \cite[(5.15)]{BS05} and  makes the source term in the consistency disappear.
 \end{rem}
 
 \begin{rem}[extension to higher order]
A (general) Morley interpolation allows for a simultaneous  analysis of four lowest-order schemes, but appears to be restricted to piecewise quadratics at first glance. But the combination of $J I_\M$ as an averaging smoother  with higher-order bubble-smoothers shall enable applications to higher-order schemes as indicated in \cite{veeser_zanotti3}.
\end{rem} 

\begin{rem}[extension to 3D]
Although the plate problem is intrinsic two-dimensional, there are three-dimensional Morley finite elements with a recent companion operator \cite{CCP_new}
that guarantees the fundamental properties in 3D such that the abstract framework applies.
\end{rem}

\subsection*{Organization}
 The remaining parts of this paper are organised as follows.   {Section~\ref{prologue} provides  an abstract characterization of the best-approximation property that applies to  the various applications considered in this and future papers  \cite{NNCCGCRDS2021}}. Section~\ref{section:quasiorthogonality} presents  preliminaries, a nonconforming discretisation, introduces a {\it novel} generalized Morley interpolation operator for discontinuous functions, and states  a best-approximation \cite{NNCC2020} result for nonconforming discretisations  with data $F \in H^{-2}(\Omega)$. Section \ref{sec:novel_interpolation} proves a crucial equivalence result of two discrete norms for a piecewise $H^2(\T)$ function, and proves approximation properties for the generalized Morley interpolation operator. Section \ref{sec:abstract}  provides a  framework for dG methods and the proof of a best-approximation property  under a set of general assumptions. Section \ref{sec:lower_order} develops the abstract result for  a~priori error estimates in weaker and piecewise Sobolev norms. 
Sections~\ref{sec:DGFEM} and \ref{sec:C0IP} recall the dG and $C^0$IP schemes and verify the assumptions of Sections \ref{sec:abstract}  and \ref{sec:lower_order}  for the best-approximation result in the energy norm as well as  weaker and piecewise Sobolev norms without data oscillations. 
The paper concludes with the equivalence \eqref{eq:eq} of errors in Section~\ref{sec:comparisonandwopsip} and a proof of quasi-optimality up
 to penalty for the WOPSIP scheme in Section~\ref{sec:wopsip}. 
\subsection*{General Notation} Standard notation of  Lebesgue and Sobolev spaces, 
their norms, and $L^2$ scalar products  applies throughout the paper
such as the abbreviation $\|\bullet\|$ for $\|\bullet\|_{L^2(\Omega)}$. 
For real $s$, $H^s(\Omega)$ denotes the Sobolev space  associated with the  Sobolev-Slobodeckii semi-norm $|\bullet|_{H^{\boldmath{s}}(\Omega)}$ 
\cite{Grisvard}; $H^s(T):= H^s({\rm int}(T))$ abbreviates the Sobolev space with respect to the interior ${\rm int}(T)\ne \emptyset$ of a (compact) triangle $T$.  The closure of $D(\Omega)$ in $H^s(\Omega)$ is denoted  by $H^s_0(\Omega)$   and $H^{-s}(\Omega)$ is the dual of $H^s_0(\Omega)$. The triple norm  
$\trinl \bullet \trinr:=|\bullet|_{H^{2}(\Omega)}$ is the energy norm
 and  $\trinl \bullet \trinr_{\text{pw}}:=|\bullet|_{H^{2}(\cT)}:=\| D^2_\text{pw}\bullet\|$ 
 is its piecewise version with the piecewise Hessian $D_\text{pw}^2$. Given any function $v \in L^2(\omega)$, 
 define the integral mean $ \fint_\omega v \dx:= {1/ |\omega| }\int_\omega v \dx$; {where $|\omega|$ denotes the area of $\omega$}. 
The notation $A \lesssim B$ abbreviates $A \leq CB$ for some positive generic constant $C$, 
which depends only on $\Omega$ and the shape regulatity of $\T$;
 $A\approx B$ abbreviates $A\lesssim B \lesssim A$. 

{
\section{Prologue}\label{prologue}}
 This section characterizes the best-approximation property of a class of non-conforming finite element methods.
 The  biharmonic  problem is put in an abstract framework in real  Hilbert spaces $X$ and $Y$ 
 and  a bounded bilinear form $a:X \times Y \rightarrow {\mathbb R}$  satisfying 
 an inf-sup  condition. Given a right-hand side 
 $F \in Y^*$, the {\em exact problem} seeks $x\in X $ with 
 \[
 a(x,\bullet)=F\quad\text{in }Y.
 \]
 The {\em discrete problem} is put in an analog framework with 
 finite-dimensional real Hilbert spaces $X_h$ and $Y_h$ and 
 a  bilinear form $a_h:X_h \times Y_h \rightarrow {\mathbb R}$. 
 The discrete space $X_h$ (resp. $Y_h$) is {\it not} a subspace of $X$ (resp. $Y$) 
 in general, but $X$ and $X_h$ (resp.  $Y$ and $Y_h$) belong to one common  bigger
 vector space that gives rise to the sum $\widehat{X}=X+X_h$ 
 (resp. $\widehat{Y}=Y+Y_h$). It is not supposed that this is a direct sum,
 so the intersection $X\cap X_h$ (resp. $Y\cap Y_h$) may be non-trivial.
 We suppose that   $\widehat{X}$ and $\widehat{Y}$ are  real Hilbert 
 spaces with (complete) subspaces $X$,  $X_h$ and $Y$, $Y_h$.
 The linear and bounded map $Q\in L(Y_h;Y) $ links the 
 right-hand side $F\in Y^* $ of the exact problem  to the right-hand side
 $F_h:=F\circ Q\in Y_h^ *$ of the discrete problem.   The map $Q$ is called smoother in \cite{veeser_zanotti1, veeser_zanotti3, veeser_zanotti2} because it maps a (possibly) discontinuous function $y_h \in Y_h$ to a smooth function $Qy_h$ in applications. The resulting  {\em discrete problem }
 seeks $x_h\in X_h$ with 
\begin{align} \label{discrete:problem}
a_h(x_h,\bullet)= F(Q \bullet) \text{ in } Y_h.
\end{align}
We also suppose that the exact and discrete problems are well-posed and this means in particular that $\dim X_h=\dim Y_h<\infty $ and that the bounded bilinear forms 
$a$ and $a_h$  satisfy inf-sup conditions 
with positive constants $\alpha$ and $\alpha_h$
and are non-degenerate such that the associated linear and bounded operators 
$A\in L(X;Y^*)$ and  $A_h \in L(X_h; Y_h^*)$ are bijective;
the associated linear operators are defined by 
$Ax:=a(x,\bullet)\in Y^*$ for all $x\in X$ and  by 
$A_hx_h:=a_h(x_h,\bullet)\in Y_h^*$ for all $x_h\in X_h$.

\medskip \noindent

The general discussion in \cite{veeser_zanotti1,veeser_zanotti3,veeser_zanotti2},  leads to an optimal smoothing $Q= \Pi_{Y}|_{Y_h}$ for the orthogonal projection $\Pi_Y \in L(\widehat{Y})$ onto $Y$.
This is a global operation in general and hence infeasible for practical computations.  Notice carefully that all examples in  \cite{veeser_zanotti1, veeser_zanotti3, veeser_zanotti2} discuss $Q \in L(Y_h; Y)$ with $Qz = z$ for all $z \in Y_h \cap Y,$ abbreviated by
\begin{align} \label{q=id}
Q={\rm id} \text{ in }Y_h \cap Y.
\end{align}
This paper introduces a smoother $JI_\M$ for the discontinuous Galerkin schemes  that satisfies \eqref{q=id} and is quasi-optimal in the following sense with a constant $\Lambda_{\rm Q} \ge 0$ that is exclusively bounded in terms of the shape regularity of the underlying triangulations. 
\begin{defn}[quasi-optimal smoother] A linear bounded operator $Q \in L(Y_h; Y)$  is called a quasi-optimal smoother if there exists some $\Lambda_{\rm Q} \ge 0$ such that 
\begin{align}\label{quasioptimalsmoother}
\|y_h -Qy_h\|_{\widehat{Y}} \le \Lambda_{\rm Q} \|y_h-y\|_{\widehat{Y}} \text{  for all } y_h \in Y_h  \text{ and all } y \in Y.
\end{align}
\end{defn}
\noindent {The proofs of Lemma~\ref{lem:smoother} and \ref{lemma2.2}  below rely on compactness arguments whence in the appendix, the constants  $\Lambda_{\rm P}$  and   $\Lambda_{\rm Q}$ depend on the discrete space.  {\it The point is  that this paper designs a smoother in Section~\ref{sec:approximation_errors}   with  a constant  that does not depend on the mesh-size}. }
\begin{lem} \label{lem:smoother}The operator $Q \in L(Y_h; Y)$  is a quasi-optimal smoother if and only if \eqref{q=id} holds. 
\end{lem}

The (nonconforming) finite element method is characterized by the operator 
$M\in L(X;X_h)$ that maps any $x\in X$ to a right-hand side $F:=a(x,\bullet)\in Y^ *$
and then to the solution $Mx:=x_h=A_h^ {-1}(Q^* F)$ to \eqref{discrete:problem}, i.e.,
$
Mx:= A_h^ {-1}(a(x,Q\bullet))\in X_h 
$ 
for all $x\in X$,
or, in operator form,
\[
M:=A_h^ {-1}Q^* A\in L(X;X_h).
\]
In other words,  the subsequent diagram commutes. 

\centerline{
\resizebox{.6\textwidth}{!}{
\begin{tikzpicture}\label{commut}
  \node (A) {\tiny $X$};
 \node at (-0.15, -0.1)   (E) {};
  \node (B) [below=of A] {\tiny ${X_h}$};
\node (F) [left=of B]{};
 \node at (-0.15, -1.3)   (F) {};
  \node (C) [right=of A] {\tiny $Y^*$};
  \node (D) [right=of B] {\tiny $Y_h^*$};
 \draw[-stealth,dashed] (F) -- node[left] {\tiny $P$} (E);
  \draw[-stealth] (A)-- node[right] {\tiny $M$} (B);
  \draw[-stealth] (D)-- node [above] {\tiny ${A}_h^{-1}$} (B);
  \draw[-stealth] (A)-- node [above] {\tiny ${A}$} (C);
  \draw[-stealth] (C)-- node [right] {\tiny $Q^*$} (D);
\end{tikzpicture}}}

This diagram also depicts some (linear and bounded) operator $P:X_h\to X$ that will become a quasi-optimal smoother 
in the context of the best-approximation property of $M$ below. A synonym to the best-approximation property of $M$ is to say $M$ is quasi-opimal in the following sense. 
\medskip \noindent 
\begin{defn}[quasi-optimal] The above operator $M$  is said to be quasi-optimal if  
\[
{  \text{\bf (QO)}}   \quad \exists \; C_{\rm qo} >0  \; \; \forall x \in X \quad  \|x- Mx\|_{\widehat{X}} \le  C_{\rm qo} \min_{x_h \in X_h} \|x-x_h\|_{\widehat{X}}.
\]
\end{defn}


\medskip \noindent A first characterisation of {\bf (QO)} has been given in  \cite{veeser_zanotti1} in terms of 
\begin{align}\label{m=id}
M= {\rm id} \text{ in } X_h \cap X.
\end{align}

\begin{lem} \cite{veeser_zanotti1} \label{lemma2.2}Under the present notation, $\text{\bf (QO)}$ is equivalent to \eqref{m=id}.
\end{lem}

The above lemmas characterize $C_{\rm qo}$ and $\Lambda_{\rm Q}$ by a compactness argument and it remains to control $C_{\rm qo}$ and $\Lambda_{\rm Q}$ in terms of mesh-size independent bounds in applications.  This paper designs in Section~\ref{sec:approximation_errors} a smoother in the spirit of \cite{veeser_zanotti1, veeser_zanotti2,veeser_zanotti3} based on earlier work in the context of a posteriori error control \cite{CC2005,CCHu2007, CCDGJH14} and adaptive mesh-refinement 
\cite{cc_hella_18,Carstensen2019, aCCP,CCP_new}. The outcome is a quasi-optimal smoother $P \in L(X_h; X)$   with a constant $\Lambda_{\rm P}$  that depends only on the shape regularity of the underlying finite element mesh and 
\begin{align}\label{quasioptimal}
\|x_h-Px_h \|_{\widehat X}   \le  \Lambda_{\rm P}  \|x_h-x\|_{\widehat X} \text{ for all } x_h \in X_h \text{ and all } x \in X.   
\end{align}
{The proof of the following characterization of best-approximation shall be given in the appendix.}
\begin{thm}\label{quasioptimal}
Suppose  $P \in L(X_h; X)$ and $  \Lambda_{\rm P} $ satisfy \eqref{quasioptimal}. Then {\bf (QO)} is equivalent to the existence of $\Lambda_{\rm H}>0$ with 
\begin{align*}
{  \text{\bf (H)}}&\quad  {a}_h(x_h, y_h)  -  {a}(P x_h,Qy_h)   \le \Lambda_{\rm H} \|x_h-Px_h\|_{\widehat{X}} \|y_h\|_{Y_h} \text{ for all } x_h \in X_h \text{ and } y_h \in Y_h.  
\end{align*}
\end{thm}
\noindent {In particular,  if {\bf (H)} holds,  then {\bf (QO)} follows with  a constant $C_{\rm qo}$ that depends solely on $\alpha_h,$  $\Lambda_{\rm H},$  $\Lambda_{\rm P},$ $\|Q\|$, and $\|A\|$.}

\noindent The next theorem presents  a  key estimate that is crucial for goal-oriented error control and duality arguments for weaker norm estimates.  {The proof and the dependence of contants are presented in  the appendix.}The motivation for ${\widehat{\text{\bf (QO)}}}$ is exemplified in Theorem~\ref{weakapriori} below. 
\begin{thm}\label{thm:qo2}
Suppose $P$ and $Q$ are quasi-optimal smoothers with \eqref{quasioptimalsmoother}-\eqref{quasioptimal} and suppose {\bf (QO)}.  Then  the existence of $\widehat{C_{\text{\rm qo}}} >0$ with
\begin{align*}
\hspace{-0.2in}{\widehat{\text{\bf (QO)}}} \; \; &  
a(x- PMx, y )  \le \widehat{C_{\text{\rm qo}}} \|x-Mx\|_{\widehat{X}} \|y-y_h\|_{\widehat{Y}} \; \text{ for all }  x \in X, \; y \in Y,  \text{ and } y_h \in Y_h   \quad
\end{align*}
is equivalent to the existence of  $\widehat{\Lambda_{\rm H}}>0 $ with 
\begin{align*}
{ \widehat{ \text{\bf (H)}}} \; &  {a}_h(x_h', y_h)  -  {a}(P x_h',Qy_h)   \le \widehat{\Lambda_{\rm H}}  \|x_h'-Px_h'\|_{\widehat{X}} \|y_h-Qy_h\|_{\widehat{Y}} \text{ for all } x_h' \in X_h',
\text{ and }  y_h \in Y_h.
\end{align*}
\end{thm}
\noindent {In particular,  if ${ \widehat{ \text{\bf (H)}}}$ holds, $ {\widehat{\text{\bf (QO)}}}$ follows with a constant 
 $\widehat{C_{\text{\rm qo}}} $
that depends solely on $\|a\|$,  $\Lambda_2',$  $\Lambda_{\rm P},$ and $\Lambda_{\rm Q}$. }

The a priori error estimates in weaker Sololev norms (weaker than the energy norm) are a corollary of Theorem \ref{thm:qo2} and the elliptic regularity,  the latter is written in an abstract form by the assumption that $X_s$ and $Y_s$ are two Hilbert spaces with $X \subset X_s$ and $Y_s \subset Y$ such that 
\[\hspace{-2.6in} {\text{\bf (R)}} \; \; \exists C_{\rm reg} >0 \; \forall F \in X_s^* \; \; \|A^{-*}F\|_{Y_s} \le C_{\rm reg} \|F\|_{X_s^*}
\]
for the solution $y:=A^{-*} F \in Y_s \subset Y$ to $a(\bullet, y) = F \in X_s^* \subset X^*$. 
\begin{thm}[weak a priori] \label{weakapriori} Under the assumptions of Theorem \ref{thm:qo2},  $ {\widehat {\text{\bf (QO)}}}$ and {\bf(R)} imply 
\[
\|x-PMx\|_{X_s} \le \widehat{C_{\text{\rm qo}}} \|x-Mx\|_{\widehat X} \sup_{\stackrel{y \in Y_s}{\|y\|_{Y_s} \le C_{\rm reg} }}  
\inf_{y_h \in Y_h} \|y-y_h\|_{\widehat{Y}} \text{ \rm for all } x \in X. 
\]
\end{thm}
\begin{proof} Given $x-PMx \in X \subset X_s$, a corollary of the Hahn-Banach extension theorem leads to some $F \in X_s^* \subset X^*$ with norm $\|F\|_{X_s^*} \le 1$ in $X_s^*$ and $\|x-PMx\|_{X_s} =F(x-PMx)$. The dual solution $y \in Y$ to $F=a(\bullet, y) \in X^*$ 
satisfies {{\bf (R)}} and $ {\widehat {\text{\bf (QO)}}}$ leads to 
\[ \|x-PMx\|_{X_s} =a(x-PMx,y) \le \widehat{C_{\text{\rm qo}}}   \|x-Mx\|_{\widehat X}  \|y-y_h\|_{\widehat{Y}}
\]
for any $y_h \in Y_h$. This and $\|y\|_{Y_s} \le C_{\rm reg} \|F\|_{X_s^*} \le  C_{\rm reg} $ conclude the proof. 
\end{proof}
\begin{example}[standard] For the $m$-harmonic operator $A=(-1)^m \Delta^m$ and $X=H^m_0(\Omega)=Y$,  {\bf (R)} holds for 
$X_s=H_0^{m-s}(\Omega), \: Y_s= H^{m+s}(\Omega) $ and $1/2 \le s \le 1$, $m=1$ or 2. Typical first-order approximation properties of the discrete finite element spaces result in 
\[
 \displaystyle \sup_{\stackrel{y \in Y_s}{\|y\|_{Y_s} \le C_{\rm reg} }}  
\inf_{y_h \in Y_h} \|y-y_h\|_{\widehat{Y}} =O(h_{\rm max}^s) 
\]
in terms of the maximal mesh-size $h_{\rm max}$ of the underlying finite element mesh.  \qed
\end{example}

\begin{rem}[best-approximation constant]
The paper   \cite{veeser_zanotti1} gives a formula  for the best-approximation constant $C_{\rm qo}$ for some slightly simpler problem in one Hilbert space.  
\end{rem}

\begin{rem}[injective smoother]
Under the above notation  $Q\in L(Y_h;Y)$ is injective if and only if $M\in L(X;X_h)$ is surjective  \cite{veeser_zanotti1}. Then there exists a right-inverse $P\in L(X_h;X)$
to $M$ and {\bf (H)} holds  with $\Lambda_{\rm H}=0$ (this follows with the arguments of the proof of Theorem 2.5 for $P'$ that is in fact a quasi-optimal smoother owing to \eqref{eqn:pxh-xh}.) Consequently,  the discrete scheme is  equivalent to a conforming Petrov-Galerkin scheme.   
\end{rem}
\begin{rem}[non injective smoother]
In case    $Q\in L(Y_h;Y)$ is {\em not} injective, the discrete problem may reduced to the range $X_h':=\mathcal{R}(M)$  of $M$
and the orthogonal complement $Y_h'$ of the kernel of $Q$ in $Y_h$. However, the explicit computation of the reduced discrete spaces $X_h'$ and $Y_h'$ 
may be costly and hence this paper outlines a general analysis that allows non-injective quasi-optimal smoothers. 
\end{rem}

\begin{example}[smoother for Morley]
For the standard Morley interpolation operator 
$I_\M:H^2_0(\Omega) \rightarrow \M(\T)$ and a companion operator 
$J: \M(\T) \rightarrow H^2_0(\Omega)$
 (cf. Lemma~\ref{companion}  below for details) the smoother
$Q=JI_\M$ is injective because  $J$ is a right-inverse of $I_\M$. 
\end{example}

\begin{example}[smoother for dG]
This paper advertises a  smoother $Q:=JI_\M$ for a (generalized) Morley interpolation $I_\M:(P_2(\T) + H^2_0(\Omega)) \rightarrow \M(\T)$ (cf.  \eqref{defccMorleyinterpolation} 
below  for details) followed by a companion operator $J$ from the previous example
 for the dG FEM. Then   $\text{ dim } P_2(\T) = 6 |\T|$ is strictly larger than $\text{ dim } \M(\T)=|{\mathcal V}(\Omega)|+|{\mathcal E}(\Omega) |$; whence $Q$ cannot be injective.

The situation for the $C^0$ IP with the discrete space $S^2_0(\T)$ 
(of the same dimension as $\M(\T)$) 
is more involved and is discussed in more details in Section ~\ref{sec:C0IP} below.
\end{example}

\section[dG4plates]{Preliminaries}\label{section:quasiorthogonality}
\subsection{Continuous model problem} \label{subsection:modelproblem}
Suppose $u \in V:= H^2_0(\Omega)$ solves the biharmonic equation $ \Delta^2 u=F$ for a given right-hand side $F \in V^*\equiv H^{-2}(\Omega)$
in a  planar bounded Lipschitz domain $\Omega$ with polygonal boundary $\partial \Omega$. The   weak form of this equation reads 
\begin{align} \label{eq:weakabstract}
a(u,v)= F(v) \quad\text{for all }v\in V
\end{align}
with the scalar product $a(v,w):= \int_\Omega D^2 v : D^2 w \dx $ for all $v,w \in V$. 
It is well known 
that \eqref{eq:weakabstract} has a unique solution $u$ and elliptic regularity  \cite{agmon2010,BlumRannacher,gilbargtrudinger2001,necas67}  holds in the sense that 
$F \in  H^{-s}(\Omega)$  implies $u \in V\cap H^{4-s} (\Omega)$ for all $s$ with $2-\sigma_{\rm reg} \le s \le 2$ with the index of elliptic regularity $\sigma _{\rm reg}>0$.
 The lowest-order nonconforming finite element schemes suggest a linear convergence rate in the energy norm for a  solution $u\in H^t(\Omega)$ 
at most for all $t\ge 3$. Therefore    $\sigma :=\min\{ 1, \sigma _{\rm reg}\}$    is fixed throughout this paper and exclusively  depends on $\Omega$. 
 The regularity is frequently employed in the following formulation. 

\begin{example}[regularity]\label{rem:dual}
There exists a  constant  $0 < \sigma \le 1$  such that
{ $F\in H^{-s}(\Omega)$  with $2-\sigma \le s\le 2$
satisfies  $ u \in V \cap H^{4-s}(\Omega) $ and }
\begin{align}
 \|u \|_{ H^{4-s}(\Omega)}&\le C_{\rm reg}(s)  \| F \|_{H^{-s}(\Omega)} \label{eq:extrareg} 
\end{align}
\noindent for some  constant $  C_{\rm reg}(s) <\infty$, which depends on $\Omega$ and $s$. (The dependence on $s$ results from the equivalence of Sobolev norms that may 
depend on the index $s$ in general.)
\end{example}

It is true that  pure Dirichlet boundary conditions in the model example lead to $\sigma >1/2$ and then allow for a control of the traces $D^2u$ in the jump terms. 
This paper circumvents this argument and all the results hold for  $\sigma \ge 0$. The new discrete analysis is therefore much more flexible and allows for generalizations of
the model problem e.g. for mixed and boundary conditions of less smoothness.

\begin{center}
\begin{figure*}[t]
\begin{minipage}[h]{0.5\linewidth}
\begin{center}
\begin{tikzpicture}
\node[regular polygon, regular polygon sides=3, draw, minimum size=5cm]
(m) at (0,0) {};

\fill [black] (m.corner 1) circle (2pt);

\put(-3,-5){$T$}

\fill [black] (m.corner 2) circle (2pt);
\fill [black] (m.corner 3) circle (2pt);

\draw [-latex, thick] (m.side 1) -- ($(m.side 1)!0.5!90:(m.corner 1)$);
\draw [-latex, thick] (m.side 2) -- ($(m.side 2)!0.5!90:(m.corner 2)$);
\draw [-latex, thick] (m.side 3) -- ($(m.side 3)!0.5!90:(m.corner 3)$);
\end{tikzpicture}
\end{center}
\end{minipage}
\begin{minipage}[h]{0.5\linewidth}
\begin{center}
\begin{tikzpicture}

\node[regular polygon, regular polygon sides=3, draw, minimum size=5cm]
(m) at (0,0) {};
 
\fill [black] (m.corner 1) circle (2pt);

\draw [thick] (0,0) -- (0,2.5);
\draw (0,0) -- (-2.1,-1.2);
\draw (0,0) -- (2.1,-1.2);

\draw (0,2.45) circle (0.2cm);
\draw (-2.2,-1.3) circle (0.2cm);
\draw (2.2,-1.3) circle (0.2cm);

\draw (0,0) circle (2pt);

\put(1.2,0.6){${\rm mid}(T)$}


\fill [black] (m.corner 2) circle (2pt);
\fill [black] (m.corner 3) circle (2pt);

\draw [-latex, thick] (m.side 1) -- ($(m.side 1)!0.5!90:(m.corner 1)$);
\draw [-latex, thick] (m.side 2) -- ($(m.side 2)!0.5!90:(m.corner 2)$);
\draw [-latex, thick] (m.side 3) -- ($(m.side 3)!0.5!90:(m.corner 3)$);
\end{tikzpicture}
\end{center}
\end{minipage}
\caption{(a) Morley (left) and (b) HCT (right) finite element}\label{fig}
\end{figure*}
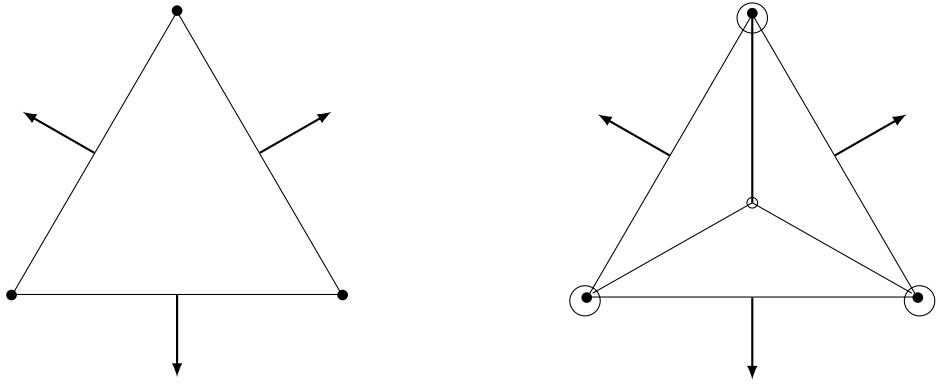
\end{center}

\subsection{Nonconforming discretisation}\label{subsection:ncfem}

{ Throughout the rest of this article,  the following notations are adopted.  Let $\T$ denote a shape regular triangulation of the polygonal Lipschitz domain into compact triangles.  Associate its piecewise constant mesh-size $h_\T \in P_0(\T)$ with $h_T:= h_\T|_T:= {\rm diam} (T) \approx |T|^{1/2}$ in any triangle $T \in \T$ of area $|T|$ and its maximal mesh-size $h_{\rm max}:= {\rm max} \; h_\T$. Let ${\mathcal V}$ (resp. ${\mathcal V}(\Omega)$ or ${\mathcal V}(\partial \Omega)$) denote the set of all (resp. interior or boundary) vertices in $\T$. Let ${\mathcal E}$ (resp. ${\mathcal E}(\Omega)$ or ${\mathcal E}(\partial \Omega)$) denote the set of all (resp. interior or boundary) edges. The length of an edge $E$ is denoted by $h_E$.  
Let  $\Pi_k$ denote the $L^2(\Omega)$ orthogonal projection onto 
the piecewise polynomials 
$P_k(\T):= \{ v \in L^2(\Omega): \forall T \in \T, \: v|_T \in P_k(T)\}$ 
of degree at most $k \in {\mathbb N}_0$. }
\medskip {
Let the  Hilbert space  $H^1(\T)\equiv \prod_{T\in \T} H^1(T)$.  Define the jump $[\varphi]_E:=\varphi|_{T_+}-\varphi|_{T_-}$ and the average $\langle\varphi\rangle_E:=\half\left(\varphi|_{T_+}+\varphi|_{T_-}\right)$ across the interior edge 
$E=\partial T_+\cap\partial T_-\in\E(\Omega)$ of $\varphi\in H^1(\cT)$ of the adjacent triangles  $T_+$ and $T_-\in\T$ in an order such that the unit normal vector  
$\nu_{T_+}|_E= \nu_E= - \nu_{T_-}|_E$ along the edge $E$
has a fixed orientation and points outside $T_+$ and inside $T_-$;
$\nu_T$ is the outward unit normal of $T$ along $\partial T$. The edge-patch $\omega(E):=\text{\rm int}(T_+\cup T_-)$ of the interior edge 
$E=\partial T_+\cap\partial T_-\in\E(\Omega)$  is the interior of union 
$T_+\cup T_-$ of the neighboring triangles $T_+$ and $T_-$.
Extend the definition of the jump and the average to an edge 
$E\in \E(\partial\Omega)$ on the boundary by $[\varphi]_E:=\varphi|_E$ and $\langle\varphi\rangle_E:=\varphi|_E$ 
 owing to the homogeneous boundary conditions.   Jump and average are understood componentwise for any vector function. 
 The edge-patch $\omega(E):=\text{\rm int}(T_+)$ of an  edge $E\in \E(\partial\Omega)$ on the boundary is simply the interior of the one triangle
 $T_+$  with the edge $E$ in the triangulation $\T$.}


\medskip  \noindent
The  nonconforming {\it Morley} finite element space \cite{Ciarlet}
reads
\begin{eqnarray*}
{\M}'(\T)&:=&\Bigg{\{} v_\M\in P_2(\T){{\Bigg |}}
\begin{aligned}
&\; v_\M \text{ is continuous at the vertices and its normal derivatives } \\
& \nu_E\cdot \nabla_{\NC}{ v_\M} \text{ are continuous at the midpoints of interior edges}
\end{aligned}\Bigg{\}}, 
\\
{\M}(\T)&:=&\Bigg{\{} v_\M\in \M'(\T){{\Bigg |}}
\begin{aligned}
&\; v_\M \text{  vanishes at the vertices of }\partial \Omega \text{ and its normal derivatives} \\
&\; \nu_E\cdot \nabla_{\NC}{ v_\M} \text{  vanish at the midpoints of boundary edges}
\end{aligned}\Bigg{\}}.
\end{eqnarray*}
Figure~\ref{fig}.a depicts the degrees of freedom 
of the Morley finite element 
\[
\left(T, P_2(T), (\delta_z:z\in \V(T))\cup( \fint_E \partial_{\nu_E}\bullet \ds:E\in\E(T))\right)
\]
(in the sense of Ciarlet)  in the triangle $T$ with set of vertices $\V(T)$ and set of edges $\E(T)$.

The semi-scalar product $a_\pw$ is defined by the piecewise differential operator   $D^2_\pw$ and
\begin{align}
a_\pw(v_\pw,w_\pw)&:= \sum_{T \in \T}  \int_T D^2 v_\pw : D^2 w_\pw \dx \quad\text{for all }v_\pw, w_\pw \in H^2(\T). \label{eq:a_pw}
\end{align}
It induces a piecewise $H^2$ seminorm 
$\trinl \bullet \trinr_\pw= a_\pw(\bullet,\bullet)^{1/2}$ that is also a norm in $\M(\T)$.  
Then $(\M(\T), a_\pw)$ is a (finite-dimensional) Hilbert space so that, given any $F_h \in \M(\T)^*$, there exists a unique discrete solution $u_\M \in \M(\T)$ to 
\begin{align} \label{eq:discrete}
a_\pw(u_\M,v_\M)= F_h(v_\M) \quad\text{for all }v_\M \in \M(\T).
\end{align}

\subsection{Interpolation of discontinuous functions}\label{sec:interpol}
\begin{lem}[interpolation  estimates I \cite{ccdg_2014, BrennerSungZhang13}] \label{interpolationerrorestimatesI}
The Morley interpolation operator $I_\M : V \rightarrow \M(\T)$ { is defined by $(I_\M v)(z)=v(z)$ and 
$\fint_E\frac{\partial I_\M v}{\partial \nu_E}\ds=\fint_E\frac{\partial v}{\partial \nu_E}\ds$  for any  $z\in {\mathcal V}(\Omega) $ and $ E\in {\cE(\Omega)}$. }It satisfies (a) the integral mean property of the Hessian, $D^2_{\text{\rm pw}} I_\M =\Pi_0 D^2$, \\(b) 
$ \displaystyle\sum_{m=0}^2   h_T^{m-2}   |  v - I_\M v |_{H^m(T)}  \le 2 \| (1- \Pi_0)D^2v \|_{L^2(T)}$
\quad for all {$v\in H^2(T)$}
 and any $T\in\T$,  \\
(c)  $\trinl v- I_\M v \trinr_{\pw} \lesssim h^{s}_{\max} \|v\|_{H^{2+s}(\Omega)}$\qquad for all $v \in  H^{2+s}(\Omega)$ and all $0 \le s \le 1$. \qed
\end{lem}
\noindent A reformulation of Lemma \ref{interpolationerrorestimatesI}.a is the best-approximation property 
\begin{align} \label{eq:best_approx}
a_{\pw}(v-I_\M v, w_2) & =0 \quad\text{for all } v \in V \text{ and all } w_2 \in P_2(\T).
\end{align}
A reformulation of Lemma \ref{interpolationerrorestimatesI}.b is the existence of a universal constant $\kappa >0$  with 
\begin{align} \label{eq:interpolation_estimate}
\|h_\T^{-2} (v-I_\M v)\| & \le \kappa \trinl v- I_\M v\trinr_\pw \quad\text{for all } v \in V.
\end{align}
(In fact $\kappa= 0.25745784465$  from \cite{ccdg_2014}  is independent of  the shape of the triangle $T$.)

{
\begin{rem}[Pythagoras]  The functions  $v_\pw,w_\pw \in H^2(\T):= \{ v_\pw\in L^2(\Omega): \forall T\in\T, \;   v_\pw|_T\in H^2(T)\}$  are orthogonal iff \(a_\pw( v_\pw,w_\pw)=0\) holds and then the Pythagoras theorem leads to 
\begin{align} \label{eq:Pythogoras}
 \trinl v-w_\M  \trinr_\pw^2 =  \trinl v -I_\M v \trinr_\pw^2+ \trinl w_\M- I_\M v \trinr_\pw^2
\end{align}
for all $v\in  V $ and $w_\M \in \M(\T)$. In particular, $\displaystyle \trinl v- I_\M v \trinr_\pw = \min_{v_2 \in P_2(T)} \trinl v-v_2 \trinr_\pw.$
\end{rem}}
\begin{defn}[(local) Morley interpolation] \label{deflocMorleyinterpolation}
Given any $T\in\T$ and $v_\pw\in H^2(T)$, the 
(local) Morley interpolation
$I_\M^{\rm loc} v_\pw|_T\in P_2(T)$ is defined by the degrees of freedom of the Morley finite element such that,
for all $z\in \V(T)$ and for all $E\in \E(T)$,
\[
(I_\M^{\rm loc} v_\pw- v_\pw)|_T(z)=0\quad\text{and}\quad
\fint_E (\partial (I_\M^{\rm loc} v_\pw- v_\pw)|_T/\partial \nu_E) \ds=0 .
\]
\end{defn}

The Morley interpolation allows for an  extension (still denoted by $I_\M$) to piecewise $H^2$ 
functions in  $H^2(\T)\equiv \prod_{T\in \T} H^2(T)$ by averaging the degrees of freedom.

\begin{defn}[Morley interpolation]\label{defccMorleyinterpolation}
Given any $v_{\pw} \in H^2(\T)$, define $I_\M v_{\pw} := v_\M \in \M(\T)$ by the degrees of freedom 
as follows. For any interior vertex    $z \in {\mathcal V}(\Omega) $ with set of attached triangles 
$\T(z) $ {that has} cardinality $|\T(z) |\in\mathbb{N}$
and any  interior edge 
$E = \partial T_+ \cap \partial T_-\in {\E(\Omega)}$ and its mean value operator $\mean{\bullet}$ (the arithmetic mean of the two traces
from the triangles $T_+$ and $T_- \in \T$ along their common edge $E=\partial T_+\cap \partial T_-$), set 
\[
v_\M(z):= 
 |\T(z)|^{-1}
\sum_{T \in \T(z)} (v_\pw|_T)(z)
\quad\text{and}\quad
 \fint_E\frac{\partial v_\M}{\partial \nu_E} \ds := \fint_E \mean{\frac{\partial v_{\pw}}{\partial \nu_E}} \ds.
\]
(The remaining degrees of freedom at vertices  and edges on the boundary are 
zero for homogeneous boundary conditions.)
\end{defn} 

\begin{rem}[ {standard Morley interpolation} 
vs Definition \ref{defccMorleyinterpolation}]
The interpolation operator $I_\M$ of Definition \ref{defccMorleyinterpolation} extends that of 
{ standard Morley interpolation operator}
in the sense that the two definitions coincide for functions in $H^2_0(\Omega)$.
This justifies the use of the same symbol $I_\M$.
\end{rem}

\subsection{Companion operator and  best-approximation  for the Morley FEM}\label{sec:companion}
A conforming finite-dimensional subspace of $H^2_0(\Omega)$ is provided by the Hsieh-Clough-Tocher $(HCT)$ FEM  \cite[Chap. 6]{Ciarlet}. 
For any $T\in\mathcal{T}$, let $\mathcal{K}(T):=\{T_E:\ E\in\mathcal{E}(T)\}$ denote the triangulation of $T$ into three sub-triangles $T_E:=\textup{conv}\{E,\textup{mid}(T)\}$ with edges $E\in\mathcal{E}(T)$ and common vertex $\textup{mid}(T)$ depicted in Figure~\ref{fig}.b. Then, 
\begin{align}
{HCT}(\mathcal{T})&:=\{v\in H^2_0(\Omega):\ v|_T\in P_3(\mathcal{K}(T))\text{ for all }T\in\mathcal{T}\}.\label{eq:HCT}
\end{align}
The degrees of freedom in a triangle $T\in\mathcal{T}$  are the nodal values    $\psi(z)$  and its derivative $\nabla \psi(z)$ 
  of the function  $\psi\in {HCT}(\mathcal{T})$ 
at any vertex  $z\in\V(T)$ and the values  $\partial\psi/\partial\nu_{E}(\textup{mid}(E))$ of the normal derivatives at the midpoint $\textup{mid}(E)$
of any edge $E\in\mathcal{E}(T)$.  

\begin{lem}[right-inverse \cite{DG_Morley_Eigen,aCCP,veeser_zanotti1}]\label{companion} 
There exists a linear map $J: {\rm M}(\mathcal{T})\to (HCT(\mathcal{T})+P_8(\mathcal{T})) \cap H^2_0(\Omega)$ 
and a constant   $\Lambda_\jc$ (that exclusively depends on the shape regularity of $\T$) 
such that any $v_{\rm M} \in {\rm M}(\mathcal{T})$ satisfies (a)--(e).
\begin{enumerate} 
\label{lem:MorleyCompanion}
\item[(a)]\label{MorleyCompanion:Nodal} $Jv_{\rm M}(z)=v_{\rm M}(z)$ for any  $z\in\mathcal{V}$;
\item[(b)]\label{MorleyCompanion:Derivative}  $\nabla ({J}v_{\rm M})(z)=
	|\mathcal{T}(z)|^{-1}\sum_{T\in\mathcal{T}(z)}(\nabla v_{\rm M}|_T)(z)
	\quad \text{ for }z\in\mathcal{V}(\Omega)$;
\item[(c)] \label{MorleyCompanion:Edges}$\fint_E \partial J v_{\rm M}/\partial\nu_E \,\textup{d}s=\fint_E \partial v_{\rm M}/\partial\nu_E \,\textup{d}s$ for any $E\in\mathcal{E}$;
\item[(d)] 
$ \trinl v_\M- J v_\M \trinr_\pw \le \Lambda_\jc  \min_{v\in V}  \trinl v_\M- v \trinr_\pw $; 
\item[(e)]  
$v_\M- J v_\M \perp P_2(\T)$ in $ L^2(\Omega)$. \qed
\end{enumerate}
\end{lem}

\noindent The operator $J$ of Lemma~\ref{companion} with $(a)$-$(c)$ 
is a right-inverse for  $I_\M : V \rightarrow \M(\T)$,   i.e.,
\begin{align} 
& I_\M J = \text{id in } \M(\T).   \label{eq:right_inverse}
\end{align}
Examples are provided in   \cite{DG_Morley_Eigen,aCCP,veeser_zanotti1}.  {For earlier references in the literature, see \cite{BrennerMR1620215}, \cite{CCEMHRHWLC2012}, \cite{CCMC2013}. }A right-inverse with benefits like (d)-(e) is called companion operator and \cite{aCCP} defines  $J: {\rm M}(\mathcal{T})\to V$ so that 
(a)-(e) of Lemma \ref{companion} hold (cf. in particular \cite[Lemma 5.1]{aCCP} for the analysis of (d)-(e)).

\medskip\noindent Given $F \in L^2(\Omega)$, we may choose $F_h \equiv F$ in the  discrete scheme \eqref{eq:discrete}; but otherwise $F_h= F \circ J$ is the option throughout this paper; other choices are proposed in \cite{BS05,veeser_zanotti3}.
Given any Lebesgue function $F \equiv F_h \in L^2(\Omega)$ with its $L^2$ projection $\Pi_2 F$ onto $P_2(\T)$, define its oscillations
$ 
\osc_2(F,\T):= \|h_\T^2 (F-\Pi_2F)\|
$. 

\begin{thm}[best-approximation up to data approximation \cite{Gudi10,CarstensenGallistlNataraj2015}] \label{thm:sec2.energynorm0}
The constant  $\constant [1]:=\max\{ \kappa \Lambda_\jc, 1+ \Lambda_\jc\}$, the solution  $u \in V$ to \eqref{eq:weakabstract} 
with $F\in L^2(\Omega)$, 
and the solution  $u_\M \in \M(\T)$ to  \eqref{eq:discrete} with  $F_h \equiv F$ satisfy 
$\constant{}^{-1}  \trinl u -  u_\M \trinr_\pw  \le 
\trinl u - I_\M u\trinr_\pw + \osc_2(F,\T)$. \qed
\end{thm}

The  discrete scheme \eqref{eq:discrete} requires a discrete right-hand side $F_h$  for a general $F \in H^{-2}(\Omega)$. 
The evaluation of $F_h:= F \circ J$ is feasible with $F_h(v_\M):= F(J v_\M)$ for all $v_\M \in \M(\T)$ and  
the {\it (modified) nonconforming scheme} seeks the solution $u_\M \in \M(\T)$ to
\begin{equation}\label{eq:modified_discrete}
a_\pw(u_\M, v_\M)= F(J v_\M) \quad\text{for all }v_\M \in \M(\T).
\end{equation}
Let $\Lambda_0$ denote the norm of $1-J$, where the right-inverse $J\in L(\M(\T); V)$ is regarded as a linear map between $\M(\T)$ and $V$,
\begin{equation}\label{eq:lambda0}
\Lambda_0:=\sup_{v_\M \in \M(\T) \setminus\{0\}} \trinl v_\M -J v_\M  \trinr_{\pw} /\trinl v_\M \trinr_{\pw} \le \Lambda_\jc. 
\end{equation}

\begin{thm}[best-approximation  \cite{NNCC2020,veeser_zanotti2}] \label{thm:energy_norm}
The solution  $u \in V$ to \eqref{eq:weakabstract}  with $F \in V^*$
and the solution  $u_\M \in \M(\T)$ to   \eqref{eq:modified_discrete} satisfy  
$\trinl u -  u_\M \trinr_\pw $$~\le ~\sqrt{1+\Lambda_0^2}\trinl u-I_\M u\trinr_\pw. $ {The constant $\sqrt{1+\Lambda_0^2}$ is optimal.} \qed 
\end{thm}



\begin{rem}[extra orthogonality in Lemma~\ref{lem:MorleyCompanion}.e.]
The $L^2$ orthogonality  in Lemma~\ref{lem:MorleyCompanion}.e  allows
control over dual norm estimates of the form 
$\| v_\M- J v_\M \|_{H^{-s}(\Omega)} \lesssim \| h_\T^s ( v_\M- J v_\M )\|$ for $0\le s \le 2$. This is critical in eigenvalue analysis or problems with low-order terms; for e.g. in \cite{aCCP, Carstensen2019}.
The  $L^2$ orthogonality  in Lemma~\ref{lem:MorleyCompanion}.e  also allows a direct proof of Theorem~\ref{thm:sec2.energynorm0}   that circumvents the a posteriori error analysis of the consistency 
term as part of the medius analysis \cite{Gudi10}.  Notice that the proof of the best-approximation of Theorem~\ref{thm:energy_norm} for the modified scheme does not require 
the $L^2$ orthogonality  in Lemma~\ref{lem:MorleyCompanion}.e.
\end{rem}

\begin{rem}[minimal assumptions on the smoother]
The series of papers \cite{veeser_zanotti1}-\cite{veeser_zanotti2}  addresses the question on the minimal assumptions on the smoother (partly as a right inverse only). This paper utilizes a smoother $J$ 
with the properties of Lemma~\ref{lem:MorleyCompanion}.a-d.  
\end{rem}

\medskip \noindent
The point in the subsequent example  is that the smoother $JI_\M$ 
may be more costly than averaging in other examples but it is at almost no extra costs for the case of point forces, which are of practical importance in civil engineering.  

{ \begin{example}[point forces]
Let  $m$ denote the point forces in the right-hand side,  i.e. , let
\begin{align}\label{eq:F}
F & = \sum_{j=1}^m \alpha_j \delta_{a_j},
\end{align}
the triangulation can be adopted such that the concentration point $a_j$ becomes a vertex in the triangulation. 
The right-inverse property of $I_\M$ displays that $(J I_\M v_\M)(z) = v_\M(z)$ holds at any vertex $z\in \V$  and for any Morley function $v_\M \in \M(\T)$. 
Hence the evaluation of the modified right-hand side $ F_h:= F\circ J$
leads for \eqref{eq:F} to $ F_h v_\M=\sum_{j=1}^m \alpha_j v_\M({a_j})$. The averaging of $I_\M$ in  Definition~\ref{defccMorleyinterpolation}  shows in the more general case $V_h\subset P_2(\Omega)$ that
$F_h:= F\circ J I_\M$ leads to $ F_h v_\M=\sum_{j=1}^m \alpha_j  / | \T(a_j)|\sum_{T\in \T(a_j) } (v_h|_T)({a_j})$. The same formula applies to other smoothers like the
enrichment in \cite{BS05} and \cite{veeser_zanotti1}-\cite{veeser_zanotti2}.
\end{example}}

\section{Interpolation of piecewise $H^2$ functions} \label{sec:novel_interpolation}

\subsection{Equivalent norms} \label{subsectionEquivalentnorms}

\medskip \noindent
 The Hilbert space  $H^2(\T)\equiv \prod_{T\in \T} H^2(T)$ is endowed with a norm
 $\|\bullet \|_h$ from \cite{CarstensenGallistlNataraj2015}  defined by 
\begin{align} \label{common:norm}
\| v_{\pw}\|^2_h:= \trinl v_\pw \trinr^2_{\pw} + j_h(v_{\pw})^2\quad\text{for all }
 v_{\pw}\in H^2(\T).
\end{align}
The homogeneous  boundary conditions in $H^2_0(\Omega)$ are included in the the jump contributions 
\begin{align} \label{eq:jh}
 j_h(v_{\pw})^2:= 
 \sum_{E \in \E} \sum_{z \in {\mathcal V} (E)} h_E^{-2} |[v_{\pw}]_E(z)|^2+
 \sum_{E \in \E}  \left|  \fint_E  \jump{{\partial v_{\pw}}/{\partial \nu_E}} \ds \right|^2 
\end{align}
by $[v_{\pw}]_E(z)=v_{\pw}|_{\omega(E)}(z)$ for $z\in\V(\partial\Omega)$ and  
$\jump{ \frac{\partial v_{\pw}}{\partial \nu_E} }  =\frac{\partial v_{\pw}}{\partial \nu_E }|_E$
for $E\in\E(\partial\Omega)$ at the boundary with jump partner zero owing to the homogeneous boundary conditions in 
\eqref{eq:weakabstract}.

The discontinuous Galerkin schemes of  \cite{Baker:1977:DG, FengKarakashian:2007:CahnHilliard} 
are associated with a another family  of norms  $\dgnorm{\bullet}$
depending on the two positive parameters $\sigma_1,\sigma_2>0 $ in  
the semi-norm scalar product 
\begin{equation}\label{eq:ccdefineJsigma1sigma2}
    c_\dg(v_{\pw},w_{\pw}) :=\sum_{E\in\E}
    \frac{\sigma_1}{h_E^3} \int_E \jump{v_\pw}\jump{w_\pw} \ds      
     +\frac{\sigma_2}{h_E} 
     \int_E \jump{\frac{\partial v_\pw}{\partial \nu_E}}\jump{\frac{\partial w_\pw}{\partial \nu_E}} \ds
\end{equation}
for all $v_{\pw},w_{\pw}\in H^2(\T)$.  The DG norm $\dgnorm{\bullet}$ is the square root of 
\begin{equation}\label{eq:ccdefinedGnorm}
 \dgnorm{v_\pw}^2 
 := \trinl v_{\pw} \trinr_{\pw}^2  +  c_\dg (v_\pw,v_\pw)
\end{equation}
for all $v_{\pw}\in H^2(\T)$. It depends on the parameters  $\sigma_1,\sigma_2>0 $ and so do all 
constants in the sequel; in particular those suppressed in the abbreviations $\lesssim$ and $\approx$. The conditions
on the ellipticity of the scheme in Lemma \ref{ellipticity} below  will assert that $\sigma_1$ and 
$\sigma_2$ are
sufficiently large. The analysis of this paper assumes this and simplifies the notation $\sigma_1\approx 1\approx \sigma_2$. 

One result in  \cite[Theorem 4.1]{CarstensenGallistlNataraj2015} shows that
$\hnorm{\bullet} \approx \dgnorm{\bullet}$ in  $H^2_0(\Omega) + P_2(\T)$; but the two norms 
are equivalent in  the  larger vector space  $H^2(\T)$.

\begin{thm}[$\|\bullet \|_h  \approx \dgnorm{\bullet}$]\label{lemmaonhnormisanorm}  
The  function  $\|\bullet \|_h$ from  \eqref{common:norm} and  $\dgnorm{\bullet}$ from 
\eqref{eq:ccdefinedGnorm} define norms in $H^2(\T)$ with
\[
\|v_\pw \|_h \approx \dgnorm{v_\pw}\lesssim
 \sum_{m=0}^2     | h_\T^{m-2}  v_{\rm pw} |_{H^m(\T)} 
\quad\text{for all }v_\pw\in H^2(\T).
\]
\end{thm}


\begin{rem}[$\{j_h=0\}\cap P_2(\T)= \M(\T)$] \label{remark3.1}
For any $v_2\in P_2(\T)$,  the condition $j_h(v_2)=0$ is equivalent to $v_2\in \M(\T)$. (This follows from the definitions of  $\M(\T)$ and $j_h$.)
\end{rem}

\begin{proof}[Proof of $\|\bullet \|_h\lesssim \dgnorm{\bullet}$] 
The (possibly discontinuous) piecewise affine interpolation 
$v_1\in P_1(\T)$ of $v_\pw\in H^2(\T)$ is defined by
nodal interpolation $v_1|_T(z)=v_\pw|_T(z)$ at the three vertices  $z\in\V(T)$ in each triangle $T\in\T$.
It is well known from standard finite element interpolation \cite{Braess,BS05,Ciarlet} that 
the error $w:= v_\pw-v_1\in H^2(T)$ satisfies 
\begin{equation}\label{eqccstandardfiniteelementinterpolation}
\sum_{m=0}^2 h_T^{m-2}  |w |_{H^m(T)}  \lesssim | v_\pw|_{H^2(T)}
\end{equation}
for each triangle $T\in\T$  with explicit constants \cite{CarstensenGedickeRim}
that exclusively depend on
the maximal angle in the triangulation. The nodal interpolation implies $[v_{\pw}]_E(z)=[v_1]_E(z)$ at each vertex  $z\in\V(E)$ of an edge $E\in\E$. Since 
${[v_{1}]_E}$ is an affine function along the edge $E\in\E$, an inverse estimate shows
\begin{equation}\label{eqccconstantsfromtheeigenvaluesofthe2times2mass}
h_E/6  \sum_{z \in {\mathcal V} (E)} |[v_1]_E(z)|^2\le \|  [v_{1}]_E\|_{L^2(E)}^2\le 
2 \|  [v_{\pw}]_E\|_{L^2(E)}^2+2 \|  [w]_E\|_{L^2(E)}^2 
\end{equation}
with a triangle inequality in the last step for $w= v_\pw-v_1$.
(The constant $h_E/6$ in the first inequality of
\eqref{eqccconstantsfromtheeigenvaluesofthe2times2mass}
stems from the eigenvalues $h_E/2$ and $h_E/6$ of the $2\times 2$ mass matrix of piecewise linear functions in 1D.)
This implies an estimate for the   first term of the definition of  $j_h(v_{\pw})$: 
$$ \sum_{E \in \E}  \sum_{z \in {\mathcal V} (E)} h_E^{-2} |  [v_{\pw}]_E(z)|^2 
\le 12 \sum_{E \in \E} h_E^{-3} ( \|  [v_{\pw}]_E\|_{L^2(E)}^2 + \|  [w]_E\|_{L^2(E)}^2).$$
A typical contribution  $(\fint_E  \jump{{\partial v_{\pw}}/{\partial \nu_E}} \ds)^2 $
for the second term (in the definition of $j_h$) is controlled  with a Cauchy inequality by 
$h_E^{-1}   \|  \jump{{\partial v_{\pw}}/{\partial \nu_E}} \|_{L^2(E)}^2 $.
This results in 
\[
 j_h(v_{\pw})^2\le \sum_{E \in \E}  h_E^{-1} 12 \left( 
  h_E^{-2} ( \|  [v_{\pw}]_E\|_{L^2(E)}^2 + \|  [w]_E\|_{L^2(E)}^2)
+   \|  \jump{{\partial v_{\pw}}/{\partial \nu_E}} \|_{L^2(E)}^2\right) .
\]
A triangle inequality  $\|  [w]_E\|_{L^2(E)} \le \|  w|_{T_+} \|_{L^2(E)}+ \|  w|_{T_-} \|_{L^2(E)}$ 
for an interior edge $E=\partial T_+\cap \partial T_-\in\E(\Omega)$ shared by the two triangles  $T_\pm\in\T$ 
plus trace inequalities on $T_\pm$ 
show 
\begin{equation}\label{eqccplustraceinequalitieson}
h_E^{1/2}  \|  [w]_E\|_{L^2(E)} \lesssim   \|  w \|_{L^2(\omega(E))}  + h_E  \|  \nabla_\pw w \|_{L^2(\omega(E))}
\lesssim h^2_E \|  D^2_\pw v_\pw \|_{L^2(\omega(E))} 
\end{equation}
with \eqref{eqccstandardfiniteelementinterpolation} in the end. 
The omission of $T_-$ in the above arguments for an edge $E\in\E(\partial\Omega)$ on the  boundary
provide \eqref{eqccplustraceinequalitieson} with  $\overline{\omega(E)}=T_+$. 
This and the finite overlap show
$ \sum_{E \in \E}  h_E^{-3} \|  [w]_E\|_{L^2(E)}^2 \lesssim   \trinl v_\pw  \trinr^2_{\pw}$.
In conclusion,
\[
 j_h(v_{\pw})^2\lesssim    \trinl v_\pw  \trinr^2_{\pw}+
 \sum_{E \in \E}   h_E^{-3} \|  [v_{\pw}]_E\|_{L^2(E)}^2 
+  \sum_{E \in \E}   h_E^{-1} \|  \jump{{\partial v_{\pw}}/{\partial \nu_E}} \|_{L^2(E)}^2.
\] 
The upper bound in the latter estimate is $ \dgnorm{v_\pw}^2 $ up to the weights  $\sigma_1\approx1\approx\sigma_2$.
\end{proof}

\begin{proof}[Proof of $ \dgnorm{\bullet}  \lesssim \|\bullet \|_h$]  Recall the 
piecewise affine interpolation  $v_1\in P_1(\T)$ of $v_\pw\in H^2(\T)$ and  $w:= v_\pw-v_1\in H^2(\T)$ 
with \eqref{eqccstandardfiniteelementinterpolation} from the previous part of the proof. 
Standard trace inequalities   as in  \eqref{eqccplustraceinequalitieson} for the first term 
(and an analog for the second term $h_E^{-1/2}\|  [\partial w/\partial\nu_E]_E \|_{L^2(E)} $) for $E\in \E$
provide
\begin{equation}\label{eqccplustraceinequalitieson2} 
h_E^{-3/2}\|  [w]_E \|_{L^2(E)}+ h_E^{-1/2}\|  [\partial w/\partial\nu_E]_E \|_{L^2(E)}
\lesssim \|  D^2_\pw v_\pw \|_{L^2(\omega(E))}. 
\end{equation}
This and  triangle inequalities result in 
\begin{align*}
& h_E^{-3/2}\|  [v_\pw]_E \|_{L^2(E)}+ h_E^{-1/2}\|  [\partial v_\pw/\partial\nu_E]_E \|_{L^2(E)} \\
&\quad \lesssim  
h_E^{-3/2}\|  [v_1]_E \|_{L^2(E)}+ h_E^{-1/2}\|  [\partial v_1/\partial\nu_E]_E \|_{L^2(E)} 
+  \|  D^2_\pw v_\pw \|_{L^2(\omega(E))} .
\end{align*}
The constant factor $1/2$ in the upper bound of the first subsequent  inequality (displayed as $2$ in the lower bound) 
stems from the eigenvalues $h_E/2$ and $h_E/6$ of the $2\times 2$ mass matrix of piecewise linear functions in 1D, 
\[
2 h_E^{-3}  \|  [v_{1}]_E\|_{L^2(E)}^2 \le  h_E^{-2} \sum_{z \in {\mathcal V} (E)} |[v_1]_E(z)|^2
=h_E^{-2}  \sum_{z \in {\mathcal V} (E)} |[v_\pw]_E(z)|^2
\]
with the nodal interpolation property $v_1|_T(z)= v_\pw|_T(z)$ for $z\in\V(T)$, $T\in\T$, in the last step.
The jump $[\partial v_1/\partial\nu_E]_E$ is constant along the edge $E$ and so 
\begin{align*}
& h_E^{-1/2}\|  [\partial v_1/\partial\nu_E]_E \|_{L^2(E)} = | \fint_E    [\partial v_1/\partial\nu_E]_E \ds | \\
&\quad \le      | \fint_E    [\partial v_\pw/\partial\nu_E]_E \ds |  +   | \fint_E    [\partial w/\partial\nu_E]_E \ds |  
\end{align*}
with a triangle inequality in the last step.   A Cauchy inequality 
$ \|  \nabla w_{\pw} \|_{L^1(E) }\le h_E^{1/2}  \|  \nabla w_{\pw} \|_{L^2(E) }$ and a trace inequality show
(as above in \eqref{eqccplustraceinequalitieson2})  that
\[
 | \fint_E    [\partial w/\partial\nu_E]_E \ds |  \le
h_E^{-1/2} \|   [\partial w/\partial\nu_E]_E\|_{L^2(E)}
\lesssim \|  D^2_\pw v_\pw \|_{L^2(\omega(E))}. 
\]
The combination of all aforementioned estimates reads
\begin{align*}
& h_E^{-3}\|  [v_\pw]_E \|^2_{L^2(E)}+ h_E^{-1}\|  [\partial v_\pw/\partial\nu_E]_E \|_{L^2(E)} ^2
  \lesssim   \|  D^2_\pw v_\pw \|_{L^2(\omega(E))}^2 \\ & \quad +
  h_E^{-2}  \sum_{z \in {\mathcal V} (E)} |[v_\pw]_E(z)|^2+
    | \fint_E    [\partial v_\pw/\partial\nu_E]_E \ds |^2 .
\end{align*}
The sum of all those estimates over $E\in\E$ plus  $  \trinl v_\pw  \trinr^2_{\pw}$ leads to 
an estimate with  the lower bound  $ \dgnorm{v_\pw}^2 $ up to the weights  $\sigma_1\approx1\approx\sigma_2$.
The finite overlap of the edge-patches  $(\omega(E):E\in\E) $ shows  that the resulting  upper bound
is $\lesssim \| v_\pw\|_h $.
\end{proof}

\begin{proof}[Proof of the upper bound]
The proof of the  asserted inequality starts with triangle inequalities for the jumps of
$v_{\pw}\in H^2(\T)$ and the shape regularity for $h_E\approx h_T$ for $E\in \E(T)$.
This and a Cauchy inequality  $ \|  \nabla v_{\pw} \|_{L^1(E) }\le h_E^{1/2}  \|  \nabla v_{\pw} \|_{L^2(E) }$ lead to 
\[
 j_h(v_{\pw})^2\lesssim 
 \sum_{T\in \T}      \left(    h_T^{-2}  \sum_{z \in {\mathcal V} (T)}    |(v_{\pw}|_T)(z)|^2
 +  h_T^{-1}  \sum_{E \in \E(T) }    \|  \nabla v_{\pw} \|_{L^2(E) }^2  \right) .
\]
 A one-dimensional trace inequality (with a factor $1$ that follows from 1D integration) 
\[
 |(v_{\pw}|_T)(z)|\le h_E^{-1/2}  \|  v_{\pw} \|_{L^2(E) }+ h_E^{1/2}  \|   \nabla v_{\pw} \|_{L^2(E) }
\]
along the edge $E\in\E(T)$  of the triangle $T\in\T$ with vertex $z\in\V(E)$ results in  
\[
 j_h(v_{\pw})^2\lesssim 
 \sum_{T\in \T}      \left(     
 h_T^{-3}    \|  v_{\pw} \|_{L^2(\partial T) }^2 
 +    h_T^{-1}    \|  \nabla v_{\pw} \|_{L^2(\partial T) }^2 \right) 
 \lesssim  \sum_{T\in \T}  \sum_{m=0}^2     | h_\T^{m-2}  v_{\rm pw} |_{H^m(T)} ^2
\]
with standard trace inequalities on $\partial T$ for  $v_{\pw}$ and $\nabla  v_{\pw}$ in the last step.
The right-hand side is  $\sum_{m=0}^2     | h_\T^{m-2}  v_{\rm pw} |_{H^m(\T)} ^2$ as asserted.
The remaining details are omitted for brevity.
\end{proof}

\subsection{Interpolation errors}
The interpolation error estimates are summarised in one theorem.

\begin{thm}[interpolation]\label{int_err}
Any $v_{\rm pw}\in H^2(\T)$ and its Morley interpolation $I_{\rm M} v_{\rm pw} \in \M(\T)$ from Definition \ref{defccMorleyinterpolation} satisfy \\
(a) 
\(\displaystyle 
\sum_{m=0}^2     | h_\T^{m-2}  (v_{\rm pw}- I_{\rm M} v_{\rm pw}) |_{H^m(\T)}
 \lesssim \| (1-\Pi_0)D^2_{\rm pw} v_{\rm pw}\|
+ j_h(v_{\rm pw})\le \| v_{\rm pw}\|_h\);  
\\
(b) \( \displaystyle \sum_{m=0}^2     | h_\T^{m-2}  (v_{\rm pw}- I_{\rm M} v_{\rm pw}) |_{H^m(\T)}
\hspace{-1mm}\approx \hspace{-3mm}\min_{w_\M\in \M(\T)}  \hspace{-2mm} \| v_{\rm pw}- w_\M \|_h  
\hspace{-1mm}\approx \hspace{-3mm}
\min_{w_\M\in \M(\T)}  \hspace{-1mm}  \sum_{m=0}^2    \hspace{-1mm}  | h_\T^{m-2}  (v_{\rm pw}- w_\M)  |_{H^m(\T)}    \).
\end{thm}

\begin{proof}[Proof of (a)]
{\it The first step } reduces the analysis to piecewise quadratic functions by the piecewise
Morley interpolation $I_\M^{\rm loc}$. Definition \ref{deflocMorleyinterpolation} shows
 $\int_E \nabla (I_\M^{\rm loc} v_\pw- v_\pw)|_T\, \ds=0$ for an edge $E\in\E(T)$ of the triangle $T$
and therefore $D^2 (I_\M^{\rm loc} v_\pw)|_T=\Pi_0  D^2  v_\pw |_T$ a.e. in $T \in \T$.
Notice that the piecewise defined Morley interpolation $v_2:= I_\M^{\rm loc}v_\pw \in P_2(\T)$
is discontinuous   (and shares none of the compatibility or boundary conditions)  in general. The 
 interpolation error estimates of Lemma~\ref{interpolationerrorestimatesI}.b read
\[
\sum_{m=0}^2 h_T^{m-2} |  v_2- v_\pw |_{H^m(T)} \le 2  \|  (1-\Pi_0) D^2 v_\pw\|_{L^2(T)}.
\]
This and a triangle inequality show that it remains to prove that  $v_\M:= I_\M v_\pw$ satisfies 
\begin{equation}\label{eqn:cc:morleyinterpolationerrorestimateproof1}
\sum_{m=0}^2 h_T^{m-2} |  v_2- v_\M |_{H^m(T)} \lesssim  j_h(v_\pw,T) 
\end{equation}
for the jump terms localised to a neighbourhood $\Omega(T)$ of $T \in \T$ as follows. The neighbourhood $\Omega(T)$ is the
interior of the union $\cup\{ K\in\T:\dist(T,K)=0 \}$ of $T \in \T$ plus one layer of triangles in $\T$ around. Then
\[
 j_h(v_\pw,T)^2:=  
\sum_{z \in \V(T) }   \sum_{F \in \E(z) }  h_F^{-2} |[v_{\pw}]_F(z)|^2+
 \sum_{E \in \E(T)}  \left| \fint_E  \jump{\frac{\partial v_{\pw}}{\partial \nu_E}} \ds \right|^2 
\]  
is the contribution from $T$ and its neighbourhood  $\Omega(T)$  to the full jump term 
$j_h(v_\pw)^2$ with the spider $\E(z):=\{  F\in\E : z\in \V(F)\}$ of edges with one end-point $z\in\V(T)$. 

{\it The second step } reduces the analysis to piecewise quadratic functions. The first obervation
is that the averaging of the degrees of freedom in the definition of $I_\M$ merely employs the data 
of  $v_2 =I_\M^{\rm loc} v_\pw$ in the sense that $v_\M=I_\M v_\pw= I_\M v_2$. This explains why 
$ j_h(v_\pw,T)=  j_h(v_2,T)$ in the asserted estimate 
\eqref{eqn:cc:morleyinterpolationerrorestimateproof1}. The second observation is that 
the left-hand side of \eqref{eqn:cc:morleyinterpolationerrorestimateproof1} involves the polyonomial
$ (v_2- v_\M)|_T\in P_2(T)$ that allows for inverse estimates
\[
\sum_{m=0}^2 h_T^{m-2} |  v_2- I_\M v_2 |_{H^m(T)} \lesssim   h_T^{-2} \|  v_2- I_\M v_2\|_{L^2(T)}.
\] 
The overall conclusion is that it suffices to prove, for all $v_2\in P_2(\T)$, that 
\begin{equation}\label{eqn:cc:morleyinterpolationerrorestimateproof2}
h_T^{-4} \|  v_2- I_\M v_2\|^2_{L^2(T)} \lesssim  j_h(v_2,T)^2. 
\end{equation}
In fact,  \eqref{eqn:cc:morleyinterpolationerrorestimateproof2} and the aforementioned 
arguments lead to a localised form of the assertion. The sum over all $T\in\T$ and the bounded
overlap of $(\Omega(T):T\in\T)$ then conclude the proof of the theorem.

{\it The third step } reduces the proof of  \eqref{eqn:cc:morleyinterpolationerrorestimateproof2}
to six coefficients. The six degrees of freedom 
on a triangle $T\in\T$ are the three point evaluations $\delta_z$ at the three vertices $z\in \V(T) $
and the three integral means of the normal derivatives 
$ \fint_E \partial_{\nu_E}\bullet \ds$ along the three edges $E\in\E(T)$. 
The six dual basis functions  $\psi_z$ for $z\in \V(T)$ and $\psi_E$ for $E\in\E(T)$    
in $P_2(T)$ are defined by the duality relations $\psi_E(z)=0=\fint_E \partial_{\nu_E}\psi_z \ds$
and  $\psi_z(z)=1=\fint_E \partial_{\nu_E}\psi_E \ds$
for all $z\in\V(T)$ and $E\in \E(T)$, while  $\psi_y(z)=0=\fint_E \partial_{\nu_E}\psi_F \ds$
for all vertices $z\ne y\in \V(T)$ and edges $E\ne F\in \E(T)$. Those functions are known and given
explicitly (e.g., in \cite{CCDGJH14} in the context of a short implementation of the Morley 
FEM in 30 lines of Matlab)  with a scaling (which is generally understood and follows from the explicit formulas)
\[
\| \psi_z \|_{L^2(T)}\approx |T|^{1/2}\approx h_T \quad\text{and}\quad \| \psi_E \|_{L^2(T)}
\approx h_T |T|^{1/2}\approx h_T^2
\]
for  all $z\in\V(T)$ and $E\in \E(T)$. On the other hand, given the dual basis of $P_2(T)$,
any function  $w_2:= v_2-  v_\M  \in P_2(T)$ for  $ v_\M:=I_\M v_2$
allows for a representation 
\[
w_2= v_2-  v_\M =\sum_{z\in \V(T)} w(z)\,  \psi_z  + \sum_{E\in \E(T)} w(E) \,  \psi_E \quad\text{in }T
\]
with the real coefficients  $w(z):= v_2|_T(z)-v_\M(z)$ 
and $w(E):= \fint_E \partial_{\nu_E} v_2|_T  \ds - \fint_E \partial_{\nu_E} v_\M  \ds$
for $z\in\V(T)$ and $E\in \E(T)$. Notice that the contributions of the piecewise quadratic 
$v_2$ are taken from $T\in\T$
and this is written explicitly by $v_2|_T$ in the coefficients, while the corresponding values 
of the Morley function $v_\M\in \M(\T)$ are  independent of $T$ as long as $z\in\V(T)$
or $E\in\E(T)$.  Given the coefficients $w(z)$ and $w(E)$, the triangle inequality in $L^2(T)$
and the scaling  of the dual basis functions lead to
\begin{equation}\label{eqn:cc:morleyinterpolationerrorestimateproof3}
\|  v_2- v_\M\|_{L^2(T)} \lesssim 
h_T \sum_{z\in \V(T)} |w(z)|     + h_T^2  \sum_{E\in \E(T)} |w(E)|.
\end{equation}
{\it The fourth step } analyses the coefficients  in 
 \eqref{eqn:cc:morleyinterpolationerrorestimateproof3}.
Let the triangles    $\T(z):=\{ T\in \T :z\in\V(T)\}  = \{T(1),\dots, T(J)\}$  at the vertex $z\in\V$ 
be enumerated such that  $T(j)$ and $T(j+1)$ share an edge 
$\partial T(j)\cap\partial T(j+1)=:E(j)\in \E(z)$ for  $j=1,\dots, J$. For an interior vertex $z\in\V(\Omega)$,
the patch is closed and then $T(1)$ and $T(J)$ share an edge $\partial T(1)\cap\partial T(J)=:E(J)\in \E(z)$
as well.   Define $x_j:=  ((v_2- v_\M)|_{T(j)})(z)$ for $j=1,\dots,J$ and observe for an interior vertex
$z\in\V(\Omega)$ that $\sum_{j=1}^J x_j=0$ (from the choice of $v_\M(z)$ as the arithmetic mean 
of the ${v_2}|_{T(j)}$) and  that 
\[
\sum_{E\in \E(z)} |\jump{v_2}(z)|= \sum_{j=1}^J  |x_{j+1}-x_j|
\]
with $x_{J+1}:=x_1$ (recall $z \in \V(\Omega)$ here). Since the arithmetic mean of the real numbers 
$x_1,\dots,x_J$ vanishes, zero belongs to their convex hull;
whence $\underline{m}:= \min_{j=1,\dots,J} x_j\le 0\le \max_{j=1,\dots,J} x_j=:\overline{m}$. A triangle inequality in this  
sequence $x_1,\dots,x_J$ shows that $\overline{m}-\underline{m} \le \sum_{j=1}^J  |x_{j+1}-x_j|$ (even with an omitted  factor $1/2$).
It follows $|x_1|,\dots, |x_J|\le  \sum_{j=1}^J  |x_{j+1}-x_j|$ and so, for a triangle $T\in\T(z)$ 
in the notation of \eqref{eqn:cc:morleyinterpolationerrorestimateproof3},
\begin{equation}\label{eqn:cc:morleyinterpolationerrorestimateproof4}
|w(z)|\le \sum_{E\in \E(z)} |\jump{v_2}(z)|\le J^{1/2} \sqrt{  \sum_{E\in \E(z)} |\jump{v_2}(z)|^2}
\end{equation}
follows (with a Cauchy inequality in $\mathbb{R}^J$ in the end).
This is suboptimal and the best constant in a 
squared version of this argument is contained in 
\cite[Appendix C]{aCCP}. 
Observe that $J\lesssim 1$ is bounded from above by the shape regularity of the triangulation $\T$.

In the remaining case of a vertex $z\in\V(\partial\Omega)$ on the boundary, $v_\M(z)=0$ and, 
in the above notation   $\T(z)=\{ T \in \T :z\in\V(T)\}  = \{T(1),\dots, T(J)\}$   and 
 $x_j =  ((v_2- v_\M)|_{T(j)})(z)= (v_2|_{T(j)})(z)$ for $j=1,\dots,J$. The homogeneous 
 boundary conditions enter in the jump terms for $ E(1):=T_1\cap\partial\Omega $ and 
 $E(J):=T_J\cap\partial\Omega $ and 
 \[
\sum_{E\in \E(z)} |\jump{v_2}(z)|=   |x_1|+ |x_J|+ \sum_{j=1}^{J-1}  |x_{j+1}-x_j| . 
\] 
Triangle inequalities show  $|x_1|,\dots, |x_J|\le  |x_1|+ |x_J|+ \sum_{j=1}^{J-1}  |x_{j+1}-x_j| $ 
(even with an omitted  factor $1/2$) and the above arguments lead to \eqref{eqn:cc:morleyinterpolationerrorestimateproof4} as well. 
(The optimal constant  for this argument may be found in 
\cite[Lemma 4.2]{CH17}.)
Recall the design of the Morley interpolation in Definition~\ref{defccMorleyinterpolation} with
the arithmetic mean 
$ \fint_E \partial_{\nu_E} v_\M  \ds=\fint_E \mean{\frac{\partial v_2}{\partial \nu_E}} \ds$
of the two normal traces for an interior edge $E=\partial T_+\cap \partial T_- \in\E(\Omega)$. This  leads 
to the edge contribution
\[
w(E)= \fint_E \partial_{\nu_E} v_2|_T  \ds - \fint_E  \mean{\frac{\partial v_2}{\partial \nu_E}}   \ds
= \pm \frac{1}{2} \fint_E  \jump{\partial_{\nu_E} v_2}  \ds
\]
in  \eqref{eqn:cc:morleyinterpolationerrorestimateproof3} with a sign $\pm$ for $T=T_\pm$.
The boundary conditions for a  boundary edge $E\in \E(\partial \Omega) $ and the jump convention
for $[\bullet]_E$ (recall that  $\nu_E$ points outwards for $E\subset\partial\Omega$) 
 directly show $w(E)=  \fint_E  \jump{\partial_{\nu_E} v_2}  \ds$.
It follows  
\begin{equation}\label{eqn:cc:morleyinterpolationerrorestimateproof5}
|w(E)|\le | \fint_E \jump {\partial_{\nu_E} v_2} \ds | \text{ for any }E\in\E(T).
\end{equation}
{\it The fitfh step } finishes the proof.  Recall that the coefficients $w(z)$ for $z\in\V(z)$  and 
$w(E)$ for any $E\in\E(T)$  in 
\eqref{eqn:cc:morleyinterpolationerrorestimateproof3}
satisfy 
\eqref{eqn:cc:morleyinterpolationerrorestimateproof4}-\eqref{eqn:cc:morleyinterpolationerrorestimateproof5}.
The resulting estimate reads 
\[ 
h_T^{-4} \|  v_2- v_\M\|_{L^2(T)}^2   \lesssim 
h_T^{-2}  \sum_{z\in \V(T)}   \sum_{E\in \E(z)} |[v_2]_E(z)|^2
+ \sum_{E\in \E(T)}  | \fint_E [{\partial_{\nu_E} v_2}]_E \ds |^2  \approx  j_h(v_2,T)^2
\]
with  the shape regularity $h_F\approx h_T$ for $F\in \E(z)$ and $z\in\V(T)$ in the end. 
This concludes the proof of  \eqref{eqn:cc:morleyinterpolationerrorestimateproof2} and 
thus that of (a)  as outlined at the end of the second step. 
\end{proof}

\begin{proof}[Proof of (b)] Given any $w_\M\in \M(\T)$, part (a) shows that the first term $T_1$ in 
the equivalence (b) is 
\(
T_1\lesssim  \trinl  v_{\rm pw} - w_\M  \trinr_{\text{pw}} + j_h(v_{\rm pw}- w_\M)
\le \| v_{\rm pw} - w_\M\|_h=: T_2
\)
with $j_h(v_{\rm pw})= j_h(v_{\rm pw}- w_\M)$ in the last step. 
Theorem~\ref{lemmaonhnormisanorm} applies to $v_{\rm pw} - w_\M \in H^2(\T)$
and proves $T_2\lesssim  \sum_{m=0}^2     | h_\T^{m-2}  (v_{\rm pw}-w_\M) |_{H^m(\T)} =:T_3$.
The estimates $T_1\lesssim T_2\lesssim T_3$ hold for all $w_\M\in \M(\T)$ and so for the respective 
minima  as well. Since $I_\M v_\pw \in \M(\T)$, the remaining estimate $\min_{w_\M\in \M(\T)} T_3
\le T_1$ is obvious. 
\end{proof}

\begin{rem}[$I_\M J  I_\M = I_\M$ { in $H^2(\T)$}]
Let $J$ be any  right-inverse of $I_\M$ in  the sense of \eqref{eq:right_inverse}. 
Since  $I_\M J $ is identity in $\M(\T)$, $I_\M J  I_\M v_\pw= I_\M v_\pw$ holds for any $v_\pw \in H^2(\T)$.
\end{rem}

\subsection{Approximation errors} \label{sec:approximation_errors}
The subsequent theorem discusses the approximation properties of $J \circ I_\M$ for piecewise smooth and 
piecewise quadratic functions. It is formulated in terms of $\|\bullet\|_h\approx \dgnorm{\bullet}$ and 
the norm equivalence implies an (undisplayed) analog  
for $\dgnorm{\bullet}$ as well.

\begin{thm}[approximation]\label{lemmaccapproximationproperties}
Any $v_\pw\in H^2(\T)$ and $v_2\in P_2(\T)$ satisfy (a)-(d).
\\
(a) \quad$\|  v_\pw - J I_\M v_\pw\|_h\lesssim \|   (1-\Pi_0)D^2_\pw v_\pw\|_{L^2(\Omega)}+
\min_{v\in H^2_0(\Omega)}  \|  v_\pw-v \|_h$;\\
(b)\quad  $\sum_{m=0}^2 |  h_\T^{m-2} (v_\pw - J I_\M v_\pw)|_{H^m(\T)}
\lesssim  \|   (1-\Pi_0)D^2_\pw v_\pw\|_{L^2(\Omega)}$ \\
\hspace*{79mm} $+
\min_{v\in H^2_0(\Omega)}  \sum_{m=0}^2 |  h_\T^{m-2} (v_\pw - v)|_{H^m(\T)} $; \\
(c)\quad 
$\|  v_2 - J I_\M v_2 \|_h\approx  \min_{v\in H^2_0(\Omega)}  \|  v_2-v \|_h $\\
 \hspace*{2.7cm} $ \approx 
 \sum_{m=0}^2 |  h_\T^{m-2} (v_2 - J I_\M v_2)|_{H^m(\T)} 
\approx  \min_{v\in H^2_0(\Omega)}  \sum_{m=0}^2 |  h_\T^{m-2} (v_2 - v)|_{H^m(\T)} $;
(d)\quad $ \|  v_2 - J I_\M v_2 \|_{ H^s(\T)} \lesssim  h_{\max}^{2-s} \min_{v\in H^2_0(\Omega)}  \|  v_2-v \|_h$ \, holds for any $0\le s \le 2$.
\end{thm}
{ \begin{rem} Theorem~\ref{lemmaccapproximationproperties}  implies that $P=Q= J \circ I_\M$ is a  quasi-optimal smoother with constant $\Lambda_{\rm P}=\Lambda_{\rm Q}$  that depends only on the
shape regularity of the triangulation. 
\end{rem}}
\begin{rem}[remainder in (a)-(b)]
The extra term  $\|   (1-\Pi_0)D^2_\pw v_\pw\|_{L^2(\Omega)}$ in the upper bound will vanish
for piecewise quadratic functions but cannot be omitted  in (a)-(b). For a proof of the latter statement by contradiction 
consider some $v_\pw \in H^2_0(\Omega)\setminus (HCT(\T)+P_8(\T) )$. Since  
$ J I_\M v_\pw\in HCT(\T)+P_8(\T)$, the left-hand side in (a)-(b) is positive, while 
$v=v_\pw  \in H^2_0(\Omega)$ leads to a right-hand side zero if the  term 
$\|   (1-\Pi_0)D^2_\pw v_\pw\|_{L^2(\Omega)}$ was neglected. 
\end{rem}

\begin{proof}[Proof of (a)] 
Theorem~\ref{lemmaonhnormisanorm} implies the first estimate 
(\ref{eqccproofofapart1}.a) below and 
Theorem~\ref{int_err}.a  asserts the second (\ref{eqccproofofapart1}.b) for the Morley interpolation 
$v_\M:= I_\M v_\pw$ of $v_\pw\in H^2(\T)$ 
in  
\begin{equation}\label{eqccproofofapart1}
\|  v_\pw- v_\M\|_h \lesssim  \sum_{m=0}^2     | h_\T^{m-2}  (v_{\rm pw} - v_\M)|_{H^m(\T)} 
\lesssim  \| (1-\Pi_0)D^2_{\rm pw} v_{\rm pw}\|_{L^2(\Omega)}+ j_h(v_{\rm pw}).
\end{equation}
Notice that  $j_h(v_\M-Jv_\M)=0$ implies $\|  v_\M- J v_\M\|_h= \trinl  v_\M - J v_\M  \trinr_{\text{pw}}
\lesssim  \trinl  v_{\rm \M} - v  \trinr_{\text{pw}}$ for any $v\in V$ from  Lemma~\ref{companion}.d   
 in the last step. This and  a triangle inequality,   \eqref{eqccproofofapart1}, 
 and  $ \trinl  v_{\text{pw}} - v  \trinr_{\text{pw}}+j_h(v_{\rm pw})\le  \sqrt{2}\|  v_{\pw} - v  \|_h $ show
\begin{align}
\|  v_\M- J v_\M  \|_h\lesssim   \trinl  v_{\pw} - v  \trinr_{\text{pw}}+ \trinl  v_{\rm pw} - v_\M  \trinr_{\text{pw}} 
 \label{eqccproofofapart1a}
\lesssim   \| (1-\Pi_0)D^2_{\rm pw} v_{\rm pw}\|_{L^2(\Omega)}+     \|  v_{\pw} - v  \|_h .
 \end{align}
This and a triangle inequality   $\|  v_\pw - J I_\M v_\pw\|_h\le \|  v_\pw - v_\M\|_h+\|  v_\M - J v_\M\|_h$ and 
 \eqref{eqccproofofapart1}-\eqref{eqccproofofapart1a} conclude the proof of (a). 
\end{proof}

\begin{proof}[Proof of (b)]  Adapt the notation of part (a) and recall that 
Theorem~\ref{int_err}.a provides (\ref{eqccproofofapart1}.b), the second
estimate in  \eqref{eqccproofofapart1}.  Since $v_\M-Jv_\M= I_\M Jv_\M-Jv_\M$ 
(from \eqref{eq:right_inverse}),  Lemma~\ref{interpolationerrorestimatesI} controls this interpolation error of 
$Jv_\M\in H^2_0(\Omega)$ and shows 
\[
\sum_{m=0}^2     | h_\T^{m-2}  (v_\M - J v_\M)|_{H^m(\T)}     \le 2   \trinl  v_\M - Jv_\M  \trinr_{\text{pw}}
\lesssim    \| (1-\Pi_0)D^2_{\rm pw} v_{\rm pw}\|_{L^2(\Omega)}+ \|  v_\pw - v  \|_h
\]
with    \eqref{eqccproofofapart1a}  in the last step.
Theorem~\ref{lemmaonhnormisanorm}  applies to $v_\pw - v\in H^2(\T)$. The combination  of the resulting estimate 
with the previous one concludes the proof of (b).
\end{proof}

\begin{proof}[Proof of (c)] The assertions (a)-(b) apply to $v_\pw:=v_2\in P_2(\T)$ and the extra term $\| (1-\Pi_0)D^2_{\rm pw} v_{\rm pw}\|_{L^2(\Omega)}$ vanishes. The resulting
estimates allow for obvious converse inequalities and so prove, for $v_2\in P_2(\T)$ 
and $v_\M:= I_\M v_2\in \M(\T)$,   that 
\begin{align*}
&\|  v_2 - J v_\M \|_h\approx 
\min_{v\in H^2_0(\Omega)}  \|  v_2-v \|_h \\  &\quad \lesssim 
 \sum_{m=0}^2 |  h_\T^{m-2} (v_2 - J v_\M)|_{H^m(\T)} 
\approx
\min_{v\in H^2_0(\Omega)}  \sum_{m=0}^2 |  h_\T^{m-2} (v_2 - v)|_{H^m(\T)} 
 \end{align*}
with  Theorem~\ref{lemmaonhnormisanorm} in between the two  equivalences. 
{ A triangle inequality, the estimate (\ref{eqccproofofapart1}.b), the estimate for $(1-J)v_\M$ in the proof of $(b)$ and \eqref{common:norm}
 applies to $v_\pw:=v_2\in P_2(\T)$  and shows 
 $ \sum_{m=0}^2 |  h_\T^{m-2} (v_2 - J v_\M)|_{H^m(\T)} \lesssim j_h(v_2) +  \|  v_2-v \|_h= j_h(v_2-v)+\|  v_2-v \|_h \le 2 \|  v_2-v \|_h $ for any  $v\in H^2_0(\Omega)$.}
This concludes the proof of (c). 
\end{proof}

\begin{proof}[Proof of (d)]  
The equivalence of the Sobolev-Slobodeckii norm and the norm by interpolation of Sobolev spaces \cite[Remark 9.1]{LM72}, for instance for a fixed 
reference triangle   $T=T_{\rm ref}$ with 
$\constant{}\label{ccint3}(s) =\constant[\ref{ccint3}](s,T_{\rm ref})$, provides for  ${w}:=(v_2-JI_\M v_2 )|_T\in H^2(T)$  the estimate 
\begin{equation}\label{eqininterpolationofSobolevnormosonTc}
 \|{w} \|_{H^s(T)}   \le \constant[\ref{ccint3}](s)\,   \|{w} \|_{ H^{1}(T)}^{2-s}  \|{w} \|_{H^{2}(T)}^{s-1} \quad\text{for }1< s< 2  .
 \end{equation}
 A straightforward transformation of Sobolev norms \cite[Theorem 3.1.2]{Ciarlet}  show
\eqref{eqininterpolationofSobolevnormosonTc} for any triangle  $T\in\T $ with 
$\constant[\ref{ccint3}]  (s)=\constant[\ref{ccint3}]  (s,T) 
= \kappa^{1+s} \constant[\ref{ccint3}](s,T_{\rm ref}) $ for the condition number $\kappa=\sigma_1/\sigma_2 $ of the affine transformation $a+Bx$
of $T_{\rm ref}$ to $T$ with the $2 \times 2 $ matrix $B$ and its positive singular values $\sigma_2 \le \sigma_1$. 
A more detailed analysis \cite{book_nncc} reveals that $\constant[\ref{ccint3}](s)$ exclusively depends on $s$ (but exploits
singularities as $s$ approaches the end-points $0$ and $1$). The estimate \eqref{eqininterpolationofSobolevnormosonTc} shows the first inequality in 
\[
\constant[\ref{ccint3}] (s)^{-2}\, 
\|{w} \|_{H^s(T)}^2   \le    \|{w} \|_{ H^{1}(T)}^{2(2-s)}  \|{w} \|_{H^{2}(T)}^{2(s-1)} 
\le   \|{w} \|_{ H^{1}(T)}^{2} +  \|{w} \|_{ H^{1}(T)}^{2(2-s)}  |{w} |_{H^{2}(T)}^{2(s-1)}
\]
with the subadditivity $(a+b)^p\le a^p + b^p $ for $a,b\ge 0$ and $0<p=s-1<1$ (e.g. from the concavity of $x\mapsto x^p$ for non-negative $x$) in the last step. 
An elementary estimate is followed by the  Young inequality $ab\le a^p/p+b^q/q$ for $p=(2-s)^{-1}$, $q=(s-1)^{-1}$, $a=  \| h_T^{-1} {w} \|_{ H^{1}(T)}^{2(2-s)}$, and
$b= |{w} |_{H^{2}(T)}^{2(s-1)}$ 
to prove
\[
 h_{\max} ^{ 2(s-2)} \|{w} \|_{ H^{1}(T)}^{2(2-s)}  |{w}|_{H^{2}(T)}^{2(s-1)}\le  \|h_T^{-1} {w} \|_{ H^{1}(T)}^{2(2-s)}|{w} |_{H^{2}(T)}^{2(s-1)}\le 
 \|h_T^{-1}  {w} \|_{ H^{1}(T)}^{2}  +     |{w} |_{H^{2}(T)}^{2} .
\]
This and the trivial estimate ${h_{\max} ^{ 2(s-1)} \le \text{\rm diam}(\Omega)^{2(s-1)}}$ leads to 
\[
 \|{w} \|_{H^s(T)}^2 \le \constant[\ref{ccint4}] h_{\max} ^{ 2(2-s)} \left( \|h_T^{-1}  {w} \|_{ H^{1}(T)}^{2}  +   |{w} |_{H^{2}(T)}^{2} \right)
\]
for $\constant{}\label{ccint4}=\constant[\ref{ccint3}] (s)^{2} (1+\text{\rm diam}(\Omega)^{2(s-1)})$. The sum over all those contributions over $T\in\T$ proves
\begin{align*}
\| v_2-JI_\M v_2 \|_{H^s(\T)}&\le  \constant[\ref{ccint4}]^{1/2}  h_{\max}^{ 2-s}  \left(  \| h_\T^{-1}(v_2-JI_\M v_2) \|_{H^1(\T)}+ 
\trinl  v_2-JI_\M v_2\trinr_\pw\right)\\
&\lesssim  h_{\max}^{2-s}  \min_{ v\in V }   \|v_2-v\|_h\
\end{align*}
with Theorem~\ref{lemmaccapproximationproperties}.c in the last step.  This concludes the proof of (d) for $1<s<2$. The assertion (d) is included in 
Theorem~\ref{lemmaccapproximationproperties}.c for $s=0,1,2$. The remaining case $0<s<1$ is similar to the above analysis with 
$\|{w} \|_{H^s(T)}   \le  \constant[\ref{ccint3}] (s) \|{w} \|_{ L^{2}(T)}^{1-s}  \|{w} \|_{H^{1}(T)}^{s} $  replacing \eqref{eqininterpolationofSobolevnormosonTc}
and analogous arguments; hence  further details are omitted.
\end{proof}

\section[dG4plates]{Abstract framework for best-approximation of lower-order methods}\label{sec:abstract}
\subsection{Discretisation}\label{sec:abstractSeconddiscretisation}
Suppose that $V_h\subset H^2(\T)$ is the finite-dimensional trial  and test space of an abstract (discontinuous Galerkin) scheme with a
bilinear form 
\[
A_h:(V_h+\M(\T))\times (V_h+\M(\T))\rightarrow {\mathbb R}
\]
 that is coercive and continuous  with respect to some  norm 
$\|\bullet \|_h$   in  $H^2(\T)$ in the sense that, \text{for all } $v_h,w_h \in V_h$,
\begin{equation}\label{eq:a_ellipticity} 
\alpha \| v_h \|_h^2 \le A_h(v_h,v_h) \text{ and } A_h(v_h, w_h) \le M \|v_h \|_h \|w_h\|_h  
\end{equation}
hold for some universal constants $ 0< \alpha, M<\infty $.
Suppose that $\|\bullet \|_h$ is a norm in  $H^2(\T)$ and equal to the norm $ \trinl \bullet \trinr_\pw:=  a_\pw (\bullet,\bullet)^{1/2}$ in $V+\M(\T)$ 
and stronger in general, i.e., 
\begin{equation}\label{eq:comparabilityofnorms} 
 \text{ (a) } \trinl \bullet \trinr_\pw \le \|\bullet \|_h \text{ in }   H^2(\T)
\quad\text{and}\quad
 \text{ (b) } \trinl \bullet \trinr_\pw=\|\bullet \|_h \text{ in } V+\M(\T).
\end{equation} 
Given a linear operator $JI_\M : V_h \to V $ with the companion operator   $J$  from Lemma~\ref{companion} and the (extended) linear interpolation operator $I_\M$ 
from  Subsection \ref{sec:interpol} 
the discrete problem reads: Given $F\in V^*=H^{-2}(\Omega)$  seek the   discrete solution $u_h \in V_h $ to
\begin{equation}\label{eq:discrete2} 
 A_h(u_h, v_h) = F(J I_\nc v_h) \quad\text{for all  } v_h \in V_h.
\end{equation}
The  Lax-Milgram lemma assures the existence of a unique discrete solution $u_h$ to \eqref{eq:discrete2}. 

\begin{rem}[$\|\bullet \|_h$] The examples of Sections~\ref{sec:DGFEM} and \ref{sec:C0IP} utilize $\|\bullet \|_h$ given in \eqref{common:norm}-\eqref{eq:jh}, but the abstract framework allows more general $A_h$ and $\|\bullet \|_h$ with \eqref{eq:a_ellipticity}-\eqref{eq:comparabilityofnorms} in \eqref{eq:discrete2}.
\end{rem}

{ \subsection{First glance at the analysis}\label{sec:motivation}}

This subsection motivates the abstract conditions and emphasises the relevance of the discrete consistency condition (dcc)
\begin{align} \label{a1}
a_\pw(I_\M u, e_h - I_\M e_h)+ b_h(I_\M u , e_h - I_\M e_h) & \le \Lambda_{\rm dc} \trinl u - I_\M u\trinr_{\pw} \|e_h\|_h
\end{align}
that leads to the best-approximation in terms of $\trinl u- I_\M u \trinr_\pw = \min_{v_2 \in P_2(T)} \trinl u-v_2 \trinr_{\pw}$ from \eqref{eq:Pythogoras}. The test function $e_h:= I_h I_\M u - u_h \in V_h \subset H^2(\T)$ is the discrete approximation of the error $u-u_h$ with $I_\M: H^2(\T) \rightarrow \M(\T)$ from Definition~\ref{defccMorleyinterpolation} and a transfer operator $I_h: \M(\T) \rightarrow V_h$ from Subsection~\ref{sec:transfer} below. For the dGFEM of Section \ref{sec:DGFEM} and the WOPSIP scheme of Section~\ref{sec:wopsip}, $I_h$ is the identity $1$ and otherwise it is controlled nicely (cf. \eqref{eq:mod_interpolation_estimate} below for details) \cite{CarstensenGallistlNataraj2015}. So we may neglect the difference $1-I_h$  for the sake of this first look at the analysis and suppose $I_h=1$. The key identity from the continuous problem \eqref{eq:weakabstract} and the discrete one \eqref{eq:discrete2} reads
\begin{equation} \label{a2}
a(u, JI_\M e_h) =F(J I_\M e_h) =A_h(u_h, e_h).
\end{equation}
\noindent The stability of the scheme $\alpha \|e_h\|^2_h \le A_h(e_h,e_h)\le M \|e_h\|_h^2$ motivates the investigation~of 
\begin{align} \label{a4}
A_h(e_h,e_h) & = a_\pw(e_h,e_h)+b_h(e_h,e_h) +c_h(e_h,e_h)
\end{align}
for the three bilinear forms that define the class of problems in \eqref{eq:bilinear} displayed in Table~\ref{table:summary}. The stability term $c_h(\bullet,\bullet)$ is controlled nicely in harmony with the discrete norm
$\|\bullet \|_h$, while $b_h(\bullet,\bullet)$ drives the method and completes the leading term 
$a_\pw(\bullet,\bullet).$

{
\noindent The definition of $e_h$ in \eqref{a4} leads to 
\begin{align} \label{a4a}
A_h(e_h,e_h) & = a_\pw(I_\M u,e_h)- A_h(u_h,e_h) + b_h(I_\M u,e_h) +c_h(I_\M u ,e_h)
\end{align}}
\medskip \noindent
Since $J$ is a the right-inverse of $I_\M$, \eqref{eq:best_approx} implies $ a_\pw(I_\M u,  I_\M e_h -J I_\M e_h)~=~0$.
This and elementary algebra show
\begin{align}
a_\pw(I_\M u, e_h) & = a_\pw(I_\M u, e_h -I_\M e_h) + a_\pw(I_\M u,  J I_\M e_h). \notag 
\end{align}
This in  combination with \eqref{a2} leads in \eqref{a4a} to
\begin{align} \label{a5}
A_h(e_h,e_h) & = a_\pw( I_\M u, e_h -I_\M e_h) +a_\pw(I_\M u -u,  J I_\M e_h)  \nonumber \\& \quad +b_h( I_\M u, e_h) +c_h( I_\M u, e_h).
\end{align}
The second term in the right-hand side of \eqref{a5} is equal to $a_\pw(I_\M u -u,  J I_\M e_h-I_\M e_h)$ and the stabilisation term is equal to $c_h(I_\M u - u, e_h)$. They are controlled by $\trinl u -I_\M u\trinr_\pw \|e_h\|_h$.  The bilinear form $b_h$ enjoys the miraculous property $b_h(I_\M u, I_\M e_h) =0$ for the discontinuous Galerkin schemes of this paper.  The remaining term on the right-hand side of \eqref{a5} is
$a_\pw(I_\M u, e_h -I_\M e_h) +b_h(I_\M u, e_h -I_\M e_h)$ and  in fact controlled by the dcc \eqref{a1}. The proof of dcc in Section~\ref{sec:C0IP} is one key argument in this paper.
{

\begin{rem} The arguments in this section applies to the case  where $A_h(\bullet,\bullet)$ satisfies an inf-sup condition; (and not the coercivity condition).  The key idea is to estimate the consistency error $F(JI_\M v_h) -A_h(I_\M u, v_h)$ using \eqref{a2}-\eqref{a4} and the orthogonality of the interpolation operator.
\end{rem}
}

\bigskip

\subsection{Transfer operators between $V_h$ and $\M(\T)$} \label{sec:transfer}
Recall $I_\M: H^2(\T) \rightarrow \M(\T)$ from Definition \ref{defccMorleyinterpolation}, 
 and suppose the existence of some constant $\Lambda_\M \ge 0 $ with
\begin{align} \label{eq:Inc_bound}
\|v_h- I_\M v_h \|_h & \le \Lambda_\M \|v_h - v \|_h\qquad\text{for all }v_h \in V_h \text{ and all }v \in V.
\end{align}
Suppose the existence of constants $\Lambda_\nc',M_\nc \ge 0$, and  boundedness  in the sense that 
\begin{align}
 \|v_h- I_\nc v_h \|_h & \le \Lambda_\nc' \|v_h  \|_h\qquad\text{for all }v_h \in V_h ,\notag \\
\label{eq:continuityconstantsccMnc}
 \trinl I_\nc v_h   \trinr_\pw & \le M_\nc  \| v_h \|_h\; \quad\text{ for all }v_h \in V_h .
\end{align}
Apparently $\Lambda_\nc'\le \Lambda_\M$ (with $v=0$) and  $M_\nc\le 1+ \Lambda_\nc'$ (with (\ref{eq:comparabilityofnorms}.a) and a  triangle inequality). 
The possibly smaller constant $M_\nc$ enters in  Theorem \ref{thm:abstract_main}.a., while  $\Lambda_\M $ appears  in 
 Theorem \ref{thm:abstract_main}.b.  and Theorem \ref{thm:aux1} below.  
The above conditions control the transfer from $V_h$ into $\M(\T)$ via $I_\nc: V+V_h+\M(\T)\to \M(\T)$. 

The transfer from  $\M(\T)$ into $V_h$ is modeled by some linear map  $I_h: \M(\T) \rightarrow V_h$  
that  is bounded in the sense that there exists some constant $\Lambda_h>0$ such that
\begin{equation} \label{eq:mod_interpolation_estimate}
\|v_\nc - I_h v_\nc\|_h \le \Lambda_h \trinl v_\nc - v \trinr_\pw
\quad\text{for all }v_\nc \in \M(\T) \text{ and for all }  v \in V.
\end{equation} 

The examples of this paper concern the discrete norm from  \eqref{common:norm}-\eqref{eq:jh} and then the estimates of this subsection follow for 
 piecewise quadratic discrete spaces. 

\begin{example}[\eqref{eq:Inc_bound}-\eqref{eq:continuityconstantsccMnc} hold for  $V_h \subseteq  P_2(\T)$ and  \eqref{common:norm}-\eqref{eq:jh}]\label{ex:new1} \label{ex:new2} 
Suppose that the discrete norm $\|\bullet\|_h$  is defined by \eqref{common:norm}-\eqref{eq:jh} and  $V_h \subseteq  P_2(\T)$. 
Then \eqref{eq:comparabilityofnorms} and  \eqref{eq:Inc_bound}-\eqref{eq:continuityconstantsccMnc} follow. 
\end{example} 

\begin{proof}[Proof of \eqref{eq:Inc_bound}-\eqref{eq:continuityconstantsccMnc}]
Given any $v_2\in P_2(\T)$ and any $v\in V$, a  triangle inequality shows
\[
\|v_2- I_\M v_2 \|_h \le \|v_2 - JI_\M v_2 \|_h + \| I_\M v_2 - JI_\M v_2 \|_h=: t_1+t_2.
\]
Theorem~\ref{lemmaccapproximationproperties}.c controls the first term $t_1:=\|v_2 - JI_\M v_2 \|_h \lesssim \|v_2 -v\|_h$ on the right-hand side.
 Since $j_h (v_\M)=0$ in \eqref{eq:jh} vanishes for $v_\M:=I_\M v_2 \in \M(\T)$, the second term 
$ t_2:=\trinl  v_\M  - Jv_\M \trinr_\pw\lesssim  \trinl v_\M -  v \trinr_\pw$ with  Lemma~\ref{companion}.d  in the last step.
Since  $\|(1-\Pi_0) D^2_\pw v_2\|=0$ vanishes for $v_2 \in P_2(\T)$, 
Theorem \ref{int_err}.a shows 
\(\trinl v_\M -  v_2\trinr_\pw \lesssim j_h (v_2) = j_h (v_2-v)   \le   \| v_2 -  v \|_h \)
with  \eqref{common:norm}-\eqref{eq:jh} in the last two steps.
This and a triangle inequality prove $ t_2\lesssim   \| v_2 -  v \|_h$.
The combination of the estimates for $t_1+t_2\lesssim   \| v_2 -  v \|_h$
proves  \eqref{eq:Inc_bound}; and  \eqref{eq:Inc_bound} immediately  implies \eqref{eq:continuityconstantsccMnc} as discussed above.
\end{proof}

\subsection{Sufficient conditions for best-approximation}\label{sec:abstractSufficientconditionsforbest-approximation}
The bilinear forms $A_h, a_\pw, b_h, c_h:(V_h+\M(\T))  \times  ( V_h+\M(\T))  \rightarrow {\mathbb R}$ in  the discrete problem \eqref{eq:discrete2}, 
(all bounded because $ V_h+\M(\T)$ is finite dimensional) 
read
\begin{align} \label{eq:bilinear}
A_h(\widehat{v},\widehat{w}):=
 a_\pw (\widehat{v},\widehat{w}) +  b_h(\widehat{v}, \widehat{w}) 
  + c_h(\widehat{v},    \widehat{w}) \quad\text{for all }\widehat{v},\widehat{w}\in  V_h+\M(\T).
 \end{align} 
The key assumption in abstract form is the {\it discrete consistency condition} with a  constant $ 0< \Lambda_\dc <\infty$: 
All functions  $v_\nc \in \M(\T)$, $ w_h \in V_h$, and all $v ,w\in V$ satisfy 
\begin{equation}\label{eq:dis_consistency}
 a_\pw(v_\nc, w_h - I_\nc w_h) + b_h(v_\nc, w_h- I_\nc w_h) \le
\Lambda_\dc   \trinl v_\nc-v \trinr_\pw \|w_h -w\|_h .
\end{equation}
(This is a straightforward generalization of \eqref{a1} from Subsection~\ref{sec:motivation}.)
Assume  that $b_h :(V_h + \M(\T))\times (V_h + \M(\T))\to\mathbb{R} $ is bounded in $V_h + \M(\T)$
by  a  constant $  0< M_{\rm b} <\infty$ and 
vanishes in $\M(\T)\times \M(\T)$ in the sense that
all $ v_h, w_h \in V_h $ and all $v_\nc,w_\nc \in \M(\T)$ satisfy 
\begin{align}
b_h(v_\nc,w_\nc)&=0, \label{eqnanihilationpropertyonbh} \\
b_h(v_h + v_\nc, w_h+ w_\nc)  &\le M_{\rm b} \|v_h+v_\nc\|_h \|w_h+ w_\nc\|_h. 
\label{eq:b_bound}
\end{align}
Suppose that the bilinear form $c_h:(V_h + \M(\T))\times (V_h + \M(\T))\to\mathbb{R}$ and  a constant  $0<\Lambda_{\rm c} <\infty$ satisfy 
\begin{equation} \label{eq:c_bound}
c_h(v_h, w_h) \le \Lambda_{\rm c} \| v - v_h\|_h  \|w- w_h\|_h \quad\text{for all }v_h, w_h \in V_h \text{ and } v,w \in V.
\end{equation}

\begin{thm}[best-approximation] \label{thm:abstract_main} 
Suppose \eqref{eq:a_ellipticity}-\eqref{eq:comparabilityofnorms} and  { \eqref{eq:Inc_bound}}- 
\eqref{eq:c_bound}. Let 
 $u \in V$ solve \eqref{eq:weakabstract} and let $u_h \in V_h$ solve \eqref{eq:discrete2}.  Then 
\[
(a) \; \|u-u_h\|_h \le C_{\qo} \trinl u - I_\nc u \trinr_\pw \text{ and }
(b) \;  \trinl  u- J I_\nc u_h \trinr \le   (1+\Lambda_\nc)(1+\Lambda_{\rm J} ) \| u - u_h \|_h
\]
hold with the constant 
$C_\qo:= 1+ \Lambda_h + \alpha^{-1}(
 \Lambda_0 M_\nc + \Lambda_h  (1+M_{\rm b})+ \Lambda_\dc +\Lambda_{\rm c} (1+\Lambda_h)) $.
\end{thm}

{
\begin{rem} The Morley FEM is included in the (non-symmetric) abstract framework of Theorem \ref{thm:abstract_main}  and leads 
to a sub-optimal best-approximation constant  $C_{\rm qo} = 1+ \Lambda_0$.
\end{rem}}

The error analysis of a post-processing dates back at least to \cite{BS05}  with a design of an enrichment operator for C$^0$IP functions
replaced here by the smoother $J I_\M$. For $F\in H^{-s}(\Omega)$ with $2-\sigma \le s \le 2$ (and   $u\in H^{4-s}(\Omega)$ from elliptic  regularity),
Theorem~\ref{thm:abstract_main} verifies
 \begin{align*}
 \trinl  u- J I_\nc u_h \trinr \le 
 C_{\rm qo}  (1+\Lambda_\nc)(1+\Lambda_{\rm J} )\trinl u - I_\nc u\trinr_\pw\lesssim h_{\max }^{2-s} \| F \|_{H^{-s}(\Omega)}. 
\end{align*}

\subsection{Proofs}
\label{subsectionProofoftheoremthm:abstract_main}
Abbreviate  $e_h:= I_h I_\nc u - u_h \in V_h$ and $(a_\pw+b_h)(\bullet,\bullet):= a_\pw(\bullet,\bullet)+b_h(\bullet,\bullet)$ for the sum of the  bilinear forms.

\begin{lem}[Key identity] \label{lem:key_identity}  It holds 
\begin{align*}
A_h(e_h,e_h) & = a_\pw(u, (1-J)I_\nc e_h ) + (a_\pw+b_h)  ((I_h-1) I_\nc u , e_h)
\\ & \quad +   (a_\pw+b_h)(I_\nc u, e_h-I_\nc e_h) +  c_h(I_h I_\nc u, e_h).
\end{align*}
\end{lem}

\begin{proof} 
The test function $v:=J I_\nc e_h\in V$ in  \eqref{eq:weakabstract} and the test 
function $v_h:= e_h\in V_h$ in \eqref{eq:discrete2} lead to
\[
a(u, J I_\nc e_h) =F(J I_\nc e_h)  = A_h(u_h,e_h).
\]
This and the definition $e_h= I_h I_\nc u - u_h$ result in 
\[
A_h(e_h,e_h) = A_h(I_h I_\nc u, e_h)- a(u, J I_\nc e_h). 
\]
The identity  $a_\pw(u, I_\nc e_h)= a_\pw(I_\nc u, I_\nc e_h )$ from \eqref{eq:best_approx} shows that this is equal to 
\[
a_\pw(u, (1-J)I_\nc e_h ) + A_h(I_h I_\nc u, e_h) -a_\pw(I_\nc u, I_\nc e_h ).
\]
The last term $a_\pw(I_\nc u, I_\nc e_h )$ is part of $A_h(I_\nc u, I_\nc e_h )$ by  \eqref{eq:bilinear}, while 
 $b_h(I_\nc u, I_\nc e_h )=0$ owing to  \eqref{eqnanihilationpropertyonbh}. 
This and  elementary algebra conclude the proof. 
\end{proof}

\begin{lem} \label{lem:term1} 
The assumptions  \eqref{eq:best_approx}, \eqref{eq:right_inverse}, \eqref{eq:lambda0},
 and 
\eqref{eq:continuityconstantsccMnc} imply
$$a_\pw(u, (1-J)I_\nc e_h ) \le  \Lambda_0 M_\nc  \trinl u - I_\nc u \trinr_\pw  \| e_h\|_h.$$ 
\end{lem}
\begin{proof}
Set $v_\nc:= I_\nc e_h$ and recall $I_\nc (v_\nc - J v_\nc)=0$ from \eqref{eq:right_inverse}. 
The orthogonality  \eqref{eq:best_approx}   and the Cauchy inequality with respect to  $a_\pw(\bullet,\bullet)$ imply 
\[
 a_\pw(u, (1-J) v_\nc ) = a_\pw(u - I_\nc u, (1-J) v_\nc  ) \le  \trinl u - I_\nc u \trinr_\pw \trinl (1-J) v_\nc \trinr_\pw. 
\]
The boundedness of $1-J$ in \eqref{eq:lambda0} and  the boundedness of $I_\nc$ in \eqref{eq:continuityconstantsccMnc} show
\[
\trinl (1-J) v_\nc \trinr_\pw   \le \Lambda_0  \trinl v_\nc \trinr_\pw \le \Lambda_0 M_\nc \| e_h\|_h. 
\]
The combination of the two displayed inequalities  concludes the proof. 
\end{proof}

\begin{lem}\label{lem:term2} 
The assumptions \eqref{eq:comparabilityofnorms}, \eqref{eq:mod_interpolation_estimate},  and   \eqref{eq:b_bound}  imply 
$$ 
 (a_\pw+b_h)  ((I_h-1) I_\nc u , e_h)
\le  \Lambda_h  (1+M_{\rm b}) \trinl u - I_\nc u\trinr_\pw \|e_h\|_h. 
$$  
\end{lem}

\begin{proof}
For  $e_h \in V_h$ and  $I_\nc u \in \M(\T)$, the Cauchy inequality  plus (\ref{eq:comparabilityofnorms}.a) show
\[ a_\pw((I_h-1) I_\nc u , e_h)\le  \|(1-I_h) I_\nc u \|_h \|e_h\|_h .\]
This  and the   boundedness of $b_h$ in   \eqref{eq:b_bound} result in 
\begin{align*}
 (a_\pw+b_h)  ((I_h-1) I_\nc u , e_h) \le     (1+M_{\rm b}) \|(1-I_h) I_\nc u \|_h \|e_h\|_h.
\end{align*}
The inequality \eqref{eq:mod_interpolation_estimate} with $v_\nc=I_\nc u$ and $v = u$ reads 
$$\|(1-I_h) I_\nc u \|_h \le \Lambda_h \trinl u - I_\nc u\trinr_\pw.$$ 
The combination of the last two displayed inequalities concludes the proof. 
\end{proof}

\begin{lem}\label{lem:term3} 
The assumptions \eqref{eq:comparabilityofnorms}, \eqref{eq:mod_interpolation_estimate}, \eqref{eq:dis_consistency}, and \eqref{eq:c_bound} imply
$$ 
  (a_\pw+b_h)(I_\nc u, e_h-I_\nc e_h) +  c_h(I_h I_\nc u, e_h)
 \le \left(\Lambda_\dc +\Lambda_{\rm c} (1+\Lambda_h)    \right) \trinl u - I_\nc u\trinr_\pw \| e_h\|_h.
$$
\end{lem}

\begin{proof}
The   discrete consistency condition \eqref{eq:dis_consistency} for $v_\nc:=I_\nc u$, $w_h:= e_h$, $v:=u$, and $w:=0$ 
lead to the upper bound 
\[
 (a_\pw+b_h)(I_\nc u, e_h-I_\nc e_h)\le \Lambda_\dc \trinl u - I_\nc u\trinr_\pw \|e_h\|_h
\]
for  the first term on the left-hand side of the asserted estimate.
The remaining contribution $c_h(I_h I_\nc u, e_h)$  is controlled in \eqref{eq:c_bound} with $v_h:=I_hI_\nc u$,  $w_h:= e_h $, $v=u$, and $w=0$ by 
\[
c_h(I_h I_\nc u, e_h)\le\Lambda_{\rm c} \| u - I_h I_\nc u\|_h \|e_h\|_h .
\]
 A triangle inequality in $\|\bullet\|_h$, \eqref{eq:mod_interpolation_estimate} with 
$v_\nc:= I_\nc u$ and $v=u$, and   (\ref{eq:comparabilityofnorms}.b) 
show 
\[
 \| u - I_h I_\nc u\|_h\le (1+\Lambda_h)  \trinl u - I_\nc u\trinr_\pw.
 \]
The combination of the resulting inequalities concludes the proof.
\end{proof}

\begin{proof}[Proof of best-approximation in Theorem \ref{thm:abstract_main}.a]
The discrete ellipticity \eqref{eq:a_ellipticity}  is followed by Lemma~\ref{lem:key_identity} with terms controlled in 
 Lemmas \ref{lem:term1}-\ref{lem:term3}. This leads (after a division by  $ \|e_h\|_h$, if positive) to 
\[
\alpha \|e_h\|_h \le
\left(  \Lambda_0 M_\nc + \Lambda_h  (1+M_{\rm b})+ \Lambda_\dc +\Lambda_{\rm c} (1+\Lambda_h)    \right)
 \trinl u - I_\nc u \trinr_\pw.
\]
 On the other hand,  $ \|(I_h -1) I_\nc u \|_h \le \Lambda_h \trinl u - I_\nc u\trinr_\pw$ from 
 \eqref{eq:mod_interpolation_estimate} 
 for $v_\nc:= I_\nc u$ and  $v:=u$.  Triangle inequalities in $\|\bullet\|_h$,  (\ref{eq:comparabilityofnorms}.b),  and the last two inequalities result in 
\begin{align*}
\|u-u_h\|_h & \le \trinl u-I_\nc u \trinr_\pw + \|I_\nc u - I_h I_\nc u \|_h + \|e_h \|_h \le  C_{\qo} \trinl u - I_\nc u \trinr_\pw
\end{align*}
with the constant $C_{\qo}$ displayed in the assertion. 
 \end{proof}

\noindent{\it Proof for post-processing in Theorem~\ref{thm:abstract_main}.b.}  The assertion (b) is formulated in terms of  $u$ and $u_h$ but holds for general 
 $v_h\in V_h$ and $v\in V$ and the abbreviation  $v_\M:= I_\M v_h\in M(\T)$. 
A triangle inequality and Lemma~\ref{companion}.d prove
\[
 \trinl v - J v_\nc  \trinr 
 \le \trinl v-v_\nc \trinr_\pw + \trinl (1- J)v_\nc \trinr_\pw
 \le(1+ \Lambda_{\rm J})\trinl v-v_\nc \trinr_\pw .
\]
A triangle inequality, (\ref{eq:comparabilityofnorms}.a) twice, 
and \eqref{eq:Inc_bound} show
\[
\trinl v-v_\nc \trinr_\pw \le \trinl v-v_h \trinr_\pw+ \| v_h-v_\nc  \|_h
\le(1+\Lambda_\M)\|  v-v_h \|_h .
\]
The combination of the two displayed estimates reads 
$$
\trinl v - J v_\nc  \trinr \le (1+ \Lambda_{\rm J})(1+\Lambda_\M)\|  v-v_h \|_h .  \qquad \qquad \qquad 
\qed
$$


\section{Weaker and piecewise Sobolev norm error estimates}\label{sec:lower_order}

\subsection{Assumptions and result}
This subsection presents one further condition sufficient for a lower-order a priori error estimate for the discrete problem \eqref{eq:discrete2} 
beyond the hypotheses of Subsections~\ref{sec:abstractSeconddiscretisation}-\ref{sec:abstractSufficientconditionsforbest-approximation}:
The {\em dual discrete consistency}  with a constant $ 0\le \Lambda_{\rm ddc} <\infty$ asserts that 
 any  $v_h \in V_h$, $w_\nc \in \M(\T)$, and any $v,w \in V$ satisfy 
\begin{equation}
 a_\pw ( v_h -I_\nc v_h, w_\nc)+b_h( v_h, w_\nc)      \le  \Lambda_{\rm ddc}\| v-v_h \|_h   \trinl w- w_\nc \trinr_\pw .
\label{eq:dual_dis_consistency} 
\end{equation}


{ \begin{rem}[symmetry]\label{exampleoneq:dual_dis_consistency} 
If the bilinear form $ (a_\pw + b_h)(\bullet,\bullet) $   is symmetric, then 
\eqref{eq:dis_consistency}-\eqref{eqnanihilationpropertyonbh} imply 
\eqref{eq:dual_dis_consistency}.  ( Rewrite the left-hand side in \eqref{eq:dual_dis_consistency} with \eqref{eqnanihilationpropertyonbh} and symmetry 
into the left-hand side of \eqref{eq:dis_consistency} with $v_h$ and $w_\M$ replacing $w_h$ and $v_\M$ to establish \eqref{eq:dual_dis_consistency}.)
\end{rem}}


\noindent Since $u_h\in V_h$ may not belong to $H^{s}(\Omega)$ for $2-\sigma \le s \le 2$  in general, the post-processing $J I_\nc u_h \in V$ arises in 
the duality argument with $0 < \sigma \le 1$   from Example~\ref{rem:dual}.

\begin{thm}[lower-order error estimates]\label{thm:aux1}
Suppose  the assumptions of Theorem~\ref{thm:abstract_main}, 
 \eqref{eq:dual_dis_consistency}, and   $2-\sigma \le s \le 2$. Then there exist constants 
  $\constant{}\label{2}(s), \constant{}\label{3}(s)>0$ such that (a)-(b) hold for any 
$F\in H^{-s}(\Omega) $ with solution $u \in V$ to \eqref{eq:weakabstract} and the solution $u_h \in V_h$ to \eqref{eq:discrete2}. (a)  $\|u - J I_\nc u_h \|_{H^{s}(\Omega)}  \le\constant[\ref{2}](s)  h_{\max}^{2-s} \|u-u_h\|_h $ and
(b),  if   $u_h \in  P_2(\T)$, then $   \|u -  u_h \|_{{H^s}(\T)} \le  \constant[\ref{3}](s) h_{\max}^{2-s} \|u-u_h\|_h$.
\end{thm}

\subsection{Duality and algebra}\label{sec:dual}
The duality of  $H^{-s}(\Omega)$ and  $H^{s}_0(\Omega)$ reveals for the exact solution $u\in V $ to  \eqref{eq:weakabstract}
and the post-processing $v:= J I_\nc u_h \in V$ of the discrete solution  $u_h \in V_h$ to \eqref{eq:discrete2} that
\[
\|u -v\|_{H^{s}(\Omega)} = \sup_{0 \ne G \in H^{-s}(\Omega)} \frac{G(u-v)}{\|G\|_{H^{-s}(\Omega)}}= G(u-v).
\]
The supremum is attained for some  $G\in H^{-s}(\Omega)\subset V^*$ with norm  $\|G\|_{H^{-s}(\Omega)} =1$ 
owing to a corollary of the Hahn-Banach theorem. The functional $a(z,\bullet)=G\in V^*$ has a 
 unique Riesz representation $z\in V$ in the Hilbert space $(V,a)$;  
 $z\in V$  is the weak solution to the PDE  $\Delta^2 z=G $. 
The   elliptic  regularity (as in Example \ref{rem:dual}) leads to $ z \in V \cap H^{4-s}(\Omega) $ with $2\le 4-s\le 2+\sigma$ and \eqref{eq:extrareg}; hence
\[
\|u -v\|_{H^{s}(\Omega)}=a(u-v,z)\quad\text{and}\quad
\|z \|_{ H^{4-s}(\Omega)}\le C_{\rm reg}.
\]
The proof of Theorem \ref{thm:aux1} consists of  a series of lemmas to establish an upper bound  of
$a( u-v , z)$ for the above $ z \in V \cap H^{4-s}(\Omega) $. The  notation 
\[
v:= J I_\nc u_h \in V , \qquad  z_h :=I_h I_\nc z \in V_h\quad\text{and}\quad\zeta:= J I_\nc z_h \in V
\]
for the discrete, exact, and dual solution $u_h,u,$ and $z$ applies throughout this section.

\begin{lem}[Key identity]\label{ref:key_identity}  It holds
\begin{align*}
&a(u-v,z) = a(u-v,z - \zeta) +a_\pw(u_h-v, \zeta-z_h) + a_{\pw} (I_\nc u_h-v, z_h -I_\nc z_h) \notag \\
& \quad   +  a_{\pw} ( u_h-I_\nc u_h , I_\nc z_h -\zeta)+ A_h(u_h,z_h) - a_\pw (u_h, I_\nc z_h) + a_\pw (u_h- I_\nc u_h,z_h).
\end{align*}
\end{lem}

\begin{proof} 
Let $\zeta \equiv J I_\nc z_h \in V$ substitute the test function $v$ in  \eqref{eq:weakabstract}. This and the test function $v_h:= z_h$ in \eqref{eq:discrete2} lead to
\[
a(u, \zeta) = F(\zeta) =A_h(u_h,z_h).
\]
This identity    and elementary algebra result in 
\begin{align*}
a(u-v,\zeta) & = a_\pw(u_h-v, \zeta-z_h) + a_{\pw} (I_\nc u_h-v, z_h) +a_{\pw} ( u_h , (1-J) I_\nc z_h)\notag \\
& \quad  +  A_h(u_h,z_h) - a_\pw (u_h, I_\nc z_h) +  a_{\pw} ( u_h-I_\nc u_h , z_h ).
\end{align*}
The formulas  $a_\pw((1-J) I_\nc u_h, I_\nc z_h) =0= a_\pw(I_\nc u_h, (1-J)I_\nc z_h) $  (from \eqref{eq:best_approx} and \eqref{eq:right_inverse})
and   elementary algebra conclude the proof.
\end{proof}

\subsection{Elementary bounds}
\begin{lem} \label{lem:terms1to4}  
Each of the following  terms  
(a)~$a(u-v,z - \zeta)$,   (b)~$  a_\pw(u_h-v, \zeta-z_h) $,
 (c)~$a_{\pw} (I_\nc u_h-v, z_h -I_\nc z_h) $, and  
(d)~$a_{\pw} ( u_h-I_\nc u_h , I_\nc z_h -\zeta) $ 
is bounded in modulus by a constant 
 $ \le {\left (1+ (1+\Lambda_\jc) (1+\Lambda_\nc) \right) ^2} (1+ \Lambda_h)$ 
 times $  \|u-u_h\|_h \trinl  z- I_\nc z \trinr_\pw$. 
\end{lem}

\begin{proof}
The assumption \eqref{eq:Inc_bound} (with $(v,v_h)$ replaced by $(u,u_h)$ and $(z,z_h)$)  implies 
\begin{align} \label{eq:uh-I_nc uh}
\|u_h- I_\nc u_h\|_h & \le \Lambda_\nc \|u-u_h\|_h\quad \text{and}\quad \|z_h- I_\nc z_h\|_h  \le \Lambda_\nc \|z-z_h\|_h.
\end{align}
Recall $v\equiv J I_\nc u_h$ and $\zeta \equiv J I_\nc z_h \in V$ to deduce from  Lemma~\ref{companion}.d  that
\begin{align} \label{eq:diff_interpolation}
\trinl v- I_\nc u_h\trinr_\pw & \le \Lambda_\jc \trinl u- I_\nc u_h\trinr_\pw  \quad \text{and}\quad
\trinl \zeta- I_\nc z_h\trinr_\pw  \le \Lambda_\jc \trinl z- I_\nc z_h\trinr_\pw.
\end{align}
The combination of \eqref{eq:uh-I_nc uh}-\eqref{eq:diff_interpolation}  with   (\ref{eq:comparabilityofnorms}.a) 
and triangle inequalities lead to 
\begin{align} \label{eq:u-v}
\trinl u-v\trinr & \le \trinl u-I_\nc u_h \trinr_{\pw} + \trinl v-I_\nc u_h  \trinr_{\pw} \le (1+\Lambda_\jc ) \trinl  u-I_\nc u_h\trinr_\pw \nonumber \\
& \le (1+\Lambda_\jc ) (\trinl  u- u_h\trinr_\pw + \| u_h - I_\nc u_h\|_h) \le (1+\Lambda_\jc )  (1+\Lambda_\nc ) \|u-u_h\|_h.  
\end{align}
The above arguments have not utilized any  solution property and hence also apply for  $(z,\zeta,z_h)$ replacing $(u,v,u_h)$  to reveal (\text{instead of } \eqref{eq:u-v})
\begin{align} \label{eq:z-zeta}
\trinl z-\zeta\trinr & \le (1+\Lambda_\jc )  (1+\Lambda_\nc )  \|z-z_h\|_h.
\end{align}
Consider  $v_\nc:=I_\nc z \in \M(\T)$  with $z_h \equiv I_h I_\nc z=I_h v_\nc $  in \eqref{eq:mod_interpolation_estimate} to show
\begin{equation}\label{eq:z-zetaintermediateccnew}
 \|z_h - I_\nc z\|_h \le \Lambda_h \trinl z- I_\nc z\trinr_\pw. 
 \end{equation}
 This, a  triangle inequality, and  (\ref{eq:comparabilityofnorms}.b)  result in 
\begin{align}\label{eq:intermediate}
 \|  z-z_h\|_h\le (1+ \Lambda_h) \trinl  z- I_\nc z\trinr_\pw.  
 \end{align}
The combination of  \eqref{eq:z-zeta} and  \eqref{eq:intermediate}  proves
\begin{align}\label{eq:intermediatepart2}
 \trinl z-\zeta\trinr \le (1+\Lambda_\jc )  (1+ \Lambda_\nc) (1+ \Lambda_h) \trinl  z- I_\nc z \trinr_\pw. 
\end{align}
  
\begin{proof}[Proof of (a)]
This follows from a Cauchy inequality  plus \eqref{eq:u-v} and \eqref{eq:intermediatepart2}.
\end{proof}

\begin{proof}[Proof of (b)]
A triangle inequality, (\ref{eq:comparabilityofnorms}.a), and \eqref{eq:u-v} verify
\[
\trinl u_h-v\trinr_\pw \le \| u_h-v\|_h  \le (1+(1+\Lambda_\jc )  (1+\Lambda_\nc )) \|u-u_h\|_h.
\]
The triangle inequality with (\ref{eq:comparabilityofnorms}.a) 
and \eqref{eq:intermediate}-\eqref{eq:intermediatepart2} show
\[
\trinl \zeta-z_h \trinr_\pw \le \| \zeta-z_h \|_h\le (1+ \Lambda_h)\left(1+  (1+\Lambda_\jc )  (1+ \Lambda_\nc)\right) \trinl  z- I_\nc z \trinr_\pw. 
\]
A Cauchy inequality and the preceding estimates conclude the proof of (b).
\end{proof}
 
\begin{proof}[Proof of (c)]
 The estimate   \eqref{eq:diff_interpolation}, a triangle inequality, (\ref{eq:comparabilityofnorms}.a),  and \eqref{eq:uh-I_nc uh}  show
 \begin{equation}\label{eqcclowerorderestimatesproofofpartc}
\trinl  v-  I_\nc u_h\trinr_\pw \le \Lambda_{\rm J}   \trinl  u -   I_\nc u_h \trinr_\pw \le \Lambda_\jc (1+ \Lambda_\nc) \|u-u_h\|_h.
 \end{equation}
The combination of \eqref{eq:uh-I_nc uh}  and  \eqref{eq:intermediate} after (\ref{eq:comparabilityofnorms}.a) leads to 
\begin{align}\label{eq:partc}
\trinl z_h- I_\nc z_h\trinr_\pw \le  \|z_h- I_\nc z_h\|_h  \le \Lambda_\nc  (1+ \Lambda_h) \trinl  z- I_\nc z\trinr_\pw.  
\end{align}
 A Cauchy inequality and the preceding estimates conclude the proof of (c).
\end{proof}

\begin{proof}[Proof of (d)]
This follows from a Cauchy inequality  with (\ref{eq:comparabilityofnorms}.a) and the estimates for  $ \| u_h-I_\nc u_h \|_h$ in \eqref{eq:uh-I_nc uh}
and  $ \trinl   \zeta - I_\nc z_h  \trinr_\pw$ in \eqref{eq:diff_interpolation}.
\end{proof}
\renewcommand{\endproof}{}
\end{proof}

\subsection{Discrete consistency bounds}

\begin{lem} \label{lem:term5} 
It holds
\begin{align*}
& A_h(u_h,z_h) - a_\pw (u_h, I_\nc z_h) + a_\pw (u_h- I_\nc u_h,z_h) \\
& \qquad \le  (1+\Lambda_h) ( ( 2+ M_{\rm b}) \Lambda_\nc^2  + (\Lambda_\dc  + \Lambda_{\rm ddc})(1+ \Lambda_\nc) \Lambda_{\rm J}+  \Lambda_{\rm c})  \|u-u_h\|_h \trinl z- I_\nc z\trinr_\pw.
\end{align*}
\end{lem}

\begin{proof} 
Recall  $b_h( I_\nc u_h, I_\nc z_h)=0$ from  \eqref{eqnanihilationpropertyonbh} and exploit \eqref{eq:bilinear} with elementary (but lengthy) algebra 
to check that the  left-hand side {\rm LHS} of the assertion is equal to 
 \begin{align}\label{eq:key_part1a}
 \text{\rm LHS}&=  (2a_\pw+b_h)(u_h- I_\nc u_h,z_h-I_\nc z_h)  + (a_\pw+ b_h)(I_\nc u_h,z_h-I_\nc z_h) \\
 \label{eq:key_part1b}
&\quad  + (a_\pw ( u_h -I_\nc u_h, I_\nc z_h  )+b_h( u_h, I_\nc z_h))  +c_h(u_h,z_h)
 \end{align}
with the short notation, e.g., $(2a_\pw+b_h)(\bullet,\bullet):= 2a_\pw(\bullet,\bullet)+b_h(\bullet,\bullet)$, for the sum of the  bilinear forms announced in Subsection \ref{subsectionProofoftheoremthm:abstract_main}.  The two lines \eqref{eq:key_part1a}-\eqref{eq:key_part1b} of expressions 
for the {\rm LHS} give rise to four estimates. The 
continuity of $a_\pw(\bullet,\bullet)$  and  $b_h(\bullet,\bullet)$ in \eqref{eq:b_bound}, \eqref{eq:uh-I_nc uh},  and \eqref{eq:partc} prove
 \begin{align*}
(2a_\pw+ b_h)(u_h -I_\nc u_h, z_h - I_\nc z_h)  
\le (2+ M_{\rm b})   \Lambda_\nc^2 (1+ \Lambda_h) \|u-u_h\|_h  \trinl z- I_\nc z \trinr_\pw.
\end{align*}
The discrete consistency \eqref{eq:dis_consistency} leads in the last term in  \eqref{eq:key_part1a} to a product of 
$\trinl  v-  I_\nc u_h\trinr_\pw$ controlled in \eqref{eqcclowerorderestimatesproofofpartc}
and  $ \|z-z_h\|_h$ controlled in  \eqref{eq:intermediate}. This results in
\begin{align*}
 (a_\pw+ b_h)(I_\nc u_h,z_h-I_\nc z_h) \le 
 \Lambda_\dc \Lambda_{\rm J} (1+\Lambda_\nc)(1+ \Lambda_h)\|u-u_h\|_h \trinl z- I_\nc z\trinr_\pw.
\end{align*}
The dual discrete consistency in \eqref{eq:dual_dis_consistency} applies to the first two terms
in \eqref{eq:key_part1b} and leads to   $\Lambda_{\rm ddc} \|u-u_h\|_h$ times $\trinl \zeta -I_\nc  z_h\trinr_\pw $
controlled with (\ref{eq:comparabilityofnorms}.b) in \eqref{eq:diff_interpolation}. {This with (\ref{eq:uh-I_nc uh}.b) and \eqref{eq:intermediate} result in
\begin{align*}
a_\pw ( u_h -I_\nc u_h, I_\nc z_h  )+b_h( u_h, I_\nc z_h)  
 \le  \Lambda_{\rm ddc} \Lambda_{\rm J}(1+\Lambda_\M )(1+\Lambda_h) \|u-u_h\|_h \trinl z- I_\nc z\trinr_\pw.
\end{align*}}
The last term in  \eqref{eq:key_part1b} is controlled in \eqref{eq:c_bound}. This and   \eqref{eq:intermediate}  show 
\[ c_h(u_h,z_h) \le \Lambda_{\rm c} \|u-u_h\|_h \|z-z_h\|_h  \le
\Lambda_{\rm c} (1+ \Lambda_h) \|u-u_h\|_h \trinl z- I_\nc z\trinr_\pw .\]
A combination of the preceding four estimates with \eqref{eq:key_part1a}-\eqref{eq:key_part1b} concludes the proof.
 \end{proof}

\subsection{Proof of Theorem \ref{thm:aux1}}\label{subsectProofofTheoremthm:aux1}

{
Given $2-\sigma \le s\le 2$, there exists a constant $ 0< C_{\rm int}(s)  <\infty$ (which exclusively depends on the shape regularity of $\T$ and $s$)
such that the solution $z \in V$  of the dual problem  in  Section~\ref{sec:dual} satisfies  (with Lemma \ref{interpolationerrorestimatesI}.c)  that
\begin{equation}
\trinl  z - I_\nc z \trinr _\pw \le C_{\rm int}(s)  h_{\rm max}^{2-s}  \|z \|_{ H^{4-s}(\Omega)} 
\le   C_{\rm int}(s)   C_{\rm reg}(s)  h_{\rm max}^{2-s} \| G \|_{H^{-s}(\Omega)} .  \label{eq:est}
\end{equation}}
\begin{proof}[Proof of (a)]
Recall $\|u -v\|_{H^{s}(\Omega)}=a(u-v,z)$ from Subsection \ref{sec:dual} and its formula
in Lemma \ref{ref:key_identity}. Lemma \ref{lem:terms1to4} applies to the first four terms and  Lemma  \ref{lem:term5} to the remaining three.
The resulting estimate reads 
\[
\|u -v\|_{H^{s}(\Omega)}\le \constant{} \label{theoremlowerorder}   \|u-u_h\|_h \trinl  z- I_\nc z \trinr_\pw
\]
with 
$\constant[\ref{theoremlowerorder}]=(1+\Lambda_h){\big(4(1+(1+ \Lambda_\jc)(1+\Lambda_\nc))^2} + ( 2+ M_{\rm b}) \Lambda_\nc^2  + (\Lambda_\dc  + \Lambda_{\rm ddc})(1+ \Lambda_\nc) \Lambda_{\rm J}+  \Lambda_{\rm c})\big).$ This and 
\(
\trinl z-I_\nc z\trinr_\pw  \le C_{\rm int}(s)C_{\rm reg }(s) h_{\max}^{2-s}
\)
from \eqref{eq:est} prove  Theorem \ref{thm:aux1}.a.
\end{proof}

\begin{proof}[Proof of (b)] The norm in ${H^s(\T)}=\prod_{T\in\T} H^s(T)$ is the $\ell^2$ norm of those contributions $ \|\bullet  \|_{H^s(T)}$ for all $T\in\T$. The 
Sobolev-Slobodeckii semi-norm over $\Omega$ involves double integrals over $\Omega\times\Omega$ and so is larger   than or equal to the sum of the 
contributions over $T\times T$ for all the triangles $T\in\T$, i.e.,    $ \sum_{T\in\T} |\bullet  |_{H^s(T)}^2\le  |\bullet  |_{H^s(\Omega)}^2$ for any $1<s<2$.
Hence  Theorem \ref{thm:aux1}.a implies 
\begin{equation}\label{eqrefccnewa1b}
 \|u -J I_\nc u_h\|_{H^{s}(\T)} \le \constant[\ref{2}](s)\, h_{\max}^{2-s}  \|u-u_h\|_h\quad\text{for all }s\text{ with }2-\sigma \le s\le2.
\end{equation}
Since  $u_h \in  P_2(\T)$, Theorem~\ref{lemmaccapproximationproperties}.d provides the estimate
\[
 \|u_h -J I_\nc u_h\|_{H^{s}(\T)}\lesssim  h_{\max}^{2-s}  \|u-u_h\|_h.
\]
The triangle inequality in the norm of $H^s(\T)$ concludes the proof of  Theorem~\ref{thm:aux1}.b. 
\end{proof}
{{
\subsection{Verification of {\bf (H)} and $\widehat{\text{\bf (H)}}$ }}
 For  the choice $P=Q= J \circ I_\M$,   Theorem \ref{lemmaccapproximationproperties}.c.  shows that \eqref{quasioptimalsmoother}-\eqref{quasioptimal} hold for all the lowest-order schemes $V_h \subseteq P_2(\T)$ considered in this paper.
This subsection verifies {\bf (H)}-${\widehat{\text{\bf (H)}}}$.
\begin{lem}[Verification of {\bf (H)}-${\widehat{\text{\bf (H)}}}$] \label{lemma:linear_bound} Suppose the assumptions of Theorems~\ref{thm:abstract_main} and \ref{thm:aux1}.  Any $v_h \in V_h$,  $v:= J I_\nc v_h \in V $, $w_h \in V_h$, and $w = JI_{\rm M} w_h \in V $ satisfy
\begin{align*}
{\text{\bf (H)}}  \; A_{h}(v_h,w_h) - a(v,w) & \le \Lambda_1   \|v_{h} - v\|_{h} \|w_{h}  \|_h, \\
{\widehat{\text{\bf (H)}}} \; A_{h}(v_h,w_h) - a(v,w) & \le \Lambda_2 
  \|v_{h} - v\|_{h} \|w_{h} - w \|_h
\end{align*}
with 
$\Lambda_1:= (1+ \Lambda_\M) (\Lambda_{\rm dc} +  \Lambda_{\text{\rm J}} M_{\rm M}  \|J\|  ) + (1 + M_{\rm b}) \Lambda_{\rm M} +\Lambda_{\rm c} $,  and
$\Lambda_2:=\Lambda_{\rm ddc} + (1 + M_{\rm b}) \Lambda_{\rm M}^2  +  (1 + \Lambda_{\rm M})(1+\Lambda_{\rm M}+ \Lambda_{\rm dc}  ) + \Lambda_{\rm c}$. 
\end{lem}
\begin{proof}[Proof of {\bf (H)}]
For $w_{\rm M} := I_{\rm M} w, $ and $v_{\rm M} := I_{\rm M} v $,  \eqref{eq:right_inverse} implies  $w_{\rm M} = I_{\rm M} w_h$,  and
$v_{\rm M} = I_{\rm M} v_h$.   The definition of $A_h(\bullet,\bullet)$, algebraic manipulations, and \eqref{eqnanihilationpropertyonbh}  result in
	\begin{align}
A_{h}(v_h,w_h) - a(v,w) 
	={}& a_{\rm pw}(v_h,w_h) + b_{h}(v_h,w_h) + c_{h}(v_h,w_h) - a(v,w)  \nonumber\\
	&\hspace{-4cm}={} a_{\rm pw}(v_h - v_{\rm M},w_{h}) +
b_{h}(v_h -v_\M,w_{h}) + a_{\rm pw}( v_{\rm M},w_{h} - w_{\rm M})    + b_h( v_{\rm M},w_h - w_{\rm M})  \nonumber \\
&\hspace{-3.6cm}
	+ c_{h}(v_h,w_h)+ a_{\pw}(v_\M, w_\M) - a(v,w) . \label{eqn:algebra_split1}
	\end{align}
The boundedness of $a_{\rm pw}(\bullet,\bullet)$,  \eqref{eq:b_bound},  \eqref{eq:Inc_bound}, and  \eqref{eq:comparabilityofnorms} prove
	\begin{align*} 
	a_{\rm pw}(v_{h} - v_{\rm M},w_{h} ) + b_h(v_h - v_{\rm M},w_h ) \le{}& 
	(1 + M_{\rm b}) \Lambda_{\rm M} \|v_{h} - v\|_{h} \|w_{h}  \|_h.
	\end{align*}
The discrete consistency condition \eqref{eq:dis_consistency} (with $w=0$),  a triangle inequality, \eqref{eq:Inc_bound},  and \eqref{eq:comparabilityofnorms} show
	\begin{align*} 
	a_{\rm pw}(v_{\rm M},w_h - w_{\rm M}) + b_{h}(v_{\rm M},w_{h} - w_{\rm M}) \le  \Lambda_{\rm dc}(1+ \Lambda_\M) \|v_h - v\|_{h} \|w_{h}  \|_h.
	\end{align*}
The bound in \eqref{eq:c_bound} with the choice $w=0$ implies
	\begin{align*} 
	c_{h}(v_h,w_h) \le \Lambda_{\rm c} \|v_{h} - v\|_{h} \|w_{h}  \|_h.
	\end{align*}
The orthogonality condition \eqref{eq:best_approx}, Lemma \ref{lem:MorleyCompanion}.d.,    \eqref{eq:Inc_bound},  and \eqref{eq:continuityconstantsccMnc},   result in 
	\begin{align*} 
	a_{\rm pw}(v_{\rm M},w_{\rm M}) - a(v,w)  &= a_{\pw}((1-J)v_\M, J w_\M) 
\le \Lambda_{\text{J}}(1+\Lambda_\M) M_{\rm M}  \|J\| \|v_{h} - v\|_{h}\|w_{h}  \|_h.
	\end{align*}
A combination of the last four displayed estimates in \eqref{eqn:algebra_split1} leads to the desired result.
\end{proof}
\begin{proof}[Proof of ${\widehat{\text{\bf (H)}}}$]
An alternate split of the left-hand side of the desired estimate leads to 
	\begin{align}
A_{h}(v_h,w_h) - a(v,w) 
	&={} a_{\rm pw}(v_h - v_{\rm M},w_{\rm M}) + b_{h}(v_h,w_{\rm M}) + a_{\rm pw}(v_{h} - v_{\rm M},w_{h} - w_{\rm M}) \nonumber \\ 
&\hspace{-1cm}+ b_h(v_h - v_{\rm M},w_h - w_{\rm M})  
	+ a_{\rm pw}(v_{\rm M},w_h - w_{\rm M}) + b_{h}(v_{\rm M},w_{h} - w_{\rm M}) \nonumber \\
&\hspace{-1cm} +  a_{\rm pw}(v_{\rm M},w_{\rm M}) - a(v,w) + c_{h}(v_h,w_h). \label{eqn:algebra_split}
	\end{align}
	The discrete consistency condition \eqref{eq:dual_dis_consistency} shows
	\begin{align*} 
	a_{\rm pw}(v_h - v_{\rm M},w_{\rm M}) + b_{h}(v_h,w_{\rm M}) \le \Lambda_{\rm ddc} \|v - v_{h}\|_h \trinl w - w_{\rm M} \trinr_{\rm pw}.
	\end{align*}
	The boundedness of $a_{\rm pw}(\bullet,\bullet)$,  \eqref{eq:b_bound},  and \eqref{eq:Inc_bound} prove
	\begin{align*} 
	a_{\rm pw}(v_{h} - v_{\rm M},w_{h} - w_{\rm M}) + b_h(v_h - v_{\rm M},w_h - w_{\rm M}) \le{}& 
	(1 + M_{\rm b}) \Lambda_{\rm M}^2 \|v_{h} - v\|_{h} \|w_{h} - w \|_h.
	\end{align*}
	The discrete consistency condition \eqref{eq:dis_consistency} and \eqref{eq:comparabilityofnorms} lead to 
	\begin{align*} 
	a_{\rm pw}(v_{\rm M},w_h - w_{\rm M}) + b_{h}(v_{\rm M},w_{h} - w_{\rm M}) \le  \Lambda_{\rm dc}(1+ \Lambda_\M) \|v_{h} - v\|_{h} \|w_{h} - w \|_h.
	\end{align*}
	The orthogonality condition \eqref{eq:best_approx} and \eqref{eq:Inc_bound}  result in 
	\begin{align*}
	a_{\rm pw}(v_{\rm M},w_{\rm M}) - a(v,w)  &= -a_{\rm pw}(v-v_{\rm M},w - w_{\rm M}) 
	\le (1 + \Lambda_{\rm M})^2\|v_{h} - v\|_{h} \|w_{h} - w \|_h. 
	\end{align*}
The bound in \eqref{eq:c_bound} implies
	\begin{align*} 
	c_{h}(v_h,w_h) \le \Lambda_{\rm c} \|v_{h} - v\|_{h} \|w_{h} - w \|_h.
	\end{align*}
A substitution of the last five displayed estimates 
in~\eqref{eqn:algebra_split} leads to the desired result.
\end{proof}

\begin{rem}[Theorem \ref{thm:qo2} implies Theorem \ref{thm:aux1}]
For $v:= JI_\M u_h \in V$, recall $\|u -v\|_{H^{s}(\Omega)}=a(u-v,z)$ from Subsection \ref{sec:dual}.
Theorem~\ref{thm:qo2} applies as Lemma \ref{lemma:linear_bound} holds and with \eqref{eq:comparabilityofnorms} leads to 
$a(u-v,z) \le \widehat{C_{\rm qo}} \|u-u_h\|_h \|z-I_\M z\|_h$.  
\end{rem}

}


\section{Modified dGFEM} \label{sec:DGFEM}
\noindent   The bilinear form $A_h(\bullet,\bullet):=A_\dg(\bullet,\bullet)$  \cite{Baker:1977:DG, FengKarakashian:2007:CahnHilliard} is 
defined, for all $v_2,w_2 \in V_h:=P_2(\T)$, by
\begin{align}
 A_\dg(v_2,w_2)
  &:=
   a_\NC (v_2,w_2)  +b_h(v_2, w_2) + c_\dg(v_{2},w_{2}), \label{eq:AhindG} 
   \end{align}
\begin{subequations}
\begin{align}
   b_h(v_2, w_2)& := -\Theta {\cal J}(v_2, w_2) -   {\cal J}(w_2, v_2), \label{eq:bhindG} \\
     {\cal J}(v_2, w_2) & :=  \sum_{E\in\E} 
    \int_E \jump{\nabla_\NC v_2}  \cdot \mean{D^2_\NC w_2}\nu_E  \ds
\end{align}
\end{subequations}
with  $c_\dg(\bullet,\bullet)$  from \eqref{eq:ccdefineJsigma1sigma2}  and given $-1 \le \Theta \le 1$. Let the jumps $[\bullet]_E$ across and the averages $\mean{\bullet}$  at  an edge $E\in\E$ from Subsection~\ref{subsectionEquivalentnorms}
act componentwise. Recall from Theorem~\ref{lemmaonhnormisanorm} the equivalent discrete norms $\|\bullet\|_h\approx \|\bullet\|_\dg$ in $H^2(\T)$ defined in \eqref{common:norm}-\eqref{eq:jh} 
and  \eqref{eq:ccdefinedGnorm}. Set $\Theta=1$ (resp. $\Theta=-1$) to obtain the symmetric (resp. non symmetric) interior penalty Galerkin formulation; 
see \cite{MozoSuli07} for an alternative  formulation.
Appropriate positive parameters $\sigma_1,\sigma_2$  in  \eqref{eq:ccdefineJsigma1sigma2} guarantee \eqref{eq:a_ellipticity}.
 
\begin{lem}[Boundedness and ellipticity of $A_{\dg}(\bullet,\bullet)$] { \cite{FengKarakashian:2007:CahnHilliard,MozoSuli07}} \label{ellipticity}
(a) Any $v_2, w_2 \in P_2(\T)$ satisfy
$ A_{\dg}(v_{2},w_{2}) \lesssim \|v_{2}\|_\dg \|w_{2} \|_\dg.$ 
(b) For  $\Theta =-1$ and any  $\sigma_\dg= \sigma_1=\sigma_2>0$, $\|v_{2}\|_{\dg}^2 \leq A_{\dg}(v_{2},v_{2}) $ holds for all $ v_{2}\in\St$. 
(c)  For $-1 < \Theta \le 1$ and  a sufficiently large parameter  $\sigma_\dg= \sigma_1=\sigma_2>0$, there exists  $\alpha>0$ (which depends on $\sigma_\dg$ and the shape regularity of $\T$) such that $\alpha \|v_{2}\|_{\dg}^2 \leq A_{\dg}(v_{2},v_{2})$ for all  $v_{2}\in\St.$ 
\end{lem}

Throughout this paper, the parameter $\sigma_\dg$ is chosen to guarantee the ellipticity of $A_\dg(\bullet,\bullet)$ in Lemma \ref{ellipticity} 
with the short notation $\sigma_\dg \approx 1 \approx \alpha$.  
The modified dGFEM  \eqref{eq:discrete2} seeks the solution $u_\dg \in P_2(\T)$ to
\begin{equation} \label{eq:dg}
 A_\dg (u_\dg,v_2) = F(JI_\M v_2) 
 \quad\text{for all } v_2 \in P_2(\T).
\end{equation}
 
\begin{thm}[error estimates]\label{thm:dg}  The  solution $u \in V$ to \eqref{eq:weakabstract}  and  the solution $u_\dg \in P_2(\T)$  to \eqref{eq:dg}   satisfy
(a) $\|u- u_\dg\|_h \lesssim  \trinl u- I_\M u \trinr_\pw$ and 
(b)  if $\Theta=1$ and  $F\in H^{-s}(\Omega) $ for  $2-\sigma \le s \le 2$, then  $\|u- J I_\M u_\dg\|_{H^s(\Omega)} + \|u- u_\dg\|_{H^s(\T)}  \lesssim h_{\max}^{2-s}    \| u- u_\dg \|_h$.
\end{thm}

\noindent{\it Overview of the proof.}
The assertion (a) follows from  Theorem \ref{thm:abstract_main}  for the particular spaces, operators, norms, and  bilinear forms defined below. The application of Theorem \ref{thm:abstract_main}  requires the proof of the abstract conditions 
 \eqref{eq:a_ellipticity}-\eqref{eq:comparabilityofnorms}, 
 {\eqref{eq:Inc_bound}} 
 -\eqref{eq:c_bound}.
 The assertion (b) follows from Theorem \ref{thm:aux1} provided 
 \eqref{eq:dual_dis_consistency} holds. 
 
 \medskip
 
\noindent  {\it Setting and first consequences.} Recall  $V_h: = P_2(\T)$  and  the norms $\|\bullet\|_h$ and $\|\bullet\|_\dg$ in 
\eqref{common:norm}-\eqref{eq:jh}  and  \eqref{eq:ccdefinedGnorm}. 
Recall the Morley  interpolation  operator $I_\M$  from Definition \ref{defccMorleyinterpolation} and the 
companion operator $J$ from Lemma \ref{companion}.    
Recall that   Lemma \ref{ellipticity}  guarantees  \eqref{eq:a_ellipticity}-\eqref{eq:comparabilityofnorms}.  
The  dGFEM in \eqref{eq:dg} corresponds to  \eqref{eq:discrete2} with  the solution $u_h:=u_\dg$.  
Example \ref{ex:new2} implies  \eqref{eq:Inc_bound}-\eqref{eq:continuityconstantsccMnc}. 
Set  $I_h:=\text{id}$ and observe  \eqref{eq:mod_interpolation_estimate} holds for $\Lambda_h= 0$.
Recall $A_h(\bullet, \bullet) := A_\dg(\bullet, \bullet) $, $b_h(\bullet, \bullet) :=-\Theta {\cal J}(\bullet, \bullet)- {\cal J}^*(\bullet, \bullet)$, 
and $c_h(\bullet, \bullet) :=c_\dg(\bullet, \bullet) $ in \eqref{eq:bilinear}. 

\begin{proof}[Proof of \eqref{eq:dis_consistency}]  Since the integral $\int_E \jump{\nabla_\pw v_\M} \ds=0$ vanishes for   $v_\M \in \M(\T)$ 
and since   $\mean{D^2_\pw w_2}$ is constant  on any edge $E\in \E$ for  any $w_2 \in P_2(\T)$,
\begin{align}\label{jterm}
\J(v_\M, w_2 ) = \sum_{E \in \E} \int_E \jump{\nabla_\pw v_\nc} \cdot \mean{D^2_\pw w_2} \nu_E \ds=0.
\end{align}
Hence the term $\Theta\J(v_\M, w_2-I_\M w_2)$ disappears below in  definitions of $(a_\pw+b_h)(\bullet,\bullet )$, 
written in the short notation of Subsection \ref{subsectionProofoftheoremthm:abstract_main}; 
$ (a_\pw+ b_h)(v_\M, w_2-I_\M w_2)$ is equal to 
\[
\sum_{T \in \T} \int_K D^2_\pw v_\M : 
D^2_{\pw} (w_2-I_\M w_2) \dx - \sum_{E \in \E} \int_E \jump{\nabla_\pw (w_2-I_\M w_2)} \cdot \mean{D^2_{\pw} v_\M}\nu_E {\ds}  .
\]
A piecewise integration by parts of the term  $a_{\pw} (v_\M, w_2-I_\M w_2 )$ shows equality to
\[
\sum_{E \in \E} \int_E  \left(  \jump{\nabla_\pw (w_2-I_\M w_2) \cdot ( D^2_{\pw} v_\M \: \nu_E) } 
  - \jump{\nabla_\pw (w_2-I_\M w_2) } \cdot \mean{D^2_{\pw} v_\M} \nu_E  \right) \ds .
\]
The product rule for the jump terms results in 
\begin{align}
(a_\pw+ b_h)(v_\M, w_2-I_\M w_2)
 = \sum_{E \in \E} \int_E  \mean{\nabla_\pw (w_2-I_\M w_2) } \cdot \jump{D^2_{\pw} v_\M} \nu_E  \ds. \label{eq:internew}
\end{align}
The further analysis concerns the split of the vector  $\mean{\nabla_{\pw}(w_2-I_\M w_2) }$$ \in P_1(E; {\mathbb R}^2)\equiv  P_1(E)^{2}$   
into normal and tangential components, 
\[
\mean{\nabla_{\pw}(w_2-I_\M w_2) }= \langle \partial (w_2-I_\M w_2)/ \partial \nu_E \rangle_E \nu_E  + \langle \partial (w_2-I_\M w_2)/ \partial s   \rangle_E \tau_E.
\]  
The integral of the normal component $ \langle \partial (w_2-I_\M w_2)/ \partial \nu_E \rangle_E$ over an edge $E\in\E$ vanishes  by definition
of $I_\M w_2$  in Definition \ref{defccMorleyinterpolation}. Since the jump $\jump{\partial^2_{\nu_E \nu_E} v_\M}:= \nu_E \cdot  \jump{   D^2_\pw v_\M } \nu_E$ is constant 
along $E$, the integral 
$\int_E \langle \partial (w_2-I_\M w_2)/ \partial \nu_E \rangle_E \jump{\partial^2_{\nu_E \nu_E} v_\M}\ds=0 $ vanishes. The tangential components 
with  $\jump{\partial^2_{\tau_E \nu_E} v_\M}:=  \tau_E\cdot \jump{D^2 v_\M}\nu_E$
remain in 
\[
(a_\pw+ b_h)(v_\M, w_2-I_\M w_2)
 = \sum_{E \in \E} \int_E  \mean{ \partial (w_2-I_\M w_2) /\partial s }  \jump{ \partial^2_{\tau_E \nu_E}v_\M}  \ds.
\]
The  Hadamard jump condition asserts that the  jump in the derivative of a globally continuous function that is smooth up to the boundary on either side of an interface $E$
points merely  in the normal direction $\nu_E$ only.  The function $Jv_\M$ has a continuous gradient  $\nabla Jv_\M$
and  $\nabla Jv_\M$ is smooth on each triangle $\widehat{T}$ in the HCT  refinement of $\T$. Hence  $\jump{ \partial^2_{\tau_E \nu_E} Jv_\M}=0$ along $E$.
Consequently,
\begin{align} \label{eq:internewcc2}
(a_\pw+ b_h)(v_\M, w_2-I_\M w_2)
 = \sum_{E \in \E} \int_E  \mean{ \partial (w_2-I_\M w_2) /\partial s }  \jump{ \partial^2_{\tau_E \nu_E}(1-J)v_\M}  \ds.
\end{align}
For an interior edge $E=\partial T_+\cap \partial T_-=\partial \widehat T_+\cap \partial \widehat T_-$ with the neighbouring triangles 
$T_\pm\in \T$ and the two neighbouring sub-triangles  $\widehat T_\pm:=\text{\rm conv}(E,\mmid(T_\pm))$
from the HCT  refinement of $\T$ with patches
$\widehat \omega(E)=\text{int}( \widehat T_+\cup \widehat T_-)\subset \omega(E)=\text{int}( T_+\cup T_-)$,
Cauchy and  triangle inequalities show
\begin{align*}
 I(E)&:=\int_E  \mean{ \partial (w_2-I_\M w_2) /\partial s }  \jump{ \partial^2_{\tau_E \nu_E}(1-J)v_\M}  \ds\\
&\le  \frac 12 \left(\| \nabla(w_2-I_\M w_2)|_{T_+ }\|_{L^2(E)} + \| \nabla(w_2-I_\M w_2)|_{T_-} \|_{L^2(E)}\right)\\
& \hspace{3mm}\times  \left(\| D^2(1-J)v_\M|_{\widehat T_+ }\|_{L^2(E)}+\| D^2(1-J)v_\M|_{\widehat T_- }\|_{L^2(E)}\right).
\end{align*}
Since $(w_2-I_\M w_2)|_{\widehat  T_\pm}$ resp. $(1-J)v_\M|_{\widehat T_\pm }$ is a polynomial of degree at most
$2$ resp. $3$ in the triangle $\widehat T_\pm$, the discrete  trace inequalities  
\begin{align*}
 \| \nabla(w_2-I_\M w_2)|_{T_\pm }\|_{L^2(E)} &\le 
h_{E}^{1/2} \constant[\ref{standardshaperegularityconstantpart2}]  \|h_{T_\pm}^{-1} \nabla_\pw (w_2-I_\M w_2) \|_{L^2(\widehat  T_\pm)} \\
\| D^2(1-J)v_\M|_{\widehat T_\pm  }\|_{L^2(E)}&\le  h_{E}^{-1/2} \constant[\ref{standardshaperegularityconstantpart2}]  
  | (1-J)v_\M |_{H^2(\widehat T_\pm)}
\end{align*}
hold for a  constant $ \constant{}\label{standardshaperegularityconstantpart2} \approx 1$ that  solely depends on the shape regularity of 
$\widehat  T_\pm$ (and so on the shape regularity of  $\T$). 
This leads to
 \[
I(E)\le\constant[\ref{standardshaperegularityconstantpart2}]^2  \| h_\T^{-1}  \nabla_\pw(w_2-I_\M w_2) ||_{L^2(\widehat  \omega(E))}  
\| D^2_\pw (1-J)v_\M \|_{L^2(\widehat \omega(E))}
\]
for any interior edge $E\in\E(\Omega)$ with the reduced edge-patch $\widehat \omega(E)$. The same estimate follows for
a boundary edge $E\in\E(\partial\Omega)$ (the proof omits   $T_-$, $\widehat  T_-$, and some factor $1/2$ above). 
Since the reduced edge-patches $(\widehat \omega(E):E\in\E)$
have no overlap, the sum of all the above estimates of $I(E)$  in \eqref{eq:internewcc2}  and Cauchy  inequalities prove
\begin{align} \label{eq:internewcc2b}
(a_\pw+ b_h)(v_\M, w_2-I_\M w_2)
 \le \constant[\ref{standardshaperegularityconstantpart2}]^2  | h_\T^{-1}  (w_2-I_\M w_2) |_{H^1(\Omega)}   \trinl  (1-J)v_\M \trinr_\pw.
\end{align}
Recall from Theorem \ref{int_err}.a. (with $\|(1-\Pi_0) D^2_\pw w_2\|=0$ for $w_2 \in P_2(\T)$) that
 \begin{align*} 
  |h_\T^{-1}(w_2-I_\M w_2)|_{H^1(\T)} 
  \lesssim j_h (w_2-w)   
  \lesssim   \| w_2 -  w \|_h
  \end{align*}
for all $w\in V$. Lemma~\ref{lem:MorleyCompanion}.d shows 
$ \trinl  (1-J)v_\M \trinr_\pw\le \Lambda_\jc \trinl  v_\M - v \trinr_\pw $
for any $v\in V$. The combination of this with \eqref{eq:internewcc2b} concludes the proof of  \eqref{eq:dis_consistency}. 
\end{proof}

\begin{proof}[Proof of \eqref{eqnanihilationpropertyonbh}] This follows from \eqref{jterm}.
\end{proof}

\begin{proof}[Proof of \eqref{eq:b_bound}]
{ This follows from the boundedness of $b_h(\bullet, \bullet)$ (see \cite{FengKarakashian:2007:CahnHilliard,MozoSuli07}). }
\end{proof}
 
\begin{proof}[Proof of \eqref{eq:c_bound}] 
The jump contributions in \eqref{eq:ccdefineJsigma1sigma2} vanish for arguments in $V$,  $c_\dg(v_2,w_2)= c_\dg(v-v_2,w-w_2)$ on
the left-hand side of  \eqref{eq:c_bound}. 
Recall that $c_\dg(\bullet,\bullet)$ is a semi-norm scalar product and the Cauchy inequality with the induced semi-norm 
$|\bullet|_{c_\dg}:=c_\dg(\bullet,\bullet)^{1/2}\le \|\bullet \|_\dg$ is a part of the discrete norm $\|\bullet \|_\dg$.
This leads to  \eqref{eq:c_bound} with $\Lambda_{\rm c}=1$.
 \end{proof}
   
\begin{proof}[Proof of \eqref{eq:dual_dis_consistency} for $\Theta=1$]
This follows from Remark~\ref{exampleoneq:dual_dis_consistency}. 
 \end{proof}

\section{Modified $C^0$IP method} \label{sec:C0IP}
 For the right-hand side $F \in H^{-2}(\Omega)$,   the modified  $C^0$IP method is based on the continuous Lagrange $P_2$ finite element space $V_h:=S^2_0(\T):= P_2(\T) \cap H^1_0(\Omega)$ and penalty terms along edges. The scheme is a modification of the dGFEM in Section \ref {sec:DGFEM} but with trial and test functions restricted to $S^2_0(\T):=P_2(\T) \cap H^1_0(\Omega)$. The norm $\ipnorm{\bullet}$ is $\|\bullet \|_\dg$ with restriction to $S^2_0(\T)$ and excludes one of the penalty parameters of the modified dGFEM.

\medskip \noindent 
Given $\sigma_\ip>0$,
the bilinear forms  \cite{BS05,CCGMNN18}  for $v_\ip,w_\ip\in S^2_0(\T) $ are defined  by
\begin{align} 
A_\ip(v_\ip,w_\ip) :=
   &    a_\NC(v_\ip,w_\ip)  +b_h (v_\ip, w_\ip)+c_\ip(v_\ip,w_\ip),  \text{ where }  \label{eq:AhinC0IP} \\
    c_\ip(v_\ip,w_\ip)& :=  
    \sum_{E \in \E} \frac{\sigma_\ip}{h_E} \int_E \jump{\frac{\partial v_\ip}{\partial \nu_E}} \jump{\frac{\partial w_\ip}{\partial \nu_E}} \ds, \label{eq:chinC0IP} 
\end{align}
  and $ b_h(\bullet, \bullet)= b_h(\bullet, \bullet)|_{S^2_0(\T)}$ from \eqref{eq:bhindG}.

The modified $C^0$IP method is of the form \eqref{eq:discrete2}  
and seeks $u_\ip \in S^2_0(\T)$ such that 
\begin{align}\label{eq:C0IP}
 A_\ip(u_\ip,v_\ip) = F(JI_\M v_\ip)
 \quad\text{for all } v_\ip \in  S^2_0(\T).
\end{align}
 For all $v+ v_\ip \in V+ S_0^2(\T)$, the discrete norm reads
$ \displaystyle
  \ipnorm{v+v_{\ip}}
:= (\trinl{v+v_\ip}\trinr_\pw^2 + c_\ip(v_\ip, v_\ip))^{1/2}$ and  
$j_h(v_\ip) =  ( \sum_{E \in \E}  ( \fint_E  \jump{{\partial v_{\ip}}/{\partial \nu_E}} \ds )^2 )^{1/2}$.  Theorem~\ref{lemmaonhnormisanorm}  shows  $\|v_{\ip} \|_h \approx \|v_\ip\|_\ip$.
 The coercivity  $ \ipnorm{\cdot}^2 \lesssim A_\ip(\cdot,\cdot)$ on $S^2_0(\T)$ 
holds provided $\sigma_\ip$ is sufficiently large \cite{BS05,CCGMNN18}. The boundedness $A_{\ip}(v_{\ip},w_{\ip}) \lesssim  \|v_{\ip}\|_\ip \|w_{\ip} \|_\ip$ holds for all $v_\ip, w_\ip$ on $S^2_0(\T)$ and    \eqref{eq:C0IP} has a unique solution $u_\ip \in S^2_0(\T)$. 

\begin{thm}[error estimates]\label{thm:C0IP}  The  solution $u \in V$ to \eqref{eq:weakabstract}  and the solution  $u_\ip \in S^2_0(\T)$ to \eqref{eq:C0IP} 
satisfy (a) $ \|u- u_\ip\|_h \lesssim  \trinl u- I_\M u \trinr_\pw$ and
(b) if $\Theta=1$ and  $F\in H^{-s}(\Omega) $ for  $2-\sigma \le s \le 2$, then $\|u- J I_\M u_\ip\|_{H^s(\Omega)} +\|u- u_\ip\|_{H^s(\T)}   \lesssim h_{\max}^{2-s}    \| u- u_\ip \|_h$.
\end{thm}

\begin{rem}
\noindent
A $C^0$IP discrete scheme  is analysed in \cite{BS05}  for a general $F \in H^{\sigma-2}(\Omega)$.  The consistency of the scheme allows a best approximation \cite[Lemma 8]{BS05} (since $V_h$ subset $H^{2-\sigma}(\Omega)$ in the pure Dirichlet problem for $C^0$IP).  
For $F \in H^{-2}(\Omega)$, a modifed scheme and error estimates for the post-processed solution are derived in \cite[(4.17) and Theorem 4]{BS05}. 

 \end{rem}

\noindent{\it Overview of the proof of Theorem~\ref{thm:C0IP}}.
The proof follows the lines of that of Theorem \ref{thm:dg} and partly from the analysis provided there. 
The bilinear forms 
in the  $C^0$IP are exactly the respective  bilinear forms of the  dGFEM when restricted to the subspace $S^2_0(\T)+M(\T)$. 
With the single exception of \eqref{eq:mod_interpolation_estimate}, all the estimates in 
\eqref{eq:Inc_bound}-\eqref{eq:c_bound}  and \eqref{eq:dual_dis_consistency} for $\Theta=1$ 
follow for $V_h=S^2_0(\T)$ in the  $C^0$IP  from the respective properties verified  in Section~\ref{sec:DGFEM} for $V_h=P_2(\T)$ in the  dGFEM.
The remaining detail is the analysis of the 
 operator $I_h \equiv I_\C: \M(\T) \rightarrow S^2_0(\T)$ (denoted by $I_2^*$ in  \cite[Lemma 3.2]{CarstensenGallistlNataraj2015})  
 defined  by averaging the values of a function { $v_\M \in  \M(\T)$ at the midpoint of an interior edge $E$, 
\begin{align} \label{eq:ic}
(I_\C v_\M)(z)=
\begin{cases}
v_\M(z) \quad \text{for all } z \in \V,\\
\mean{v_\M}(z)  \quad \text{for } z= \text{mid}(E), \; E \in \E(\Omega), \\
0 \quad \text{for } z= \text{mid}(E), \; E \in \E(\partial \Omega).
\end{cases}
\end{align}}

\begin{proof}[Proof of \eqref{eq:mod_interpolation_estimate}]   
This is included in \cite[Lemma 3.2f]{CarstensenGallistlNataraj2015} in a slightly different notation.
In the notation of this paper, Lemma 3.2 of \cite{CarstensenGallistlNataraj2015} shows 
\begin{equation} \label{eq:dgccproof1}
\|v_\M- I_\C v_\M\|^2_h  \lesssim \sum_{E \in \E} h_E^{-2} |\jump{v_\M}({\text{mid} (E)})|^2
\end{equation}
for any $v_\M \in \M(\T)$. Lemma 3.3 of \cite{CarstensenGallistlNataraj2015} controls the upper bound of \eqref{eq:dgccproof1} by the a~posteriori terms 
$\sum_{E \in \E} h_E \|\jump{D^2_\pw v_\M} \tau_E\|^2_{L^2(E)}$. The latter is 
efficient, i.e.,   $\lesssim \trinl v_\M - v \trinr_\pw$ for any $v \in H^2_0(\Omega)$. This leads to  \eqref{eq:mod_interpolation_estimate}. 
Theorem~\ref{lemmaccapproximationproperties} allows for an alternative proof that
departs at  \eqref{eq:dgccproof1} with the  quadratic function   $\jump{v_\M}$ along the edge $E\in\E$. 
Since   $\jump{v_\M}$ vanishes at each end point $z\in\V(E)$  (owing to the continuity of the Morley function $v_\M \in \M(\T)$ at the vertices),
the (exact) Simpson's quadrature rule asserts 
\[
 \jump{v_\M}({\text{mid} (E)}) = 3/2 \fint_E  \jump{v_\M}\ds.
\]
 A Cauchy inequality, the continuity of  $J v_\M\in V$,  triangle and  trace inequalities lead to
 \[
h_E^{-2} |\jump{v_\M}({\text{mid} (E)})|^2\le  \frac 94 h_E^{-3}  \| \jump{v_\M-Jv_\M}\|_{L^2(E)}^2 \lesssim 
 \sum_{\ell=0}^1    |h_\T^{\ell -2}(v_\M-Jv_M) |_{H^\ell (\omega(E))}^2 .
 \] 
This estimate and the finite overlap of the edge-patches $(\omega(E):E\in\E)$ lead to an upper bound 
 in  \eqref{eq:dgccproof1} as in Theorem~\ref{lemmaccapproximationproperties}.c  and so to $\|v_\M- I_\C v_\M\|_h\lesssim \| v_\M-v \|_h=  \trinl v_\M-v   \trinr_\pw$. 
\end{proof}

{
 \begin{rem}[$Q$ is not injective]
  An illustration shall be given for a triangulation $\T=\{T_1,T_2\}$ of a convex quadrilateral  $\bar{\Omega} =T_1 \cup T_2 = \text{ conv} \{P_1,P_2,P_3,P_4\}.$ There is exactly one basis function $b_E \in S^2_0(\T)= V_h$ defined on $T_1= \text{ conv }\{P_1,P_2,P_4\}$ and on 
$T_2= \text{ conv }\{P_2,P_3,P_4\}$ by $b_E= 4 \varphi_2 \varphi_4$ for the nodal basis functions $\varphi_1,\varphi_2,\varphi_3,\varphi_4$. Given $b_E$ for the edge $E= \text{ conv } \{P_2,P_4\}= \partial T_1 \cap \partial T_2$, the normal derivative of $b_E$ along $E$ 
on $T_1$ reads in its integral 
\begin{align*}
 \int_E \partial b_E|_{T_1}/\partial \nu_E  \ds & = \nu_E \cdot \int_E \nabla b_E|_{T_1} \: \ds = 4 \nu_E \cdot \int_E(\varphi_2 \nabla \varphi_4|_{T_1} +
\varphi_4 \nabla \varphi_1|_{T_1}) \ds\\
& = 2|E| \nu_E \cdot (\nabla \varphi_2|_{T_1} + \nabla \varphi_4|_{T_1}) = - 2 |E| \nu_E \cdot \nabla \varphi_1|_{T_1}
\end{align*}
with  $\nabla(\varphi_1+\varphi_2+\varphi_4)=0$ in $T_1$ in the last step.  Elementary geometry shows  $\nabla \varphi_1|_{T_1}=\nu_E/\rho_{E,1}$ for the height $\rho_{E,1}= \frac{2|T_1|}{|E|}$ of $E$ in $T_1$ and $\nu_E$ pointing from $T_1$ into $T_2$. Consequently,  $\displaystyle  \int_E \partial b_E|_{T_1}/\partial \nu_E \ds= \frac{|E|^2}{|T_1|}. $ The analogous calculation for $T_2$ leads to $\displaystyle  \int_E \partial b_E|_{T_2}/\partial \nu_E \ds= -\frac{|E|^2}{|T_2|}$ with a change of sign because  $\nabla \varphi_3|_{T_2} = \nu_E/ \rho_{E,2}$. The definition of $I_\M b_E$ takes the average of the two integral means 
\[
\frac{1}{2} \left(\fint_E \partial b_E|_{T_1}/\partial \nu_E + \fint_E \partial b_E|_{T_2}/\partial \nu_E\right) \ds = \frac{|E|}{2}(|T_1|^{-1}-|T_2|^{-1})
\]
as the value for $\displaystyle \fint_E \partial I_\M b_E/ \partial \nu_E \ds$. Since this is the only degree of freedom in $\M(\T)$ for the triangulation $\T=\{T_1, T_2\}$, it follows that $I: S_0^2(\T) \rightarrow \M(\T)$ is injective if and only of $|T_1|=|T_2|$. (This condition is independent of shape-regularity of $\T$ and thus more involved.)
 \end{rem}
}
\section{Comparison}\label{sec:comparisonandwopsip}
The paper \cite{CarstensenGallistlNataraj2015} has established equivalence of discrete solutions to Morley FEM, $C^0$IP and dGFEM up to oscillations for $F \in L^2(\Omega)$ and for the original schemes with $F_h\equiv F$. The subsequent theorem establishes the three modified schemes with $F_h =F \circ J$ without extra oscillation terms. Throughout this section, the norm $\|\cdot\|_h$ is defined in \eqref{common:norm}-\eqref{eq:jh}.

\begin{thm}\label{thm:MainResult}
The discrete solutions
$u_\M$, $u_\ip$ and $u_\dg$ of the 
Morley FEM, $C^0$IP and dGFEM satisfy
\begin{align*}
\|u-u_\M\|_h
\approx
\|u-u_\dg\|_h
\approx
\|u-u_\ip\|_h
\approx \|(1-\Pi_0) D^2 u\|_{L^2(\Omega)}.
\end{align*}
The equivalence constants $\approx$ depend on shape regularity and on the stabilisation parameters $\sigma_\dg, \sigma_\ip \approx 1$.
\end{thm}

\begin{rem}[discrete dG norm equivalence, Theorem 4.1, \cite{CarstensenGallistlNataraj2015}]  \label{thm:NormEquiv}
 The norm $\hnorm{\cdot}$  satisfies
\begin{align*}
 \hnorm{\bullet} &= \trinl \bullet \trinr_\pw  \; \text{on } V+\M(\T),  \\  
\hnorm{\bullet} &\approx \dgnorm{\bullet}    \; \text{\;on } V+ P_2(\T),  \\
\hnorm{\bullet} &\approx \ipnorm{\bullet}   \; \text{\;\;on } V+ S^2_0(\T).
\end{align*}
\end{rem}

\begin{proof}
Lemma \ref{interpolationerrorestimatesI}.a.,  the Pythogoras identity, and  
\cite[Theorem 2.2]{NNCC2020}
 show 
$$ \|(1- \Pi_0) D^2 u \|_{L^2(\Omega)} = \trinl u- I_\M u \trinr_\pw \le \trinl u - u_\M \trinr_\pw \lesssim \trinl u- I_\M u \trinr_\pw.$$
The $L^2$ best-approximation property of  $\Pi_0 u$, (\ref{eq:comparabilityofnorms}.a),  Theorem \ref{thm:dg}, and Lemma \ref{interpolationerrorestimatesI}.a. lead to
\begin{align*}
 \|(1- \Pi_0) D^2 u \|_{L^2(\Omega)}  &= \min_{v_h \in P_2(\T)} \|u -v_h\|_h \le \|u-u_\dg\|_h  \\
& \lesssim   \trinl u- I_\M u \trinr_\pw =\|(1- \Pi_0) D^2 u \|_{L^2(\Omega)}.
\end{align*}
Theorem \ref{thm:C0IP} leads to similar results  for $\|u-u_\ip\|_h$. A combination of the above displayed inequalities concludes the proof.
\end{proof}

\section{Modified WOPSIP Method} \label{sec:wopsip}
The weakly over-penalized symmetric interior penalty
(WOPSIP) scheme \cite{BrenGudiSung10} is a penalty method  
with the stabilisation term 
\begin{align}\label{eq:chinWOPSIP}
c_\w(v_\pw,w_\pw)&:= \nonumber
 \sum_{E \in \E} \sum_{z \in {\mathcal V} (E)} h_E^{-4} ([v_\pw]_E(z)) ([w_\pw]_E(z)) \\
 &+ \sum_{E \in \E}  h_E^{-2} (  \fint_E  \jump{{\partial v_\pw}/{\partial \nu_E}} \ds )( \fint_E  \jump{{\partial w_\pw}/{\partial \nu_E}} \ds)
\end{align}
for piecewise smooth functions $v_\pw,w_\pw\in H^2(\T)$. This semi-norm scalar product  $c_{\w}(\bullet,\bullet)$ is an analog to that one behind the jump $j_h$ from \eqref{eq:jh}
with different powers of the mesh-size.  It follows as in Theorem~\ref{lemmaonhnormisanorm} that 
\begin{align}\label{eq:ap}
A_{\w}(v_\pw,w_\pw):=a_\pw(v_\pw,w_\pw)+c_\w(v_\pw,w_\pw)\quad\text{for all }v_\pw,w_\pw\in H^2(\T)
\end{align}
defines a scalar product and so  $\|\bullet\|_\w:= A_{\w}(\bullet,\bullet)^{ 1/2}$  is a norm in $H^2(\T)$. 
Consequently, there exists a unique  solution $ u_\w\in V_h:=P_2(\T)$ to 
\begin{align}
A_{\w}(u_\w,v_2)= F(JI_\M v_2) \quad \text{ for all } v_2 \in P_2(\T).
\label{eq:wopsip} 
\end{align}
The  increased  condition number in the  over-penalization of  the jumps by the negative powers
of the mesh-size in \eqref{eq:chinWOPSIP} can be compensated by some preconditioner \cite[p 218f]{BrenGudiSung10} and the entire 
WOPSIP  linear algebra with  \eqref{eq:wopsip} becomes intrinsically parallel.
 
\begin{thm}[error estimate]\label{thmwopsip} Any $F\in H^{-s}(\Omega) $ with   $2-\sigma \le s \le 2$, the 
 solution $u \in V$ to \eqref{eq:weakabstract} and the solution $u_\w \in P_2(\T)$ to \eqref{eq:wopsip} satisfy 
\begin{align*}
 \trinl u-u_\w  \trinr_\pw^2 + c_\w(u_\w,u_\w) & \le (1+ \Lambda_\w^2)\trinl  u- I_\M u \trinr_\pw^2 +\Lambda_\w^2 \trinl  h_\T I_\M u \trinr_\pw^2; \\
  \|u - J I_\nc u_\w \|_{H^{s}(\Omega)} + \|u -  u_\w \|_{{H^s}(\T)} &\le  \constant[\ref{ccwopsip}](s)  h_{\max}^{2-s} \|u-u_\w\|_\w.
\end{align*}
The constant  $ \Lambda_\w$  exclusively depends on the shape regularity of $\T$, while $\constant{}\label{ccwopsip}(s)$ depends on the shape regularity of $\T$ and on  $s$.
\end{thm}

The  subsequent lemma specifies the constant $\Lambda_\w$ in the best-approximation estimate.

\begin{lem} \label{ellipticityofcp} 
There exists some positive  $\Lambda_\w<\infty$, that  exclusively depends
on the shape regularity of $\T$, such that 
$  \trinl h_\T^{-1}(v_2- I_\M v_2) \trinr_\pw^2 +  \trinl (1-J)I_\M v_2 \trinr_\pw^2 \le \Lambda_\w^2  \| v_2-v \|^2_\w $
holds for all $v_2 \in P_2(\T)$ and all $v\in V$. 
\end{lem}

\begin{proof}[Proof of Lemma~\ref{ellipticityofcp}]
The analysis of $  \trinl h_\T^{-1}(v_2- I_\M v_2) \trinr_\pw$  returns to the proof of Theorem~\ref{int_err} 
that  eventually provides \eqref{eqn:cc:morleyinterpolationerrorestimateproof3}  for one fixed triangle $T\in\T$ with
its neighourhood $\Omega(T)$ for any $v_2\in P_2(\T)$. The substitution of $v_2$ by $h_T^{-1} v_2$ (with a fixed scaling factor $h_T$) in 
\eqref{eqn:cc:morleyinterpolationerrorestimateproof3} 
after a standard inverse estimate  shows, for all  $v_2\in P_2(\T)$, that
\[
h_T^{-2} | v_2- I_\M v_2|^2_{H^2(T)}\lesssim h_T^{-6} \| v_2- I_\M v_2\|^2_{L^2(T)}\lesssim j_h(h_T^{ -1} v_2, T)^2.
\]
The shape regularity of $\T$ implies  that all edge-sizes in the sub-triangulation $\T(\Omega(T))$ that covers the neighbouhood
$\Omega(T)$ (of $T$ and one layer of triangles around $T$) are equivalent to $h_T$. Hence
$ j_h(h_T^{ -1} v_2, T)$ is equivalent to the respective contributions in $c_\w(v_2,v_2)$:
\begin{align*}
 j_h(h_T^{ -1} v_2, T)^2 \lesssim 
\sum_{z \in {\mathcal V} (T)}  \sum_{F \in \E(z)}  {h_F^{-4}} ([v_2]_F(z))^2+ \sum_{E \in \E(T)}  h_E^{-2} (  \fint_E  \jump{{\partial v_\pw}/{\partial \nu_E}} \ds )^2.
\end{align*}
The combination of this estimate with the previous one and the sum over all those estimates lead to
\begin{equation}\label{cceqnewwopsipcc3}
 \trinl h_\T^{-1}(v_2- I_\M v_2) \trinr_\pw^2 \lesssim c_\w(v_2,v_2)
\end{equation}
owing to the finite overlap of the family $(\Omega(T):T\in\T)$. The second term  $\trinl (1-J)I_\M v_2 \trinr_\pw $ is controlled
with  Lemma~\ref{lem:MorleyCompanion}.d and a triangle inequality in 
\[
\Lambda_\jc^{-1} \trinl (1-J)I_\M v_2 \trinr_\pw\le    \trinl  v-I_\M v_2 \trinr_\pw \le   \trinl  v-v_2 \trinr_\pw+   \trinl  v_2 -I_\M v_2 \trinr_\pw.
\]
Since the the last term  $\trinl  v_2 -I_\M v_2 \trinr_\pw\le  h_{\max}\trinl h_\T^{-1}(v_2- I_\M v_2) \trinr_\pw $ is bounded in \eqref{cceqnewwopsipcc3},
the summary of the aforementioned estimates concludes the proof of the lemma.
\end{proof}

\begin{proof}[Proof of energy norm estimate in Theorem~\ref{thmwopsip}] 
The equations \eqref{eq:weakabstract} and \eqref{eq:wopsip} show the key identity 
\[
 a (u, JI_\M e_\w)= F(JI_\M e_\w) =a_\pw(u_\w, e_\w)+ c_\w(u_\w,e_\w)\quad\text{for }e_\w:= I_\M u - u_\w \in P_2(\T).
\]
Remark~\ref{remark3.1} applies verbatim and provides $c_\w(I_\M u,e_\w)=0$. This, 
the key identity, and the definition of the norm  $\|\bullet\|_\w:=A_{\w}(\bullet,\bullet)^{1/2}$ lead to
\begin{align} \label{eqn:cc:proof2onwopsipcc1}
\|e_\w\|^2_\w =a_\pw(I_\M u, e_\w)-a(u, JI_\M e_\w)=a_\pw(u, e_\w-JI_\M e_\w) 
\end{align}
with \eqref{eq:best_approx} in the last step. Set  $e_\M:= I_\M e_\w \in \M(\T)$ and split $e_\w-JI_\M e_\w=(e_\w-e_\M)+(1-J)e_\M$. 
The  last term in \eqref{eqn:cc:proof2onwopsipcc1} is equal to 
\begin{align*}
&a_\pw(u, e_\w-e_\M)+a_\pw(u, e_\M- Je_\M) \\
&= 
a_\pw(I_\M u, e_\w-e_\M)+a_\pw(u-I_\M u, e_\M- Je_\M) 
\\
& \le  
 \trinl  h_\T I_\M u \trinr_\pw  \trinl  h_\T^{-1}(e_\w-e_\M)\trinr_\pw 
 + \trinl  u-I_\M u   \trinr_\pw \trinl  e_\M- Je_\M   \trinr_\pw
 \\
 &\le  \Lambda_\w  (\trinl  h_\T I_\M u \trinr_\pw^2+ \trinl  u-I_\M u   \trinr_\pw^2)^{1/2} \| e_\w  \|_\w
\end{align*} 
with \eqref{eq:best_approx} twice in the first equality,  (weighted) Cauchy inequalities for the inequality in the third line, and
Lemma~\ref{ellipticityofcp} (with $v=0$) in the end. 
The combination with  \eqref{eqn:cc:proof2onwopsipcc1} 
proves 
\[
\|e_\w\|^2_\w\le   \Lambda_\w^2(\trinl  h_\T I_\M u \trinr_\pw^2+ \trinl  u-I_\M u   \trinr_\pw^2).
\]
This and the Pythagoras theorem \eqref{eq:Pythogoras} conclude the proof.
\end{proof}

\begin{proof}[Proof of error estimates in weaker (piecewise) Sobolev  norms in Theorem~\ref{thmwopsip}] 
The error analysis in weaker norms adapts the notation of the beginning of Subsection~\ref{sec:dual} on  $v:= J I_\nc u_\w \in V$
and $\zeta:= J I_\nc z \in V$ for the dual solution $z\in V$ with $\|z \|_{ H^{4-s}(\Omega)}\le C_{\rm reg}$ and 
\[
\|u -v\|_{H^{s}(\Omega)}=a(u-v,z)=a(u,z-\zeta)+a_\pw (u_\w-v,z)
\]
with the key identity $a(u,\zeta) =F(\zeta)=A_\w(u_\w,I_\M z)=a_\pw(u_\w,I_\M z)=a_\pw(u_\w, z)$ (from Remark~\ref{remark3.1}
and   \eqref{eq:best_approx}) in the last step. 
Since  $a_\pw(I_\M u,z-\zeta) =0$ (from  \eqref{eq:best_approx} with $I_\M z=I_\M \zeta$ from \eqref{eq:right_inverse}), 
the first term 
\[
a(u,z-\zeta)=a_\pw(u-I_\M u,z-\zeta)\le (1+\Lambda_{\rm J}){ (1+\Lambda_{\rm M})}  \trinl  u- I_\M u \trinr_\pw  \trinl  z-I_\M z   \trinr_\pw
\]
is controlled by \eqref{eq:z-zeta} (with $z_h=I_\M z$, $I_h=\text{\rm id}$). 
The analysis of the second term $a_\pw (u_\w-v,\zeta)$ follows the corresponding lines
of the proof of the best-approximation in Theorem~\ref{thmwopsip}  with
$u_\M :=I_\nc u_\w$, $u_\w-v= u_\w- u_\M+(1-J) u_\M $. This shows
\begin{align*}
&a_\pw (u_\w-v,z)=a_\pw(u_\w-u_\M,z)+a_\pw(u_\M- Ju_\M,z) \\
&\quad  =  a_\pw(u_\w-u_\M, I_\M z)+a_\pw(u_\M- Ju_\M, z-I_\M z) 
\\
&\quad  \le  
 \trinl  h_\T I_\M z \trinr_\pw  \trinl  h_\T^{-1}(u_\w-u_\M)\trinr_\pw 
 + \trinl  z-I_\M z   \trinr_\pw \trinl  u_\M- Ju_\M   \trinr_\pw
 \\
 &\quad \le  \Lambda_\w  (\trinl  h_\T I_\M z\trinr_\pw^2+ \trinl  z-I_\M z   \trinr_\pw^2)^{1/2} \| u-u_\w  \|_\w
\end{align*} 
from Lemma~\ref{ellipticityofcp} with $v=u$. The final argument is the regularity of $z$ and the approximation estimates 
\(
\trinl z-I_\nc z\trinr_\pw  \le C_{\rm int}(s)C_{\rm reg }(s) h_{\max}^{2-s}
\)
from Subsection~\ref{subsectProofofTheoremthm:aux1}.
The combination of the above arguments shows 
\[
\|u -J I_\nc u_\w\|_{H^{s}(\Omega)}\lesssim h_{\max}^{2-s} \| u-u_\w  \|_\w.
\]
Theorem~\ref{lemmaccapproximationproperties}.d applies to  $u_\w\in P_2(\T)$  and  the remaining arguments 
follow the last lines in the proof Theorem~\ref{thm:aux1}.b
with a triangle inequality in  $H^s(\T)$ in the end.
\end{proof}

\section*{Acknowledgements}
The research of the first author has been supported by the Deutsche Forschungsgemeinschaft in the Priority Program 1748 under the project "foundation and application of generalized mixed FEM towards nonlinear problems
in solid mechanics" (CA 151/22-2).  The finalization of this paper has been supported by   SPARC project 
(id 235) entitled {\it the mathematics and computation of plates} and SERB POWER Fellowship SPF/2020/000019.

\bibliographystyle{amsplain}
\bibliography{ref_combined}
\section*{{Appendix}}
{\it Proof of Lemma~\ref{lem:smoother}.}  For $z \in Y_h \cap Y$,  \eqref{q=id} implies  $\|z-Qz\|_{\widehat Y}=0$,  and hence \eqref{quasioptimalsmoother}  holds. The converse is a consequence of the finite dimension of $Y_h$.   In the context of Peetre-Tartar lemma \cite{tartar2010introduction}, let $\widetilde{Y}:=(1-Q)(Y_h) \subset \widehat{Y}$ denote the range of $1-Q$ and abbreviate $A:=1-\Pi_Y$ and $B:=\Pi_Y$. Then $A \in L(\widetilde{Y}; \widehat Y)$ is injective because $A \widetilde{y}=0$ means $\widetilde{y} \in \widetilde{Y} \cap Y$ for some $y_h \in Y_h$ with $\widetilde{y}=(1-Q)y_h$, whence $y_h \in Y_h \cap Y$ and $\widetilde{y}=0$ from~\eqref{q=id}. Notice that $B \in L(\widetilde{Y};Y)$ is compact (for $\widetilde{Y}$ has finite dimension).
Since $\widetilde{y} = A \widetilde{y} + B \widetilde{y}$ implies $\|\widetilde{y}\|_{\widehat{Y}} \le \|A \widetilde{y}\|_{\widehat{Y}}
+\|B \widetilde{y}\|_{\widehat{Y}} $ for all $ \widetilde{y} \in {\widehat{Y}} $, the Peetre-Tartar lemma proves 
$\gamma \|\widetilde{y}\|_{\widehat{Y}} \le \|A \widetilde{y}\|_{\widehat{Y}}$ for all $ \widetilde{y} \in {\widehat{Y}} $ and some constant $\gamma>0$. This implies for all $y_h \in Y_h $ that 
\[\gamma \|y_h - Qy_h\|_{\widehat{Y}} \le \|A(y_h - Qy_h)\|_{\widehat{Y}} = \|y_h- \Pi_Y y_h\|_{\widehat{Y}} \le \|y_h-y\|_{\widehat{Y}} \text{ for all } y \in Y.
\]
This is \eqref{quasioptimalsmoother} with $\Lambda_{\rm Q}:=1/\gamma>0$. The point in those compactness arguments is that we do know that $\gamma =1/\Lambda_{\rm Q}>0$ with \eqref{quasioptimalsmoother} exists, but we do not know how it depends, e.g., on ${\rm dim } \: Y_h$.
\qed

\medskip \noindent
{\it Proof of Lemma~\ref{lemma2.2}.}  
 It is obvious that {\bf (QO)} implies \eqref{m=id}. The converse follows from a compactness argument from ${\rm dim} \: X_h < \infty$. The subspace $\widetilde{X}:=(1-M)X \subset \widehat{X}$ is complete because $1-M:(X_h \cap X)^{\perp} \cap X \rightarrow \widetilde{X}$ is linear,  bounded, and bijective for the (complete) orthogonal complement $(X_h \cap X)^{\perp} \cap X$ of $X_h \cap X$ in $X$. Then $A:= 1- \Pi_{X_h} \in L(\widetilde{X}; \widehat{X})$ is injective (as $A \widetilde{x}=0$ implies $\widetilde{x}=(1-M)x \in X_h$, whence $x \in X_h \cap X$ and $\widetilde{x}=0$ in $\widetilde{X}$ from~\eqref{m=id}). Since $B:= \Pi_{X_h} \in L(\widetilde{X}; \widehat{X}_h) $ is compact and $\widetilde{x} = A\widetilde{x}  +B\widetilde{x} $ implies 
\[
\|\widetilde{x} \|_{\widehat{X}} \le \|A\widetilde{x} \|_{\widehat{X}} +\|B \widetilde{x} \|_{\widehat{X}}  \text{ for all } \widetilde{x}\in {\widehat{X}},
\]
the Peetre-Tartar lemma \cite{tartar2010introduction} leads to some $\gamma>0$ with 
\[\gamma \|(1-M)x\|_{\widehat{X}}  \le \|A(1-M)x\|_{\widehat{X}}  = \|x-\Pi_{X_h} x\|_{\widehat{X}}  \le \|x-x_h\|_{\widehat{X}} 
\]
for all $x \in X$ and $(1-M)x \in {\widetilde{X}} $ and for all $x_h \in X_h$. This is $\text{\bf (QO)}$ with $C_{\rm qo}:=1/\gamma$.
\qed

\medskip \noindent 
{\it Proof of Theorem~\ref{quasioptimal}.}
[Proof of "$ \Longrightarrow$"] Suppose $M$ satisfies {\bf (QO)} with constant $C_{\rm qo}$. Then for all $x_h,y_h$, the definition of $MPx_h$ leads to the identity
\begin{align*}
	a_h(x_h,y_h) - a(Px_h, Qy_h) ={}& a_h(x_h - MPx_h, Qy_h) = \langle Q^\ast A_h (x_h  - MPx_h),y_h \rangle_{Y_h^* \times Y_h} \\ 
	\le{}& \|Q^\ast A_h\| \: \|y_h\|_{Y_h} \|x_h-MPx_h\|_{X_h}. 
\end{align*}
It remains to prove that $\|x_h - MP x_{h}\|_{X_h} \le (1 + C_{\rm qo})\|x_h - Px_{h}\|_{\widehat{X}}$. 
	This follows from a triangle inequality $\|x_h-MPx_h\|_{X_h} \le \|x_h-Px_h\|_{\widehat{X}} +\|Px_h-MPx_h\|_{\widehat{X}}$ and {\bf (QO)} in $\|Px_h-MPx_h\|_{\widehat{X}} \le C_{\rm qo} \|Px_h-x_h\|_{\widehat{X}}.$ In conclusion,  {\bf(H)} holds with 
	$\Lambda_{\rm H}:= \|Q^\ast A_h\| (1+C_{\rm qo})$.
\medskip 

\noindent
[Proof of "$ \Longleftarrow$"]
Let $\Pi_{X_h} \in L(\widehat{X})$ denote the orthogonal projection 
onto the closed subset $X_h$  in the Hilbert space $\widehat{X}$.
Given any $x \in X$, let  
$x_h^{\ast} := \Pi_{X_h}x=\text{arg min}_{\xi_h \in X_h} \|x - \xi_h \|_{\widehat{X}}$ denote the best-approximation of $x$ 
in $X_h$ in the Hilbert space $\widehat{X}$
and set $e_h := x_h^{\ast} - Mx \in X_h$. 
 The inf-sup condition for $a_h(\bullet,\bullet)$ leads to  $y_h \in Y_h$ with norm 
 $\|y_h \|_{Y_h} \le 1$ such that
\begin{align*}
	\alpha_h \| e_h \|_{X_h} ={}& a_{h}(e_h,y_h) 
	= a_h(x_h^\ast,y_h)  -  a_h(Mx,y_h).
	\end{align*}
Recall the definition of $x_h:=Mx=A_h^ {-1}Q^* Ax$ as the discrete solution for
the right-hand side $a(x,Q\bullet)$ to verify
\[
a_h(Mx,y_h)=a_h(x_h,y_h)=a(x,Qy_h).
\]	
This leads to the identity
\begin{align} \label{eqprolog1}
\alpha_h \| e_h \|_{X_h} = a_h(x_h^\ast,y_h)  -  a(x,Qy_h).
\end{align}
Hypothesis {\bf (H)} and $\|y_h \|_{Y_h} \le 1$ lead to the first 
and  \eqref{quasioptimal} to the last inequality in
\begin{align} \label{eqprolog2}
a_h(x_h^\ast,y_h)  - a(Px_h^\ast,Qy_h) \le   
\Lambda_{\rm H}   \|x_{h}^\ast - Px_{h}^\ast \|_{\widehat{X}}  \le  \Lambda_{\rm H}  \Lambda_{\rm P}	\|x - x_{h}^\ast\|_{\widehat{X}}.
\end{align}
A  triangle inequality and  \eqref{quasioptimal}  imply
\[ \|x - Px_{h}^\ast\|_{{X}} \le  \|x - x_{h}^\ast\|_{\widehat{X}} + \|x_h^* - Px_{h}^\ast\|_{\widehat{X}} \le \ (1+\Lambda_{\rm P}) \|x - x_{h}^\ast\|_{\widehat{X}}. \] 
With the operator norm $|| Q^*A||$ in $L(X;Y_h^*)$ and the duality brackets $\langle \bullet,\bullet \rangle_{Y_h^\ast \times Y_h}$ in $Y_h^\ast \times Y_h$, this and $\|y_h \|_{Y_h} \le 1$ show
\begin{align} \label{eqprolog3}
a(Px_h^\ast - x,Qy_h) 
= \langle Q^*A(Px_h^\ast - x), y_h\rangle_{Y_h^* \times Y_h}
 \le \| Q^*A\|   (1+\Lambda_{\rm P})\|x - x_{h}^\ast \|_{\widehat{X}}. 
\end{align}
The combination of \eqref{eqprolog1}-\eqref{eqprolog3} reads
\begin{align*} 
\alpha_h \| e_h \|_{X_h} \le  (\Lambda_{\rm H} \Lambda_{\rm P}+ \|Q^*A\|(1+\Lambda_{\rm P})) 
\|x-x_{h}^\ast \|_{\widehat{X}}.
\end{align*}
Define $C_{\rm qo}:=1+\alpha_h^{-1} (\Lambda_{\rm H} \Lambda_{\rm P}+\|Q^*A\|(1+\Lambda_{\rm P}))$
and rewrite the last estimate as 
 \begin{align} \label{eqprolog4}
  \| Mx-x_h^* \|_{X_h} = \| e_h \|_{X_h} \le (C_{\rm qo}-1)\|x-x_{h}^\ast \|_{\widehat{X}}.
\end{align}
A triangle inequality 
$\|  x-Mx\|_{\widehat{X}} \le \|  x-x_h^*\|_{\widehat{X}}+\| Mx-x_h^* \|_{X_h}$
and \eqref{eqprolog4} prove ${ \text{\bf{(QO)}}}$ because  of
$\|  x-x_h^*\|_{\widehat{X}}=\min_{x_h\in X_h} \|  x-x_h\|_{\widehat{X}}$.
\qed

\medskip \noindent 
{\it Proof of Theorem~\ref{thm:qo2}}.
[Proof of "$\Longrightarrow$"] Given $x_h' \in X_h' $,  let  $N:= \text{ Ker } M \subset X$; 
and let $N^{\perp}$  denote the orthogonal complement of $N$ in the Hilbert space $X$.  
The restriction $M|_{N^{\perp}}: N^{\perp} \rightarrow X_h'$ of $M$ is 
linear, bounded, and bijective  and hence has a 
linear and bounded inverse $S:= (M|_{N^{\perp}})^{-1}: X_h' \rightarrow N^{\perp}$. 
Since $N$ is closed in  $\widehat{X}= X+X_h$, 
the  orthogonal projection $\Pi_N \in L( \widehat{X}) $ onto $N$ is well-defined
and so is its restriction  $\Pi_N |_{X_h'}\in L(X_h';X)$.
Given $S\in L( X_h';X)$ and $\Pi_N |_{X_h'}$, define 
$P':= \Pi_N + S\in L({X_h'}; X).$
Let $x:=P'x_h' = \Pi_N x_h'+Sx_h' \in X$ and observe $MX = M\vert_{N^\perp}(SX_h')=x_h$ and $MP'= {\rm id}$ in $X_h'$.  
Let $\xi:= \Pi_X x_h' \in X$ be the best-approximation of $x_h$ in $X$ with respect to the norm of $\widehat{X}$.  Since $\xi-P'M\xi \in N \perp
x_h'-P'x_h' \in N^{\perp}$,  the Pythogoras theorem in $\widehat{X}$ reads
$$\|\xi-x_h' +P'(x_h'-M\xi)\|^2_{\widehat{X}} = \|x_h'-Px_h'\|^2_{\widehat{X}} + \|\xi- P'M \xi\|^2_{\widehat{X}}.$$
The left-hand side of the above displayed equality is an upper bound of $ \|x_h'-Px_h'\|^2_{\widehat{X}}$ and is smaller than or equal to the square of 
$$ \|\xi-x_h' +P'(x_h'-M\xi)\|_{\widehat{X}} \le   \| (1- \Pi_X) x_h'\|_{\widehat{X}} + \|P'\| \: \|x_h'-M \xi\|_{X_h} $$
with the operator norm $\|P'\|$ of $P' \in L(X_h'; X)$. Consequently, 
$$ \|x_h'-Px_h'\|_{\widehat{X}} \le \| (1- \Pi_X) x_h'\|_{\widehat{X}} + \|P'\| \: \|x_h'-M \xi\|_{X_h}. $$
A triangle inequality and {\bf (QO)} with  $\|\xi- M \xi\|_{\widehat{X}} \le C_{\rm qo} \|\xi- x_h'\|_{\widehat{X}}$ show
$$ \|x_h'-M \xi\|_{X_h} \le \|\xi-x_h'\|_{\widehat{X}} +C_{\rm qo} \|\xi- x_h'\|_{\widehat{X}} =(1+C_{\rm qo} )\|\xi- x_h'\|_{\widehat{X}}.$$
The combination of the previous two displayed estimates with $\| (1- \Pi_X) x_h'\|_{\widehat{X}}  \le \| x_h'- P x_h'\|_{\widehat{X}} $ (from 
$Px_h' \in X$ and the definition of $\Pi_X x_h$) and  $\Lambda_{\rm P'} := 1 + \|P'\|(1 + C_{\rm qo})$ proves
\begin{equation} \label{eqn:pxh-xh}
\|x_h'-P'x_h'\|_{\widehat{X}} \le \Lambda_{\rm P'} \|x_h'-Px_h'\|_{\widehat{X}}.
\end{equation} 
This and $y: = Qy_h \in Y$ lead in ${\widehat{\text{\bf (QO)}}}$ to
\[a(P'x_h'- PMP'x_h', Qy_h) \le  \widehat{C_{\text{\rm qo}}} \|x_h' - P'x_h'\|_{\widehat{X}} \|y_h- Qy_h\|_{\widehat{Y}}. \]
Recall the definition of $x_h=Mx=A_h^ {-1}Q^* Ax$ as the discrete solution to
the right-hand side $a(x,Q\bullet)$ to verify
$a_h(Mx,y_h)=a_h(x_h,y_h)=a(x,Qy_h).$  The combination with the last displayed inequality with $MP'x_h'=x_h'$ leads to
 \[a(P'x_h'- PMP'x_h', Qy_h) = a_h(x_h',y_h)-a(Px_h', Qy_h) \le  \widehat{C_{\rm qo}}
\|x_h'-Px_h'\|_{\widehat{X}} \|y_h - Qy_h\|_{\widehat{Y}}.  \]
This and~\eqref{eqn:pxh-xh} prove ${\widehat{ \text{\bf (H)}}}$ with $\widehat{\Lambda_{\rm H}} = \widehat{C_{\text{\rm qo}}} 
\Lambda_{\rm P'} = \widehat{C_{\rm qo}}(1 + \|P'\|(1 + C_{\rm qo}))$. 

\noindent [Proof of "$\Longleftarrow$"] Given any $x \in X$ and $y_h \in Y_h$, let $x_h':=Mx \in X_h$ with $a_{h}(x_h',y_h) = a(x,Qy_h)$. This shows in ${\widehat{ \text{\bf (H)}}}$ that 
\[ a(x-PMx, Qy_h) \le  \Lambda_2'  \|x_h'-Px_h'\|_{\widehat{X}} \|y_h-Qy_h\|_{\widehat{Y}}.  \]
This and the operator norm $\|a\|$ of $a(\bullet,\bullet) $ show
\begin{align*}
 a(x-PMx,y ) &= a(x-PMx, y-Qy_h) + a(x-PMx, Qy_h) \\
&  \le \|a\| \|x-PMx\|_{X}  \|y- Qy_h\|_{Y} +
 \Lambda_2'  \|Mx-PMx\|_{\widehat{X}} \|y_h-Qy_h\|_{\widehat{Y}}.
\end{align*}
This and the elementary inequalities 
\[ \|x-PMx\|_X \le \|x-Mx\|_{\widehat{X}} +\|Mx-PMx\|_{\widehat{X}},  \; \|Mx-PMx\|_{\widehat{X}} \le \Lambda_{\rm P} 
\|x-Mx\|_{\widehat{X}},
\]
and 
\[ \|y-Qy_h\|_Y \le \|y-y_h\|_{\widehat{Y}} +\|y_h - Qy_h\|_{\widehat{Y}}, \; \|y_h -Qy_h\|_{\widehat{Y}} \le \Lambda_{\rm Q} \|y-y_h\|_{\widehat{Y}}
\]
conclude the proof of ${\widehat{\text{\bf (QO)}}} $ with 
$ \widehat{C_{\text{\rm qo}}} := \|a\| (1+\Lambda_{\rm P} ) (1+ \Lambda_{\rm Q}) + \Lambda_2' \Lambda_{\rm P} \Lambda_{\rm Q}.$
\qed

\end{document}